\documentclass[11pt]{amsart}
\usepackage{graphicx, overpic, soul}
\usepackage{amssymb}
\usepackage{xcolor}
\usepackage{hyperref}
\usepackage{mathtools}
\usepackage{comment}
\usepackage{enumerate}
\usepackage{bbm}
\usepackage{tikz-cd}

\date{}

\newtheorem{theorem}{Theorem}[section]
\newtheorem{lemma}[theorem]{Lemma}
\newtheorem{corollary}[theorem]{Corollary}
\newtheorem{proposition}[theorem]{Proposition}
\newtheorem{conjecture}[theorem]{Conjecture}
\theoremstyle{definition}
\newtheorem{definition}[theorem]{Definition}
\newtheorem{question}[theorem]{Question}
\newtheorem{ex}[theorem]{Example}
\newtheorem{remark}[theorem]{Remark}

\newcommand{\m}{\mathfrak{m}}
\renewcommand{\S}{\mathbb{S}}

\renewcommand{\L}{\mathbb{L}}
\newcommand{\F}{\mathbb{F}}
\newcommand{\D}{\mathbb{D}}
\newcommand{\N}{\mathbb{N}}
\newcommand{\Z}{\mathbb{Z}}
\newcommand{\Q}{\mathbb{Q}}
\newcommand{\R}{\mathbb{R}}
\newcommand{\C}{\mathbb{C}}
\newcommand{\T}{\mathbb{T}}
\newcommand{\SA}{\mathcal{A}}

\newcommand{\SM}{\mathcal{M}}

\newcommand{\La}{\Lambda}
\newcommand{\la}{\lambda}

\renewcommand{\d}{\Delta}

\newcommand{\Br}{\operatorname{Br}}

\newcommand{\Aut}{\operatorname{Aut}}

\newcommand{\Symp}{\operatorname{Symp}}

\newcommand{\GL}{\operatorname{GL}}
\newcommand{\PGL}{\operatorname{PGL}}
\newcommand{\PSL}{\operatorname{PSL}}

\newcommand{\Sh}{\operatorname{Sh}}

\renewcommand{\st}{\text{st}}
\newcommand{\std}{\text{st}}

\newcommand{\Cone}{\operatorname{Cone}}
\newcommand{\Hom}{\operatorname{Hom}}
\newcommand{\uf}{\operatorname{uf}}
\newcommand{\FM}{\mathfrak{M}}
\newcommand{\SX}{\mathcal{X}}
\newcommand{\w}{\mathfrak{w}}
\newcommand{\term}{simultaneously simplifiable }

\begin{document}
	\title{Legendrian doubles, twist spuns, and clusters}
	
	\author{James Hughes and Agniva Roy}

	\begin{abstract}
	
Let $\lambda$ be a Legendrian link in standard contact $\mathbb{R}^3$, such that $L_1$, $L_2$ are two exact fillings of $\la$ and $\varphi$ is a Legendrian loop of $\la$. We study fillability and isotopy characterizations of Legendrian surfaces in standard contact $\mathbb{R}^5$ built from the above data by doubling or twist spinning; denoting them $\La(L_1,L_2)$ or $\Sigma_\varphi(\la)$ respectively. In the case of doubles $\La(L_1,L_2)$, if the sheaf moduli $\mathcal{M}_1(\la)$ admits a cluster structure, we introduce the notion of mutation distance and study its relationship with the isotopy class of the Legendrian surface. For twist spuns $\Sigma_\varphi(\la)$, when $\mathcal{M}_1(\la)$ admits a globally foldable cluster structure, we use the existence of a $\varphi$-symmetric filling of the Legendrian link to build a cluster structure on the sheaf moduli of the twist spun by folding. We then use that to motivate, and provide evidence for, conjectures on the number of embedded exact fillings of certain twist spuns. Further, we provide obstructions to exact Lagrangian fillability of twist spuns satisfying a specific combinatorial criteria by analyzing fixed points of the cyclic shift action on Grassmanians.

	\end{abstract}
	



	\maketitle

\section{Introduction}

In this article, we study the exact Lagrangian fillability and isotopy characterizations of Legendrian surfaces in the standard contact $\R^5$. Along the way, we extend the correspondence between the theory of cluster algebras and exact Lagrangian fillings of a Legendrian submanifold, to the case of Legendrian surfaces built by {\em doubling} or {\em twist-spinning}. For the case of Legendrian doubles, we use the fact that their sheaf moduli is the intersection of two toric charts in the sheaf moduli of a Legendrian link to draw obstructions to their exact fillability and study their isotopy classes. On the other hand, for the case of twist spuns of braid-positive Legendrians, we identify the coordinate ring of regular functions on its sheaf moduli with a {\em skew-symmetrizable} cluster algebra, giving the first examples of cluster structures on sheaf moduli of Legendrian surfaces. We further establish an analogy between a surgery operation for Lagrangian fillings of these surfaces, and a mutation of cluster variables in the folded cluster algebra. Finally, we also obstruct the existence of embedded exact Lagrangian fillings of twist spun Legendrian tori satisfying a specific combinatorial criteria by showing that their sheaf moduli do not admit any rational points when this criterion is satisfied.

\subsection{Scientific context} Recently, a number of results \cite{CasalsGao, CasalsWeng, GSW, CasalsGao24}, 
have established cluster theory, an algebraic tool connecting combinatorics and geometry, as a valuable aid in distinguishing exact Lagrangian fillings of Legendrian links in the standard contact $\R^3$. This connection was first predicted in an earlier work \cite{STWZ}.\footnote{See Section 2.8 of \cite{CasalsWeng} for a precise discussion of the relationship of \cite{STWZ} to the works cited above.} These connections have led to further understanding of the structure of the augmentation variety \cite{ng2023linfinity, CGGS1},
as well as a conjectural classification of exact Lagrangian fillings of braid positive Legendrian links \cite{CasalsLagSkel}. Given these successes in $\R^3$, it is a natural question whether such connections between cluster theory and exact fillings can be established for Legendrian surfaces in $\R^5$. From another perspective, the symplectic geometric interpretations of cluster structures have led to advancements in the theory of cluster algebras, such as the existence of cluster structures on open Richardson varieties \cite{CGGLSS}, and a proof that the Muller-Speyer twist map on positroid varieties is the Donaldson-Thomas transformation \cite{CLSBW}. These and a collection of similar results indicate that cluster algebras and symplectic topology are a fruitful pairing.

Construction of Legendrian surfaces and their isotopy characterizations in ambient dimension higher than 3 is a subtle problem due to the flexibility of codimension-$(n+1)$ smooth embeddings, when $n>1$. Work of Murphy \cite{Murphy??} relating to {\em loose Legendrians} uses an h-principle to reduce the question of isotopy characterization of loose Legendrians to a formal problem. However, the non-loose case is largely open. In \cite{EkholmEtnyreSullivan05a} Ekholm-Etnyre-Sullivan defined Legendrian contact homology for Legendrians in $\R^{2n+1}$ for $n>1$ and proved its invariance under Legendrian isotopy, which they used to provide examples of infinite families of smoothly isotopic, pairwise Legendrian non-isotopic, non-loose Legendrian surfaces. Following that, further examples were given by Ekholm-K\'alm\'an \cite{EK} where they defined the twist-spinning construction. In later work, Ekholm \cite{ekholm2016non} also defined the doubling construction of Legendrians in arbitrary dimensions. Other constructions of Legendrian surfaces have also been studied in \cite{CasalsMurphy, CasalsZaslow}, the latter of which has contributed greatly to the cluster-theoretic applications outlined in the previous paragraph. 

The topic of exact Lagrangian fillability of Legendrian surfaces was explored in \cite{TreumannZaslow}, where the authors showed that most cubic planar Legendrians do not admit embedded exact Lagrangian fillings. In \cite{Golovko22}, it was shown that there are Legendrians in arbitrary dimensions that admit infinitely many fillings, building off of \cite{CasalsNg}. Connections between cluster algebras and Legendrian surfaces arising from cubic planar graphs have been explored in \cite{SchraderShenZaslow}.

In this article, we show that the sheaf moduli of Legendrian surfaces arising as doubles is the intersection of two toric charts in the sheaf moduli of the Legendrian link\footnote{This computation also follows from work of Li \cite{li2023lagrangian}.}, which is then used to obstruct existence of embedded exact fillings if the double is {\em asymmetric}, even though the Legendrian double is always non-loose when the sheaf moduli is irreducible. 
We define various notions of {\em mutation distance} on exact Lagrangian fillings and the toric charts they introduce,  and we explore the relationship of these distances to isotopy characterizations of doubles arising from exact fillings of torus links. 

In addition to our study of Legendrian doubles, we compute sheaf invariants of twist-spun Legendrian surfaces. We show that if the moduli stack of microlocal rank one sheaves of a braid positive Legendrian link is a globally foldable cluster algebra with respect to a cluster automorphism induced by a Legendrian loop, then the sheaf moduli of the corresponding twist-spun Legendrian surface admits a cluster ensemble structure. Applying the main result of \cite{CasalsGao24} we show that all the seeds in the corresponding cluster algebra arise from embedded exact fillings, which then implies the existence of infinitely many exact Lagrangian fillings for classes of twist-spun Legendrians corresponding to specific cluster types; see Table~\ref{tab: folded} for a description. In other cases we obtain a conjectural classification of the count of fillings, analogous to the ADE conjecture for Legendrian links in $\R^3$ \cite[Conjectures 5.1 and 5.4]{CasalsLagSkel}. On the obstructive side, we use an analysis of fixed points of the cyclic shift map on Grassmanians to show that twist spuns satisfying specific combinatorial criteria do not admit exact fillings, using a necessary condition for fillability arising from \cite{JinTreumann17}. The primary technique we employ to construct and study Legendrian surfaces is Legendrian weaves \cite{CasalsZaslow}.

\subsection{Main results} We describe the main contributions of the article here.

\subsubsection{Legendrian doubles} 

Let $\la\subseteq (\R^3, \xi_{st})$ be a Legendrian link in the standard contact $\R^3$. We denote by $\SM_1(\la)$ the moduli of microlocal rank one sheaves with singular support on $\la$. By \cite{Sheaves3} and \cite{JinTreumann17}, an embedded exact Lagrangian filling $L$ of $\la$ induces an embedding $\mathcal{T}_L=(\C^\ast)^{b_1(L)}\subseteq \SM_1(\la)$. Given two fillings $L$ and $L'$, of $\la$, we can form the Legendrian double $\La(L, L')$ by gluing their Legendrian lifts along their common boundary in the standard contact $\R^5$; see Section~\ref{sec:double}. Direct computation of sheaf invariants for $\La(L, L')$ yields the following.

\begin{theorem}\label{thm: intro_not-loose-fillable}
   If $L$ and $L'$ are embedded exact Lagrangian fillings of $\la$ such that $\mathcal{T}_L\neq \mathcal{T}_{L'}\subseteq \mathcal{M}_1(\la)$, then the asymmetric Legendrian double $\La(L, L')$ does not admit any embedded exact Lagrangian fillings.
\end{theorem}

To our knowledge, every known pair of distinct fillings of the same Legendrian link induce distinct toric charts in $\SM_1(\la)$; we conjecture that one can replace the condition that $L$ and $L'$ induce distinct toric charts with the condition that $L$ and $L'$ are not Hamiltonian isotopic. However, note that these Legendrian surfaces can have singular exact fillings or embedded non-exact fillings. For example, when the doubles arise from cubic planar graphs, such fillings exist by work of Treumann-Zaslow \cite{TreumannZaslow}.

\begin{figure}[h!]{ \includegraphics[scale=0.6]{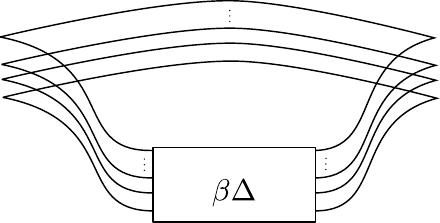}}\caption{Front projection of Legendrian given as the $(-1)$ closure of $\beta\Delta$.}
			\label{fig:-1closure}\end{figure}

We now restrict our attention to Legendrians $\la=\la(\beta\d)\subseteq (\R^3, \xi_{st})$ given as the (-1)-closure of the positive braid $\beta\d$, where we denote by $\d$ the positive half twist element in the braid group. See Figure~\ref{fig:-1closure} for a depiction of the $(-1)$-closure.  We further require that the Demazure product of $\beta$ is equal to $w_0$, the Coxeter projection of $\d$; see e.g. \cite[Section 4.4]{CGGS1} for a definition of the Demazure product.\footnote{By \cite[Theorem 1.1 (i)]{CGGS1}, the Legendrian $\la(\beta\d)$ admits an exact Lagrangian filling only if the Demazure product of $\beta$ is equal to $w_0$.} The family of Legendrian links arising as the $(-1)$-closure of $\beta\Delta$ includes positive torus links, as well as all links of isolated plane curve singularities. In this setting, $\SM_1(\la)$ is irreducible, implying that the sheaf invariants of $\La(L, L')$ never vanish. We make the following observation that also follows from the results of \cite{li2023lagrangian}.

\begin{lemma}\label{lem: intro_non_loose}
     If $\la$ is the $(-1)$-closure of a positive braid $\beta\d$, the Legendrian double $\La(L, L')$ is non-loose for any two embedded fillings $L$ and $L'$ of $\la$.
\end{lemma}

As described in \cite{CasalsZaslow}, for such a $\la,$ we can often obtain an exact Lagrangian filling $L$ as the Lagrangian projection of a Legendrian surface $\w_L$, referred to as a Legendrian weave, that can be combinatorially encoded by representing the singularities of the front projection of $\w_L$ by a colored graph. 
One of the advantages of using Legendrian weaves, is that one can combinatorially manipulate the graphs encoding their front projections to perform Legendrian isotopies and identify connect summands of standard and Clifford tori, denoted $\T^2_{\std}$ and $\T^2_C$, respectively.
The following theorem relies on this method to establish a specialization of the main result of \cite{courte2017lagrangian}.

\begin{theorem}\label{thm: intro_Symmetric}
Let $L$ be an embedded exact filling of $\la$ represented by a Legendrian weave. The symmetric double $\La(L, L)$ is Legendrian isotopic to $\#^{b_1(L)} \T^2_{\std}$.
\end{theorem}

A second advantage of Legendrian weaves is that they often give a way to combinatorially identify $\L$-compressing cycles -- specific homology cycles of $L$ that are crucial ingredients for translating between cluster algebras and symplectic geometry. In cases where particular subsets of these cycles have sufficiently simple intersections, we can give a sufficient criterion on $L$ and $L'$ so that the double is a connect sum of standard and Clifford tori.

\begin{theorem}\label{thm:intro_std_and_clifford_tori}
Let $L$ be an embedded exact filling of $\la$ represented by a Legendrian weave. Suppose that $L$ has a collection $\mathcal{C}$ of $k$ \term {\sf Y}-trees for $0\leq k \leq b_1(L)$
and $L'$ is obtained from $L$ by mutating each of the cycles in $\mathcal{C}$ exactly once. Then $\La(L, L')$ is Legendrian isotopic to $\#^{b_1(L)-k}\T^2_{\std}\#^k \T^2_C$.    
\end{theorem}

See Sections~\ref{sec:weaves} and \ref{sub: weave doubles} for definitions of {\sf Y}-trees and \term collections of homology cycles. Note that in the statement of the theorem we allow for any order of mutations of cycles in $\mathcal{C}$, but we require that each cycle in $\mathcal{C}$ appears exactly once in the sequence of mutations that we perform to obtain $L'$ from $L$.

\subsubsection{Cubic planar Legendrian doubles}\label{sub: cubic_planar_intro}

Denote by $\la(2, n)$, the max-tb representative of the Legendrian $(2, n)$ torus link, which we represent as the (-1)-closure of the positive braid $\sigma_1^{n+2}$. We will always assume that $n\geq 2$. The embedded exact fillings of $\la(2,n)$ are conjecturally in a one-to-one correspondence with triangulations of an $(n+2)$-gon, of which there are a Catalan number $C_n = \frac{1}{n+1}\binom{2n}{n}$. These triangulations are dual to trivalent graphs, and for $L$ and $L'$ corresponding to these dual trivalent graphs, the double $\La(L,L')$ can be described as a Legendrian weave corresponding to a trivalent 2-graph, or a {\em cubic planar Legendrian}. The combinatorics of trivalent planar graphs make these a particularly rich source of examples to study.

\begin{definition}
    We will refer to a cubic planar Legendrian as {\em decomposable} if it is Legendrian isotopic to a connect sum of standard and Clifford tori.
\end{definition}
 
To characterize whether the doubles coming from fillings of $\la(2,n)$ decompose as connect sums of standard and Clifford tori, we introduce the notion of {\bf mutation distance} which we denote by $d_\mu$ -- which is the distance in the cluster algebra, computed by the minimal number of mutations, between the cluster charts induced by $L$ and $L'$. Figure~\ref{fig:factorable} (left) gives one of the first examples where the double $\La(L, L')$ cannot be decomposed into a connect sum of standard and Clifford tori. We give a sufficient criterion for this behavior in terms of mutation distance.

\begin{theorem}\label{thm: intro_notdecomposable}
    Let $L$ and $L'$ be two embedded trivalent graph fillings of $\la(2,n)$ with $n\geq 4$. If $d_{\mu} (L, L') \geq n$, the double $\La(L,L')$ is not Legendrian isotopic to a connect sum of standard and Clifford tori.
\end{theorem}

For doubles arising from $\la(2,n)$ as above, we use $L_{init}$ to refer to the initial filling corresponding to the binary tree with all vertices on the outside edge; see Figure~\ref{fig: initial} for an example. These fillings satisfy the condition that $d_{\mu}(L, L')\leq n-1$, offering a partial converse to Theorem~\ref{thm: intro_notdecomposable}.

\begin{theorem}\label{thm: intro_initialfilling}
    Let $L$ be any embedded trivalent graph filling of $\la(2,n)$. The double $\La(L_{init},L)$ decomposes as a connect sum of standard and Clifford tori.
\end{theorem}

\begin{figure}[h!]{ \includegraphics[width=.4\textwidth]{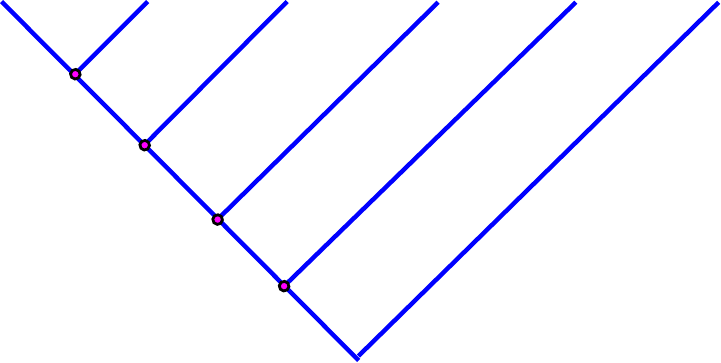}}\caption{Trivalent 2-graph representing the initial filling $L_{init}$ of $\la(2,4)$.}
			\label{fig: initial}\end{figure}

As a particularly simple example, one can enumerate all Legendrian doubles built out of the ${5 \choose 2}$ pairs of trivalent graph fillings of $\la(2, 3)$ and see that these always decompose as a connect sum of standard and Clifford tori.

The proofs of the above theorems rely on elementary combinatorics and the result of \cite{TreumannZaslow, CasalsMurphy2} that the chromatic polynomial of the dual to the 2-graph is an invariant up to Legendrian isotopy. Recall that given a graph, its chromatic polynomial evaluated at a positive integer $n$ returns the number of colorings of the graph vertices with $n$ colors, such that adjacent vertices are colored differently.

The results of \cite{TreumannZaslow} suggest a natural question, of whether the chromatic polynomial is a complete invariant for cubic planar Legendrians. It is straightforward to observe that if a cubic planar Legendrian $\La$ can be obtained by a connect sum between another cubic planar Legendrian and $\mathbb{T}^2_c$, the chromatic polynomial of the dual, denoted $P_{\La^*}(q+1)$, has a $(q-2)$ factor (see Lemma~\ref{lem:chromatic}); here we use $\La^*$ to denote the dual graph of the cubic planar Legendrian $\La$. Legendrian doubles give rise to some curious examples which are worth a mention.

First, Figure~\ref{fig:factorable} depicts a possible counterexample to the expectation that any $q-2$ term appearing in $P_{\La^*}(q+1)$ corresponds to a connect sum with a Clifford torus. Specifically, there exist fillings $L_1, L_2$ of $\la(2, 6)$ such that $\La(L_1, L_2)$ does not appear to reduce but $P_{\La^*(L_1, L_2)}(q+1)$ is divisible by $q-2$. More precisely, the cubic planar picture of $\La(L_1,L_2)$ does not have any triangles, and a search on SageMath \cite{sagemath} revealed that no cubic planar graph with 10 vertices has a chromatic polynomial, evaluated at $(q+1)$, which equals $\frac{1}{q-2}P_{\La^*(L_1, L_2)}(q+1)$. This means there is no cubic planar Legendrian isotopic to $\La(L_1, L_2)$ for which its underlying graph has a triangle.

For our second example of interest, we observe that the underlying trivalent graphs for the cubic planar Legendrians in Figure~\ref{fig:chromaticpolydouble} are non-isomorphic, but yield doubles with the same chromatic polynomial. Unlike similar examples found by Casals and Murphy in \cite[Remark 4.3]{TreumannZaslow}, these have no triangles that correspond to Clifford torus summands. This raises the question -- are these Legendrians isotopic? If not, these would provide counterexamples to the expectation that the chromatic polynomial is a complete invariant of cubic planar Legendrians.

\begin{figure}[h!]{ \includegraphics[width=.4\textwidth]{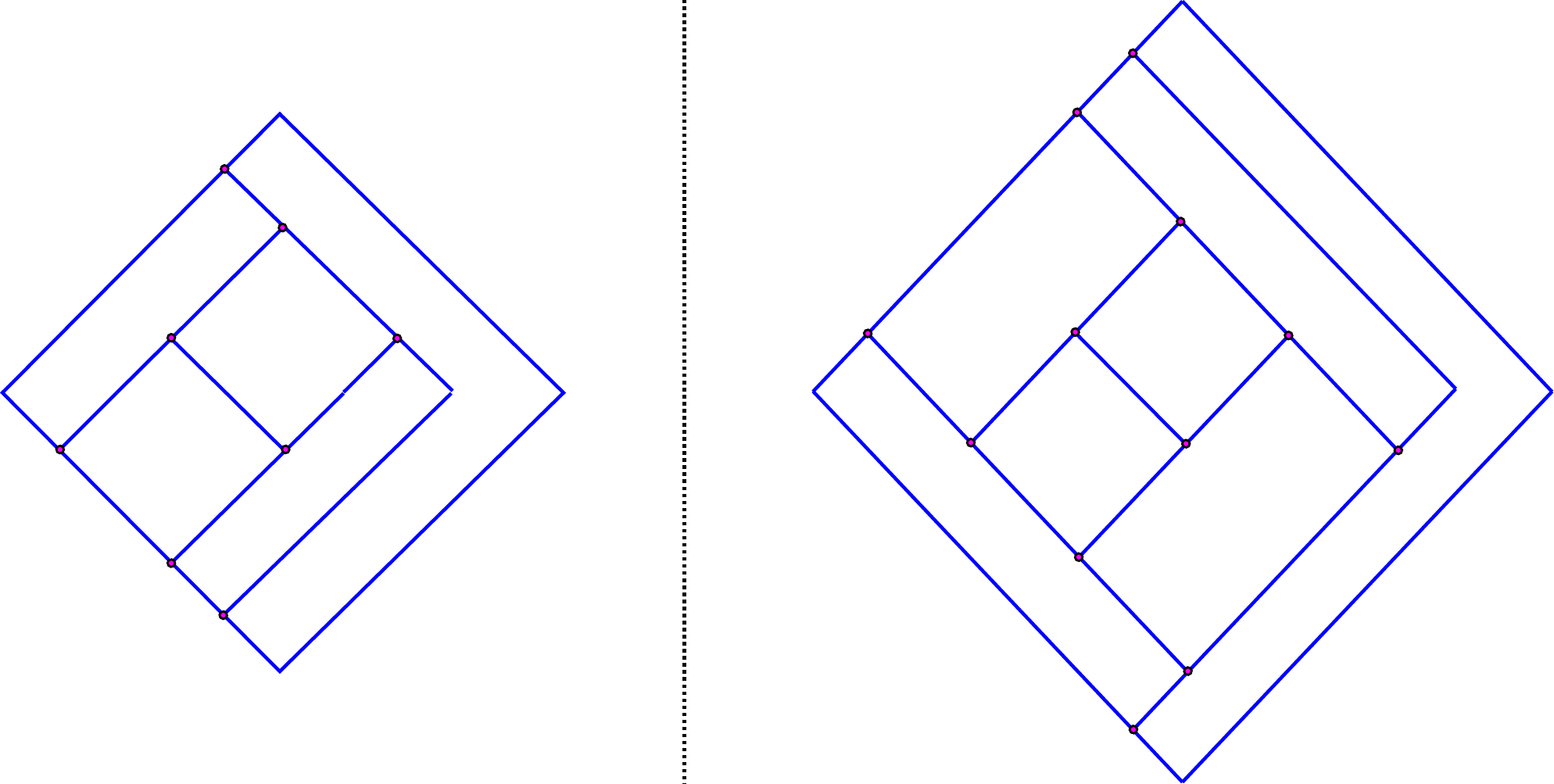}}\caption{Legendrian double that does not decompose as a connect sum of standard and Clifford tori (left) and Legendrian double formed from fillings of $\la(2, 6)$ with a chromatic polynomial divisible by $q-2$ (right).}
			\label{fig:factorable}\end{figure}

\begin{figure}[h!]{ \includegraphics[width=.4\textwidth]{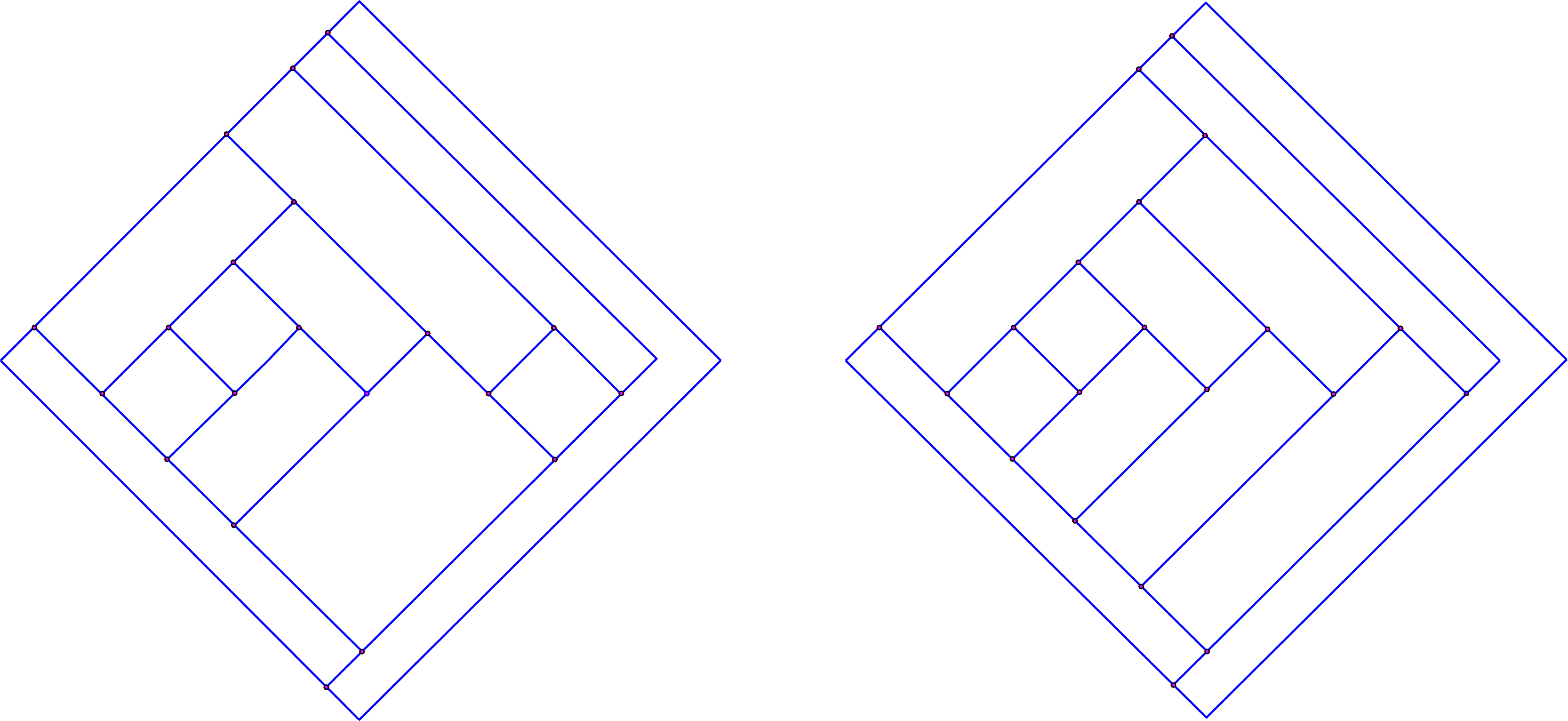}}\caption{Legendrian doubles formed from fillings of $\la(2, 10)$ with the same chromatic polynomial.}
			\label{fig:chromaticpolydouble}\end{figure}

Theorem II of \cite{whitney} (see also \cite{kauffman}) completely characterizes  bridgeless trivalent planar graphs which can be obtained by gluing together two binary trees. These are called ``two tied trees" in \cite{kauffman}. The main characterization is that one cannot “surround” a nontrivial part of a trivalent graph with two or three connected faces. It is clear that all cubic planar Legendrians corresponding to the graphs in the form described above can be obtained by doubles of $\lambda(2,n)$ fillings. This raises an interesting question regarding cubic planar graph Legendrians which cannot be obtained thus. The example of Whitney's \cite[Figure 9]{whitney} of a graph that is not of this form is made up of triangles, so by the moves in Theorem~\ref{thm:czsurgery}, the corresponding Legendrian does arise as a double.

\begin{question}\label{quest: cubic_planar_from_double}
    Is there a cubic planar Legendrian which cannot be obtained as the Legendrian double of two fillings of $\la(2,n)$ for some $n$?
\end{question} 

\begin{figure}[h!]{ \includegraphics[scale=0.2]{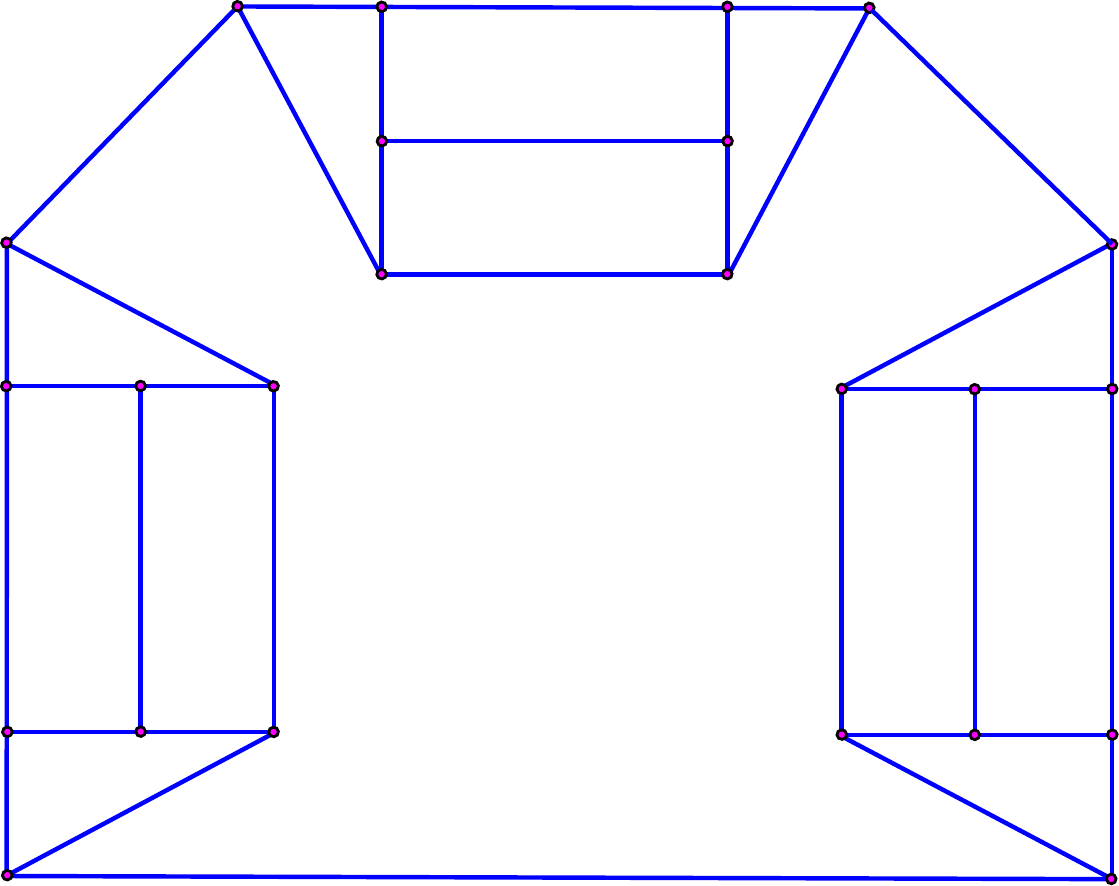}}\caption{A cubic planar Legendrian which is possibly not Legendrian isotopic to a double}
			\label{fig:notdouble}\end{figure}

\begin{ex}
    The trivalent graph $\Gamma$ in Figure~\ref{fig:notdouble} violates the assumptions in \cite[Theorem 2]{whitney} and does not have a simple closed curve passing through every region of the graph exactly once. Moreover, there are no triangle regions in the graph that can be used to decompose it, providing us with a possible example for investigating Question~\ref{quest: cubic_planar_from_double}. As of the time of writing, a SageMath \cite{sagemath} program written to enumerate the chromatic polynomial of all Legendrian doubles, built from trivalent graph fillings of $\la(2,n)$ of the appropriate genus has not output the chromatic polynomial of $\Gamma$. 
\end{ex}

\begin{remark}
    The sheaf-theoretic formulation of the chromatic polynomial also illustrates an interesting connection to the four color theorem (see also \cite[Remark A.11]{CasalsMurphy2}). Namely, the four color theorem is equivalent to the statement that for any $n$, the number of points in the moduli of microlocal rank-one sheaves with singular support on a Legendrian double built from embedded trivalent graph fillings of $\la(2,n)$, with coefficients in $\F_3$, is non-zero. Said differently, an equivalent assertion is that given the moduli of microlocal rank-one sheaves with singular support on $\la(2,n)$ with coefficients in $\F_3$, the intersection of any two toric charts is non-empty.
\end{remark}

Beyond cubic planar Legendrians, we conjecture that mutation distance gives a Legendrian isotopy invariant capable of distinguishing infinitely many distinct Legendrian doubles. We obtain a family of Legendrian doubles from an initial filling $L$ by repeatedly concatenating a Lagrangian concordances known to have infinite order in the Lagrangian concordance monoid to $L$ and then gluing the result again to $L$.  

\begin{remark}\label{rmk: intersections_bad}
    There are two particular challenges underlying this problem that we wish to highlight. The first is the well-established issue of verifying that known constructions yield a complete classification of all possible exact Lagrangian fillings. The second challenge arises from the inherent difficulties of using cluster theory to study the intersection of two distant cluster charts. Cluster mutation is a local operation, in the sense that for two charts related by a single mutation, the mutation formula defines their precise relationship. However, for charts related by an arbitrarily long sequence of mutations, the tools from cluster theory to compare such charts often only exist for finite type cluster algebras and don't appear to give us sufficient data to readily compare Legendrian isotopy classes of doubles.
\end{remark}

\subsubsection{Twist-spun Legendrians} Let $\la$ be a braid-positive Legendrian link and $\varphi$ a Legendrian loop, i.e. a Legendrian isotopy fixing $\la$ setwise. We use $\Sigma_\varphi(\la)$ to denote the corresponding twist-spun Legendrian surface. Smoothly, $\Sigma_\varphi(\la)$ is the mapping torus of $\varphi$; see Section~\ref{sub: twist-spun def} for a more detailed construction.
As described there, a filling of $\Sigma_\varphi(\la)$ can be constructed from a weave filling of $\la$ if the weave is fixed by the action of $\varphi$. See Figure~\ref{fig: symmetric fillings} for examples. 

Let $\rho$ denote the K\'alm\'an loop action on the Legendrian torus link $\la(k, n)$ given by conjugating $(\sigma_1\dots \sigma_{k-1})^n$ by $(\sigma_1\dots \sigma_{k-1})$. For twist-spuns of $\la(3,6)$ and $\la(4,4)$, we obtain infinite families of embedded exact Lagrangian fillings. 

\begin{theorem}\label{thm: intro_infinite}
    The twist-spun Legendrians $\Sigma_{\rho^3}(\la(3,6))$ and $\Sigma_{\rho^4}(\la(4,4))$ admit infinitely many embedded exact Lagrangian fillings.
\end{theorem}

To produce these families, one can perform an ad-hoc construction of initial weave fillings exhibiting the required rotational symmetry. In \cite{Math+25}, the first author, together with V. Chen, P. Galloway, and L. Wei, gave a systematic recipe for producing such symmetric weaves in the case of Legendrian torus links $\la=\la(k,n)$.   

For torus links $\la(2, n)$ where we only expect a finite number of fillings, we can obtain a precise enumeration of exact Lagrangian fillings of $\Sigma_\rho^k(\la(2, n))$ produced by our method. 
    
\begin{theorem}\label{thm: intro_catalan_fillings}
There are at least $f(k)$ embedded exact Lagrangian fillings of $\Sigma_{\rho^k}(\la(2, n))$ where 
\begin{equation}\label{eq: f(k)}
    f(k)=\begin{cases}
    C_{\frac{n+2}{3}} & k=\frac{n+2}{3}\in \N\\ 
    C_{\frac{n}{2}} & k=\frac{n+2}{2} \in \N  \\
    C_n & k=0\\
    0 & \text{otherwise}
\end{cases}
\end{equation}
and $k\in \{0, \dots, n-1\}.$
\end{theorem}

Note that in the case of $k=0,$ $\Sigma_\rho^0(\la(2, n))\cong \la(2,n )\times S^1$ is a spherical spun, and the techniques of \cite{Golovko22} could be applied to recover our enumeration in this particular case.

We prove Theorems~\ref{thm: intro_infinite} and \ref{thm: intro_catalan_fillings} as applications of a general theorem on existence of cluster structures on a moduli of sheaves of twist-spun Legendrians that satisfy a technical condition discussed below.\footnote{The case of $k=\frac{n+2}{3}$ technically corresponds to a generalized cluster structure studied by Fraser in \cite{fraser2020cyclic}; see Section~\ref{sub: twist-spun enumeration} for more details.}

\begin{theorem}\label{thm: intro_ensembles}
Let $G$ be a finite group generated by the action of a Legendrian loop $\varphi$ of a braid positive Legendrian link $\la$ and suppose that $C[\FM(\la, T)]$ is a globally foldable cluster algebra with respect to the $G$-action. Assume that $L$ is an embedded exact Lagrangian filling of $\la$ fixed by the action of $\varphi$ and equipped with a maximal linearly independent collection of $\L$-compressing cycles. The moduli stacks $\FM(\Sigma_\varphi(\la, T))$ and $\SM_1(\Sigma_\varphi(\la, T))$ form a cluster ensemble with every cluster chart induced by an embedded exact Lagrangian filling of $\Sigma_\varphi(\la)$.
\end{theorem}

In the statement of the theorem, ``globally foldable" is a technical condition that allows one to obtain a new cluster algebra from an original cluster algebra and a compatible $G$-action. See Definition~\ref{def: globally foldable} for the formal definition.

\begin{remark}
The existence of a cluster structure on the sheaf moduli of a Legendrian link equipped with a specific group action was already implicitly used in \cite{ABL2021} and made explicit in \cite[Theorem 6.6]{Hughes2024}. 
The novel content of Theorem~\ref{thm: intro_ensembles} lies in the description of how the elements of these folded cluster algebras arise from the geometry of twist-spun Legendrian links and their fillings -- see Subsection~\ref{sub: clusters from folding} in particular -- as well as the resulting ability to distinguish these fillings, as we make use of in Theorems~\ref{thm: intro_infinite} and \ref{thm: intro_catalan_fillings}. Indeed, while we conjecture that there is a one-to-one correspondence between cluster seeds of $\FM(\Sigma_\varphi(\la), \tau)$ and fillings of $\Sigma_\varphi(\la)$, in line with \cite[Conjecture 5.4]{CasalsLagSkel}, it is a priori not clear whether every filling of $\Sigma_\varphi(\la)$ must be of the form $L \times_\varphi S^1$. 

\end{remark}

The first step in proving the above involves showing that these moduli stacks of sheaves can be identified as the $G$-invariant part of the sheaf moduli $\mathcal{M}_1(\la)$. The next step is identifying the character lattices with homology groups of the twist-spun Legendrians, and then computing coordinates on the toric chart induced by $L\times_\varphi S^1$ from microlocal parallel transport along (relative) homology cycles of $\Sigma_\varphi(\la)$ representing orbits of cycles in $H_1(L)$ under the induced action of $\varphi$. The last step consists of defining a higher-dimensional version of Lagrangian disk surgery \cite{Yau}, and showing that the cluster variables for a {\em solid mutation configuration} change according to cluster transformations. We then adapt the work of \cite{CasalsGao24} to show that every cluster seed in $\SM_1(\Sigma_\varphi(\la))$ can be induced by an embedded filling.

\begin{remark}\label{rem: surgery_lit}
    We point out that the notion of Lagrangian surgery we use is related to other notions studied in the literature. It can be understood as modifying the construction in \cite[Section 2.3]{Yau} by spinning in a higher dimension. A related construction is that of a BSP-surgery defined by Chanda \cite{chanda_bsp} --- from that perspective the notion of surgery we use is the BSP-surgery with $\La = T^2_{\std}$. A version of the surgery is described in \cite[Section 2.1]{ChandaHirschiWang} for higher dimensions which generalizes the discussion in \cite[Section 5]{pascaleff_tonkonog} -- however the toric assumption, or the assumption on Lie group embeddings, is not needed to define the surgery on Lagrangians.
\end{remark}

Having established a relationship between fillings of $\Sigma_\varphi(\la)$ and cluster charts of $\FM(\la(\beta))$ in Theorem~\ref{thm: intro_ensembles}, we make the following conjecture.

\begin{conjecture}\label{conj: number_fillings}
    There are exactly $f(k)$ exact Lagrangian fillings of $\Sigma_{\rho^k}(\la(2, n))$, with $f(k)$ as in Equation~\ref{eq: f(k)}.
\end{conjecture}

Conjecture~\ref{conj: number_fillings} is motivated in the same spirit as \cite[Conjectures 5.1 and 5.4]{CasalsLagSkel}; each filling constructed is expected to correspond to a unique toric chart in the cluster algebra and $f(k)$ describes the number of seeds in the cluster algebra of Theorem~\ref{thm: intro_ensembles}.

    Theorems~\ref{thm: intro_infinite} and \ref{thm: intro_catalan_fillings} can be seen as contrasting with the general lack of fillings of Legendrian doubles as expressed in Theorem~\ref{thm: intro_not-loose-fillable}. One can interpret this as the relative compatibility of cluster structures with the folding operation described in Subsection~\ref{sub: folding} as opposed to the process of taking intersections of toric charts, which, as explained in Remark~\ref{rmk: intersections_bad} above, is not as compatible with cluster theory. In particular, when folding a cluster algebra, one obtains a (possibly generalized) cluster seed, whereas intersections of distinct cluster charts necessarily fail to yield a reasonable notion of a cluster structure on the intersection.

Despite these differences, we are also able obstruct exact Lagrangian fillability of twist-spun Legendrians in infinitely many cases.

\begin{theorem}\label{thm: Intro_nonfillable_spuns}
Let $\ell$ and $n$ be relatively prime positive integers. There are infinitely many values of $n$ such that twist-spuns $\Sigma_{\rho^\ell}(\la(2, n))$ and $\Sigma_{\rho^\ell}(\la(3, n))$ do not admit any exact Lagrangian fillings.
\end{theorem}

See Equations~\ref{eqn: divisibility1} and \ref{eqn: divisibility2} for a precise description of conditions on $n$ that we require for our proof. We expect that the conclusion of Theorem~\ref{thm: Intro_nonfillable_spuns} holds more generally; see Remark~\ref{rmk: obstruction_hard} for additional discussion.

\subsection{Organization} We review background relating to contact and symplectic topology and geometry in Section~\ref{sec:background_top}, and also define the high-dimensional Lagrangian surgery operation. In Section~\ref{sec: sheaves_and_clusters} we survey the results and background we need regarding the microlocal theory of sheaves and cluster algebras. We then prove Theorems~\ref{thm: intro_Symmetric}, \ref{thm: intro_not-loose-fillable}, and \ref{thm:intro_std_and_clifford_tori} in Section~\ref{sec: legendrian_doubles} and make other observations about doubled Legendrians. In Section~\ref{sec: torus_knot_doubles}, we define mutation distance and talk about related notions, and prove Theorems~\ref{thm: intro_notdecomposable} and \ref{thm: intro_initialfilling}. In Section~\ref{sec: twist-spuns}, we first prove Theorem~\ref{thm: intro_ensembles}, and then as applications, prove Theorems~\ref{thm: intro_infinite} and \ref{thm: intro_catalan_fillings}. We conclude Section~\ref{sec: twist-spuns} with the proof of Theorem~\ref{thm: Intro_nonfillable_spuns}.

\subsection{Acknowledgments}
This project began during the SyNC workshop at UC Davis in August 2022. Thank you to the organizers, Orsola Capovilla-Searle, Roger Casals, and Caitlin Leverson. Thanks also to Scott Baldridge, Roger Casals,  Soham Chanda, Wenyuan Li, Lenny Ng, and David Treumann for helpful conversations. Some initial findings of this article were part of A.R.'s PhD thesis advised by John Etnyre. Special thanks to Daping Weng for explaining the proof of Lemma~\ref{lem: Daping's Lemma} to us. J.H. was partially supported by NSF grant DMS-1942363. A.R. was partially supported by NSF grants DMS-2203312 and DMS-1907654, and an AMS-Simons Travel Grant. We are further grateful to an anonymous referee for their careful reading and suggestions.

\section{Background - Topology}\label{sec:background_top}

\subsection{Legendrian submanifolds and Lagrangian fillings}

We begin with the necessary background on Legendrian links and their exact Lagrangian fillings. The standard contact structure $\xi_{st}$ in $\R^{2n+1}$ is the plane field given as the kernel of the 1-form $\alpha_{\st}=dz-\sum_{i=1}^n y_idx_i$. An $n$-dimensional submanifold $\La \subseteq (\R^{2n+1}, \xi_{st})$ is Legendrian if $\La$ is always tangent to $\xi_{st}$. In $\R^3$, as any Legendrian link $\la$ can be assumed to avoid a point, we can equivalently consider Legendrians $\la$ contained in the contact 3-sphere $(\mathbb{S}^3, \xi_{st})$ \cite[Section 3.2]{Geiges08}. We consider Legendrians up to Legendrian isotopy, i.e. ambient isotopy through a family of Legendrians. 

In this work, we will often work with various projections of Legendrians. We denote by $\Pi$ the front projection $\Pi: \R^{2n+1}\to \R^{n+1}$ given by projecting out the $y_1, \dots, y_n$ coordinates and by $\pi$ the Lagrangian projection $\pi(\La): \R^{2n+1}\to \R^{2n}$ given by projecting out the $z$ coordinate.  For Legendrians in $J^1 C \cong T^*C \times \R$, we choose coordinates $x_1, \dots, x_n$ for points in $C$, $y_1, \dots, y_n$ for the cotangent fibers of $T^*C$, and $z$ for the additional $\R$ factor. With these coordinates, the front projection $\Pi:J^1C\to C\times \R$ and Lagrangian projection $\pi:J^1C \to T^*C$ are defined exactly the same as in $\R^{2n+1}$. See Figure~\ref{fig: wavefronts} for examples of local models of front projections of Legendrian surfaces.

The symplectization $\Symp(M, \ker(\alpha))$ of a contact manifold $(M, \ker(\alpha))$ is the symplectic manifold $(\R_t\times M, d(e^t\alpha))$. Given two Legendrian links $\la_-, \la_+\subseteq (\R^3,\xi_{\st})$, 
 an exact Lagrangian cobordism $L\subseteq \Symp(\R^3, \ker(\alpha_{\st}))$ from $\la_-$ to $\la_+$ is a cobordism $\Sigma$ such that there exists some $T>0$ satisfying the following: 

	\begin{enumerate}
		\item $d(e^t\alpha_{\st})|_\Sigma=0$ 
		\item $\Sigma\cap ((-\infty, T]\times \R^3)=(-\infty, T]\times \la_-$ 
		\item $\Sigma\cap ([T, \infty)\times \R^3)=[T, \infty) \times \la_+$ 
		\item $e^t\alpha_{\st}|_\Sigma=df$ for some function $f: \Sigma\to \R$ that is constant on $(-\infty, T]\times \la_-$ and $[T, \infty)\times \la_+$. 
	\end{enumerate}

	An exact Lagrangian filling of the Legendrian link $\la\subseteq (\R^3, \xi_{\st})$ is an exact Lagrangian cobordism $L$ from $\emptyset$ to $\la$ that is embedded in $\Symp(\R^3, \ker(\alpha_{\st}))$. Equivalently, we consider $L$ to be embedded in the symplectic 4-ball with boundary $\partial L\subseteq (\S^3, \xi_{\st})$. We will always assume that our fillings are orientable and, with the exception of Example~\ref{ex: immersed}, we will assume that our fillings are embedded rather than immersed. 
    
    In this work, we also require the notion of an exact Lagrangian filling of a Legendrian surface $\La$, which is a 3-dimensional Lagrangian cobordism from $\emptyset$ to $\La$ satisfying the same properties as an exact Lagrangian cobordism of a Legendrian link, replacing $\R^3$ with $\R^5$ where appropriate.

\subsection{The doubling construction of Legendrians}\label{sec:double}

Given a Legendrian link $\la \subset (\R^3, \xi_{st})$, and two exact Lagrangian fillings $L_1, L_2$, we recall the definition of the doubled Legendrian $\La(L_1,L_2)$ from Section 2.1 of \cite{ekholm2016non}.

Let $(x_1,y_1,x_2,y_2)$ be coordinates on $\R^4$ with the standard symplectic form $-d(y_1dx_1+y_2dx_2)$, and the symplectization $\R \times \R^3$, of $(\R^3, \xi_{st})$, as above.

Consider the exact symplectomorphism $\phi_T: (\R \times \R^3) \to \R^4$, defined by
\[
\phi_T(t,x,y,z) = (x,e^ty,e^t - e^T - 1, z)
\]

Suppose the Legendrian $\la$ is parametrized, and lives at $t= T$, and the parametrized exact Lagrangian filling $L_1$ lives in $t \leq T$, and is a cylinder on $\la$ near $T$. The image of $L_1$ under $\phi_T$ lands in $\{x_2 \leq -1\}$. Suppose further that $L_1$ agrees with $\{t\} \times K$ on $(T-1) < t \leq T$ -- call this the {\em conical end} of $L_1$. Then in that region, the coordinates for $L_1$ in $(\R \times \R^3)$ look like $(t,x, dz/dx, z)$ where $(x,z)$ are the coordinates for the front projection of $\la$. Under $\phi_T$, this maps to $(x, e^t(dz/dx), e^t - e^T - 1, z)$. We elaborate on that in the following.

Identify this $\R^4$ with $\{ x_2 \leq -1,w = 0\}\subset \R^5$, where $\R^5$ is equipped with the standard contact structure $\xi_{st} = \ker(dw - y_1dx_1 - y_2dx_2)$. The contact form restricts to $\{ x_2 \leq -1,w = 0\}$ to the primitive of the standard symplectic form $-d(y_1dx_1+y_2dx_2)$, mentioned above. An exact Lagrangian $\phi_T(L_1) \subset (\R^4, -d(y_1dx_1+y_2dx_2))$, by definition, admits a function $g_{L_1} \in C^{\infty}(L_1)$ such that $dg_{L_1} = -(y_1dx_1+y_2dx_2)|_{L_1}$. If $\phi_T(L_1)$ is considered to be an embedded exact Lagrangian in $\{x_2\leq-1,w=0\}$, its Legendrian lift $\Lambda(\phi_T(L_1))$ in $(\R^5, \ker(dw - y_1dx_1 - y_2dx_2))$ is parametrized by $\{(p,-g_{L_1}(p))\mid p \in L_1\}$. Our assumptions on $L_1$ will let us observe that the Legendrian surface $\Lambda(\phi_T(L_1))$ can be extended to live in $\{x_2 \leq 0\} \subset \R^5$, such that its intersection with $\{x_2=0,y_2=0\}$ is exactly the parametrized Legendrian knot $\la$.

Now consider the conical end of $\phi_T(L_1)$, i.e. $\phi_T$ restricted to \[L_1 \cap \{(t,x,y,z) \in \R \times \R^3 \mid T-1 \leq t \leq T\}\] By our assumption on $L_1$, the $\R^4$-coordinates for $\phi_T(L_1)$ are $(x,e^t(dz/dx),e^t-e^T-1,z)$, where $(x,z)$ are the coordinates on the front projection of $\la$, which is the boundary Legendrian of the exact Lagrangian $L_1$. The conical end of $\phi_T(L_1)$ thus described can be Legendrian lifted to $\big(x, e^t(dz/dx), e^t - e^T - 1, z, e^tz\big)$ (the fifth coordinate is the $w$ coordinate), whose front projection to the $(x_1, x_2, w)$ coordinates is $(x, e^t-e^T-1, e^tz)$, i.e. the front in $(x_1, w)$ plane scaled in the $x_2$ direction. 

To summarize, at this step, we have built a Legendrian surface $\Lambda(\phi_T(L_1)) \subset (\R^5, \ker(dw-y_1dx_1-y_2dx_2))$, such that on the set $\{e^{T-1}-e^T-1 \leq x_2 \leq -1\}$, $\Lambda(\phi_T(L_1))$ is parametrized as $\big(x, e^t(dz/dx), e^t - e^T - 1, z, e^tz\big)$, where $(x,z)$ parametrizes the front projection of the Legendrian knot $\la$ in $(\R^3, \ker(dz-ydx))$. By a slight abuse of notation, we will refer to this surface as $\Lambda_T(L_1)$.

Now, consider the reflection in $\R^4$ defined by $(x_1, y_1, x_2, y_2) \mapsto (x_1,y_1,-x_2,-y_2)$. Composing this reflection with $\phi_{T'}$ for some $T'$, we can map $L_2$ to an exact Lagrangian that lives in $\{x_2 \geq 1\} \subset \R^4$. Similarly as above, we can lift it to a Legendrian $\Lambda_{T'}(L_2)$ for the standard contact $\R^5$, such that $\Lambda(L_2) \subset \{x_2 \geq 1\}$. Further, we can ensure that on the set $\{1+e^{T'}-e^{T'-1} \geq x_2 \geq 1\}$, $\Lambda(L_2)$ is parametrized as $\big(x, e^t(dz/dx),-( e^t - e^{T'} - 1),- z, e^tz\big)$, where $(x,z)$ parametrizes the front projection of the Legendrian knot $\la$ in $(\R^3, \ker(dz-ydx))$.

Now, these two Legendrian surfaces with boundary, $\Lambda_T(L_1)$ and $\Lambda_{T'}(L_2)$, that live in $\{x_2 \leq -1\}$ and $\{x_2 \geq 1\}$ respectively, can be glued, as follows. For ease of notation, we can choose $T = -1$ and $T'=1$. Consider the Legendrian `annulus' in standard contact $\R^5$ corresponding to the front $(x, s(t), e^Tz)$ where $s(t)$ is a smooth function of $t \in (-1.5.1.5)$ such that 
\begin{align*}
    s(t) &= e^t-e^{-1}-1 \text{ for } t \leq -1, \\
    &= t \text{ for } t\in (-\frac{1}{2},\frac{1}{2})\\
    &= 1+e-e^t \text{ for } t\geq 1,
\end{align*} and as before $(x,z)$ are the coordinates in the front projection of $\la$. The glued Legendrian is the doubled Legendrian, denoted $\La(L_1, L_2)$, since the parametrization, and especially $T$ and $T'$, do not affect the Legendrian isotopy class of the double. In this work, we will call a double {\em symmetric} if $L_1$ and $L_2$ are Hamiltonian isotopic, and {\em asymmetric} otherwise.

In \cite{courte2017lagrangian}, the authors show that the isotopy class of a symmetric double $\La(L,L)$ is determined by the formal data associated to the filling $L$. They prove the following result.

\begin{theorem}\label{thm:trivdouble}\cite[Theorem 1]{courte2017lagrangian}
    If $L \subset (\R \times \R^{2n-1})$ is an embedded exact Lagrangian submanifold with Legendrian boundary, then the Legendrian isotopy class of $\La(L,L) \subset \R^{2n+1}$ is determined by the induced trivialization of $TL \otimes \C$.
\end{theorem}

\subsection{The twist-spun construction of Legendrians}\label{sub: twist-spun def}

Let $\la\subseteq (\R^3,\xi_{st})$ be a Legendrian link and $\varphi$ be a Legendrian loop of $\la$. That is, $\varphi=\{\varphi_\theta\}_{\theta\in [0,1]}$ is a Legendrian isotopy with $\varphi_1$ fixing $\la$ setwise. The image $\{\varphi_\theta(\la)\}:=\{\la_\theta\}$ is an $S^1$ family of Legendrians, and we can form the mapping torus of $\varphi$ to obtain a Legendrian surface in contact $(\R^5, \xi_{st})$. More explicitly, we define $\Sigma_{\varphi}(\la)$ by observing that an $S^1$ family of Legendrians $\{\la_\theta\}\subseteq \R_z\times T^*\R_x\times S^1$ 
lifts uniquely to a Legendrian $\Sigma_{\varphi}(\la)\subseteq \R_z\times T^*\R_{x\geq 0}\times T^*S^1$ with contact structure $\ker(dz-p_\theta d\theta  -  p_xdx)$, where $(p_x,x)$ are the coordinates on $T^*\R_x$, and $(p_\theta,\theta)$ are the coordinates on $T^*S^1_{\theta}$. We can then canonically identify $T^*\R_{x\geq 0} \times T^*S^1$ with $T^*\R^2$ by the map $\R_{x\geq 0}\times S^1\to \R\times \R\backslash\{0\}$ defined by $(x, \theta)\mapsto xe^{i\theta}$, giving a (strict) contact embedding $(\R_z\times T^*\R^2\times \T^*S^1, dz-p_\theta d\theta  -  p_xdx)\xhookrightarrow{} (\R^5,\xi_{st})$. See \cite[Section 1.3]{DRG21} for more details. 

In this sense, we can give the following definition. 

\begin{definition}\label{def: twist-spuns}
    Given a Legendrian loop $\varphi$ of a Legendrian link $\la$, the twist-spun Legendrian  $\Sigma_\varphi(\la)$ is the union of 
    Legendrian tori $$\la\times [0,1]/(\varphi(\la)\times \{0\}\sim \la\times \{1\})$$ embedded in $(\R^5,\xi_{st})$. 
\end{definition}

We will generally consider twist-spuns constructed by a particularly simple Legendrian loop referred to as \emph{cyclic rotation}. Given a Legendrian link $\la$ presented as the $(-1)$-closure of a positive braid $\beta$ of the form $\sigma_i\tilde{\beta}$ the cyclic rotation $\delta$ of $\la$ is given by the isotopy conjugating $\beta$ by $\sigma_i$ to yield $\tilde{\beta}\sigma_i$. When $\beta$ can be written as $w^n$ for some braid word $w$, then $\delta^{|w|}$ conjugates $\beta$ by $w$ and is a Legendrian loop of $\la$. In the case of Legendrian positive torus links $\la(k, n)$, the loop $\rho:=\delta^{k-1}$ gives a Legendrian loop studied by K\'alm\'an in \cite{Kalman} that we refer to as the \emph{K\'alm\'an loop}.   

\subsection{Legendrian weaves}\label{sec:weaves}

Let $C$ be an oriented surface. In this section, we describe Legendrian weaves, a geometric construction of Casals and Zaslow that can be used to combinatorially represent Legendrian surfaces $\Lambda$ in the 1-jet space $J^1C=T^*C\times \R_z$ 
by the singularities of their front projection in $C\times \R_z$. In practice, when $\La$ has a Legendrian link at its boundary, one can consider the Lagrangian projection of $\La$ in order to obtain a (possibly immersed) exact Lagrangian filling of $\la=\partial \La$. Although the Legendrian surfaces we construct in this work do not have boundary, the Legendrian weaves we define in this section will generally be taken to have nonempty boundary, as we construct Legendrian doubles by gluing together two such (Lagrangian projections of) Legendrian weaves. For this article we will only need $C = \D^2$ or $C = S^2$.

Let $\beta\in Br_n^+$ be a positive braid. The contact geometric description of a Legendrian weave surface with boundary $\la(\beta\Delta)$ is as follows. We construct a filling of a Legendrian $\la(\beta\Delta)$ by first describing a local model for a Legendrian surface $\La$ in $J^1\D^2=T^*\D^2\times \R_z$. We equip $T^*\D^2$ with the symplectic form $d(e^r\alpha)$ where $\ker(\alpha)=\ker(dy_1-y_2d\theta)$ is the standard contact structure on $J^1(\partial \D^2)$ and $r$ is the radial coordinate. This choice of symplectic form ensures that the flow of $e^r\alpha$ is transverse to $J^1\S^1\cong \R^2\times \partial \D^2$ thought of as the cotangent fibers along the boundary of the 0-section. The Lagrangian projection of $\La$ is then a Lagrangian surface in $(T^*\D^2, d(e^r\alpha))$. Moreover, since $\La\subseteq (J^1\D^2, \ker (dz-e^r\alpha))$ is a Legendrian, we immediately obtain the function $z:\pi(\La)\to \R$ satisfying $dz=e^r\alpha|_{\pi(\La)}$, demonstrating that $\pi(\La)$ is exact.

The boundary of $\pi(\La)$ is taken to be a positive braid $\beta$ in $J^1\S^1$  so that we regard it as a Legendrian link in a contact neighborhood of $\partial \D^2$. As the 0-section of $J^1\S^1$ is Legendrian isotopic to a max-tb standard Legendrian unknot, we can take $\partial\pi(\La)$ to equivalently be the standard satellite of the standard Legendrian unknot. Diagramatically, this allows us to express the braid $\beta$ in $J^1\S^1$ as the $(-1)$-framed closure of $\beta$ in standard contact $\S^3$.

The immersion points of a  Lagrangian projection of a weave surface $\La$ correspond precisely to the Reeb chords of $\La$. In particular, if $\La$ has no Reeb chords, then $\pi(\La)$ is an embedded exact Lagrangian filling of $\partial(\La)$. In the Legendrian weave construction, Reeb chords correspond to critical points of the function defined as the difference of heights between sheets. See the middle and bottom rows of Figure~\ref{fig:surgery} for examples where this function necessarily admits critical points.

\subsubsection{$N$-Graphs and Singularities of Fronts} To construct a Legendrian weave surface $\La$ in $J^1\D^2,$ we combinatorially encode the singularities of its front projection in a colored graph. Local models for these singularities of fronts are given by Arnol'd \cite[Section 3.2]{ArnoldSing}; the three singularities that appear in our construction describe elementary Legendrian cobordisms and are pictured in Figure~\ref{fig: wavefronts}.

	\begin{center}
		\begin{figure}[h!]{ \includegraphics[width=.8\textwidth]{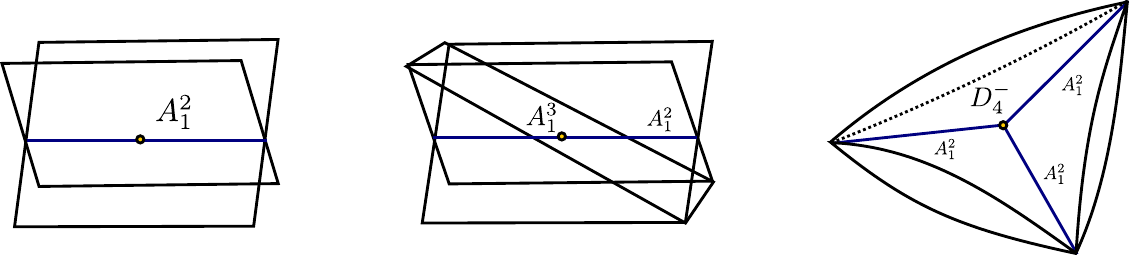}}\caption{Singularities of front projections of Legendrian surfaces. Labels correspond to notation used by Arnold in his classification.}
			\label{fig: wavefronts}\end{figure}
	\end{center}

Since the boundary of our singular surface $\Pi(\La)$ is the front projection of an $N$-stranded positive braid, $\Pi(\La)$ can be pictured as a collection of $N$ sheets away from its singularities. We describe the behavior at the singularities as follows:

\begin{enumerate}
    \item The $A_1^2$ singularity occurs when two sheets in the front projection intersect. This singularity can be thought of as the trace of a constant Legendrian isotopy in the neighborhood of a crossing in the front projection of the braid $\beta\Delta^2$. 
    \item The $A_1^3$ singularity occurs when a third sheet passes through an $A_1^2$ singularity. This singularity can be thought of as the trace of a Reidemeister III move in the front projection.
    \item A $D_4^-$ singularity occurs when three $A_1^2$ singularities meet at a single point. This singularity can be thought of as the trace of a 1-handle attachment in the front projection. 
\end{enumerate}

Having identified the singularities of fronts of a Legendrian weave surface, we encode them by a colored graph $\Gamma\subseteq \D^2$. 

\begin{definition}
    An $N$-graph $\Gamma$ is a planar graph with edges labeled by positive Artin generators of the braid group $\Br_N$. We require that any pair of edges labeled $\sigma_i$ and $\sigma_{i+1}$ can only meet at a hexavalent vertex with alternating labels, while any edges labeled $\sigma_i$ can only meet at trivalent vertices. Edges labeled by $\sigma_i$ and $\sigma_j$ for $|i-j|\geq 2$ can only meet at tetravalent vertices with alternating labels.
\end{definition}

To obtain a Legendrian weave $\Lambda(\Gamma)\subseteq (J^1\D^2,\xi_{\st})$ from an $N$-graph $\Gamma$, we glue together the local germs of singularities according to the edges of $\Gamma$. First, consider $N$ horizontal sheets $\D^2\times \{1\}\sqcup \D^2\times \{2\}\sqcup \dots \sqcup \D^2\times \{N\}\subseteq \D^2\times \R$ and an $N$-graph $\Gamma\subseteq \D^2\times \{0\}$. We construct the associated Legendrian weave $\Lambda(\Gamma)$ as follows  \cite[Section 2.3]{CasalsZaslow}.

	\begin{itemize}
		\item Above each edge labeled $\sigma_i$, insert an $A_1^2$ crossing between the $\D^2\times\{i\}$ and $\D^2\times \{i+1\}$ sheets so that the projection of the $A_1^2$ singular locus under $\Pi:\D^2\times \R \to \D^2\times \{0\}$ agrees with the edge labeled $\sigma_i$.   
		\item At each trivalent vertex $v$ involving three edges labeled by $\sigma_i$, insert a $D_4^-$ singularity between the sheets $\D^2\times \{i\}$ and $\D^2\times\{i+1\}$ in such a way that the projection of the $D_4^-$ singular locus agrees with $v$ and the projection of the $A_2^1$ crossings agree with the edges incident to $v$.
		\item At each hexavalent vertex $v$ involving edges labeled by $\sigma_i$ and $\sigma_{i+1}$, insert an $A_1^3$ singularity along the three sheets in such a way that the origin of the $A_1^3$ singular locus agrees with $v$ and the $A_1^2$ crossings agree with the edges incident to $v$.
	\end{itemize} 
	\begin{center}
		\begin{figure}[h!]{ \includegraphics[width=\textwidth]{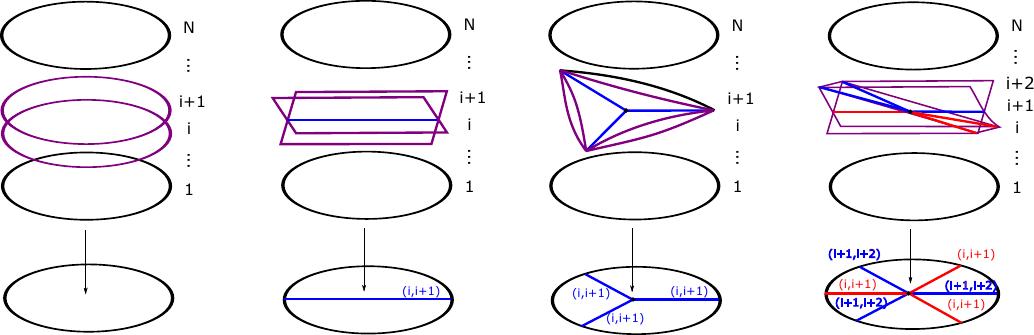}}\caption{The weaving of singularities of fronts along the edges of the $N$-graph (courtesy of Roger Casals and Eric Zaslow, used with permission). Gluing these local models according to the $N$-graph $\Gamma$ yields the weave $\Lambda(\Gamma)$.}
			\label{fig:Weaving}\end{figure}
	\end{center}
	
 Note that the tetravalent vertex does not correspond to a distinct singularity of the front projection, as it represents a pair of disjoint $A_1^2$ singularities.

    If we take an open cover $\{U_i\}_{i=1}^m$ of $\D^2\times \{0\}$ by open disks, refined so that any disk contains at most one of these three features pictured in Figure~\ref{fig:Weaving}, we can glue together the resulting fronts according to the intersection of edges along the boundary of our disks. Specifically, if $U_i\cap U_j$ is nonempty, then we define $\Pi(\La(U_1\cup U_2))$ to be the front resulting from considering the union of fronts $\Pi(\La(U_1))\cup\Pi(\La(U_j))$ in $(U_1\cup U_2)\times \R$.

	\begin{definition}\label{def: weave}
	    The Legendrian weave $\Lambda(\Gamma)\subseteq (J^1\D^2, \xi_{st})$ is the Legendrian lift of the front $\Pi(\La(\cup_{i=1}^m U_i))$ given by gluing the local fronts of singularities together according to the $N$-graph $\Gamma$.
	\end{definition}

Below we will often denote a weave by $\La$ unless we wish to emphasize its relationship with the underlying $N$-graph $\Gamma$, in which case we will retain the notation $\La(\Gamma)$.

\subsubsection{Homology of weaves}	
	In this subsection, we describe the homology of a Legendrian weave $\La(\Gamma)$. The smooth topology of $\La(\Gamma)$ is that of an $N$-fold branched cover over $\D^2$ with simple branched points corresponding to each of the trivalent vertices of $\Gamma$. 
	Assuming that $\La(\Gamma)$ is connected, the genus of $\La(\Gamma)$ is then computed using the Riemann-Hurwitz formula: 
	$$
	g(\La(\Gamma))=\frac{1}{2}(v(\Gamma)+2-N\chi(\D^2)-|\partial\La(\Gamma)|)$$
	where $v(\Gamma)$ is the number of trivalent vertices of $\Gamma$ and $|\partial\La(\Gamma)|$ denotes the number of boundary components of $\Gamma$.


We now describe a recipe for combinatorially identifying particular elements of $H_1(\La(\Gamma));\Z)$ known as $\mathbb{L}$-compressing cycles. We start by defining these objects.

\begin{definition}\label{def: L-compressing cycle}
    An $\L$-compressing cycle $\gamma\in H_1(L)$ of an exact Lagrangian filling $L$ is a homology cycle that bounds an embedded Lagrangian disk in the complement of $L$. If $H_1(L)$ admits a basis of $\L$-compressing cycles, then we refer to such a collection as an $\L$-compressing system of $L$.
\end{definition}

As described in Section \ref{sec: sheaves_and_clusters}, $\mathbb{L}$-compressing cycles play a crucial role in defining the cluster seed associated to an exact Lagrangian filling. In order to identify their Legendrian lifts in a weave, we consider an edge connecting two trivalent vertices. Closely examining the sheets of our surface, we can see that each such edge corresponds to a 1-cycle. We refer to such a 1-cycle as a short {\sf I}-cycle. Similarly, any three edges of the same color that connect a single hexavalent vertex to three trivalent vertices correspond to another 1-cycle, which we refer to as a short {\sf Y}-cycle. See Figure \ref{fig:ShortCycles}  for examples of a short {\sf I}-cycle (left) and a short {\sf Y}-cycle (right). We can also consider a sequence of edges starting and ending at trivalent vertices and passing directly through any number of hexavalent vertices, as pictured in Figure \ref{fig:longI-cycle}. Such a cycle is referred to as a long {\sf I}-cycle. Finally, we can combine any number of {\sf I}-cycles and short {\sf Y}-cycles to describe a wide class of 1-cycles as trees with leaves on trivalent vertices and edges passing directly through or branching at hexavalent vertices; we refer to such cycles as {\sf Y}-trees.

 	\begin{center}
		\begin{figure}[h!]{ \includegraphics[width=.45\textwidth]{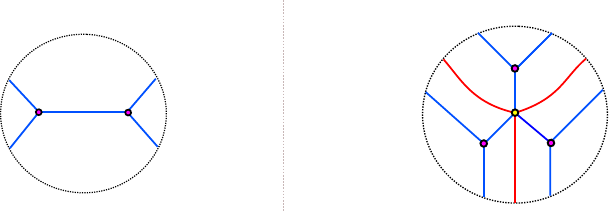}}\caption{Two local models of $\mathbb{L}$-compressing cycles in $H_1(\mathfrak{w})$.
			}
			\label{fig:ShortCycles}\end{figure}
	\end{center}

	\begin{center}
		\begin{figure}[h!]{ \includegraphics[width=.95\textwidth]{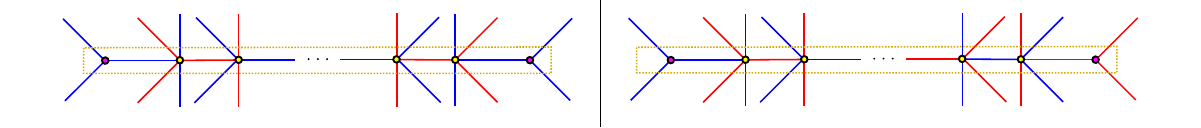}}\caption{A pair of long {\sf I}-cycles. The cycle on the left passes through an even number of hexavalent vertices, while the cycle on the right passes through an odd number. 
			}
			\label{fig:longI-cycle}\end{figure}
	\end{center}

	The	intersection form $\langle \cdot, \cdot \rangle$ on $H_1(\La(\Gamma))$ plays a key role in understanding cluster structures coming from Lagrangian fillings. For a pair of 1-cycles $\gamma_1, \gamma_2\in H_1(\La(\Gamma))$ with a fixed orientation intersecting transversely in $\La$, we refer to the signed count of the points of intersection of $\gamma_1$ and $\gamma_2$ as their algebraic intersection and denote it by $\langle \gamma_1, \gamma_2\rangle.$ 

\subsubsection{Equivalences and surgeries of $N$-graphs}
We consider Legendrian weaves to be equivalent up to Legendrian isotopy fixing the boundary. Such Legendrian isotopies can also often be combinatorially understood through $N$-graphs. We can restrict our attention to specific isotopies, pictured in Figure~\ref{fig:Moves} and refer to them collectively as Legendrian weave isotopy moves; note that they do not comprise a complete collection of local moves on Legendrian fronts corresponding to Legendrian surface isotopies. From \cite{CasalsZaslow}, we have the following theorem relating Legendrian weave isotopy moves to the corresponding $N$-graphs. 
	\begin{theorem}[\cite{CasalsZaslow}, Theorem 4.2]
		Let $\Gamma$ and $\Gamma'$ be two $N$-graphs related by one of the moves shown in Figure~\ref{fig:Moves}. The Legendrian weaves $\Lambda(\Gamma)$ and $\Lambda(\Gamma')$ are Legendrian isotopic relative to their boundaries. 
	\end{theorem}

	\begin{center}\begin{figure}[h!]{ \includegraphics[width=\textwidth]{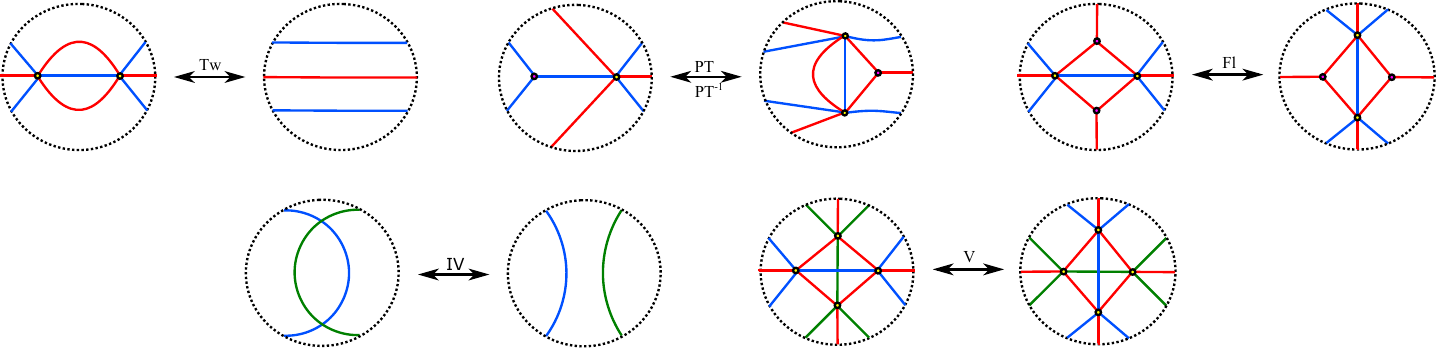}}\caption{A (not necessarily independent) collection of Legendrian weave isotopy moves for $N$-graphs. Clockwise from top left, a candy twist, a push-through, a flop, and two additional moves, denoted by Tw, PT, Fl, IV, and V respectively.} 
			\label{fig:Moves}\end{figure}
	\end{center}

 \begin{figure}[htb]{\tiny
\begin{overpic}[width=.4\textwidth,tics=10] 
{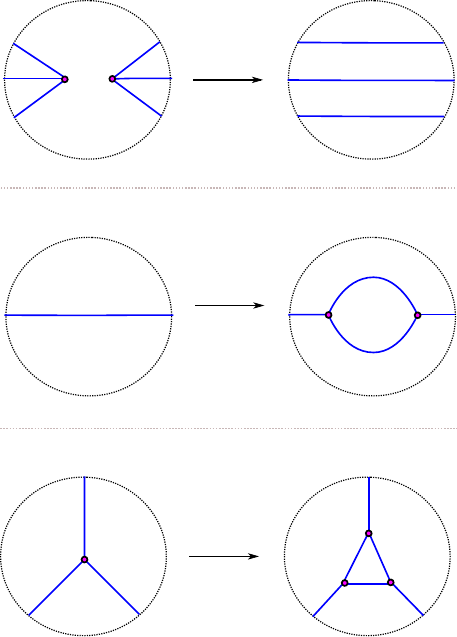}
\put(30,73){Move (1)}
\put(30,35){Move (2)}
\put(30,-3){Move (3)}
\end{overpic}}
\caption{The surgery moves in Theorem~\ref{thm:czsurgery}.}
\label{fig:surgery}
\end{figure}

We can also perform local combinatorial modifications to $N$-graphs to realize certain Legendrian surgeries.
 
\begin{theorem}[Theorem 4.10, \cite{CasalsZaslow}]\label{thm:czsurgery}
    Given two $N$-graphs, the local modifications shown in Figure~\ref{fig:surgery} correspond to the following:
    \begin{enumerate}
        \item A connect sum.
        \item Connect summing with a $T^2_{\std}$.
        \item Connect summing with a $T^2_{C}$.
    \end{enumerate}
\end{theorem}

Finally, we have one additional combinatorial move, known as Legendrian mutation, that is closely related to Lagrangian disk surgery. Given a Legendrian weave $\La$ and a short {\sf I}-cycle $\gamma\in H_1(\La;\Z)$, the Legendrian mutation $\mu_\gamma(\La)$ outputs a Legendrian weave smoothly isotopic to $\La$ but that is generally not Legendrian isotopic to $\La$. 
	
Combinatorially, we can describe Legendrian mutation in terms of the $N$-graph associated to a weave. Figure~\ref{fig:Mutations} depicts mutation at a short {\sf I}-cycle. See \cite[Section 4.9]{CasalsZaslow} for a more general description of mutation at long {\sf I}- and {\sf Y}-cycles in $N$-graphs. The geometric operation of Lagrangian disk surgery on the Lagrangian projection of $\La$ coincides with the combinatorial manipulation of the $N$-graphs. 

	\begin{center}\begin{figure}[h!]{ \includegraphics[width=.4\textwidth]{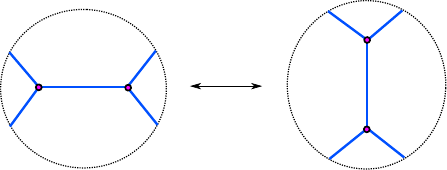}}\caption{Mutation at a short {\sf I}-cycle.}\label{fig:Mutations}\end{figure}
	\end{center}
	
In general, we can identify the effect of Legendrian weave isotopy moves and Legendrian mutation on {\sf I} and {\sf Y} cycles. See \cite[Figure 27]{ABL2021} for the behavior of specific cycles under the Legendrian weave isotopy moves.

\subsection{Doubling Legendrian weaves}
We first define a doubling operation on $N$-graphs that we then use to construct various Legendrian doubles via weaves.

\begin{definition}\label{def: doublegraph}
    Consider two properly embedded $N$-graphs $\Gamma_1 \subset \D^2$ and $\Gamma_2 \subset \D_2$, with the same labeled boundary. Then the doubled $N$-graph $\La_\Gamma(\Gamma_1, \Gamma_2)$ is defined to be the $N$-graph obtained on $S^2$ by gluing the two disks via their boundaries and identifying the labels on the boundaries of $\Gamma_1$ and $\Gamma_2$.
\end{definition}

Consider the following elementary observation. Note that we did not need to assume embedded to define the doubled Legendrian in Section~\ref{sec:double}.

\begin{proposition}\label{prop:double}
    Let $\Gamma_1$ and $\Gamma_2$ be $N$-graphs describing two exact Lagrangian fillings $\pi(\La(\Gamma_1))$ and $\pi(\La(\Gamma_2))$ of a link $\la$. Then the Legendrian $\La(\pi(\La(\Gamma_1)),\pi(\La(\Gamma_2)))$ is the Legendrian weave corresponding to the doubled $N$-graph $\La_\Gamma(\Gamma_1,\Gamma_2) \subset S^2$. 
\end{proposition}

\begin{proof}
    This is easy to observe after choosing the right $\D^2 \subset \R^5$ to satellite the weave along. Consider the Legendrian disk (call it $D_1$) whose front projection is half the Legendrian unknot (there is a unique Legendrian $\D^2$ in $\R^5_{st}$ but this helps in identifying with the setup in Section~\ref{sec:double}). Suppose the boundary of the disk is along $x_2 = -1$, and the disk extends to $\{x_2 \leq -1\}$. Similarly choose the other half of the Legendrian unknot (call it $D_2$) and translate it to live in $\{x_2 \geq 1\}$. Interpolate between them with a Legendrian annulus whose front is the cylinder in the $x_2$ direction of the front for the 1-dimensional Legendrian unknot in $(x_1, z)$. Now, considering the Legendrians $\La(\Gamma_i)$ after satelliting over $D_i$, and gluing them as in the doubling construction, is exactly the same as considering the Legendrian weave in $J^1(S^2)$ over the graph $\La_\Gamma(\Gamma_1,\Gamma_2) \subset S^2$, and satelliting it over the standard unknot in $\R^5$.
\end{proof}

For simplifying notation, henceforth we will denote the double coming from weaves by $\La(\Gamma_1,\Gamma_2)$, instead of $\La(\pi(\La(\Gamma_1)), \pi(\La(\Gamma_2)))$. In the following, we show a refinement of Theorem~\ref{thm:trivdouble} of \cite{courte2017lagrangian} for weaves. In particular, for an exact Lagrangian filling $\La(\Gamma)$ of a link given by an $N$-graph $\Gamma$, we describe exactly the symmetric double $\La(\Gamma,\Gamma)$.

\subsection{Lagrangian Surgery} \label{sec: lag_surgery}

Here we describe a mutation or surgery procedure for 3-dimensional Lagrangians in 6-space.  It can be understood as modifying the construction in \cite[Section 2.3]{Yau} by spinning in a higher dimension. 

Let $W$ be a Weinstein domain, and let $L \subset W$ be a properly embedded exact Lagrangian filling of $\La = \partial L$. Let $S$ denote an {\em attaching Lagrangian solid torus}, i.e. a Lagrangian embedding of $S^1 \times \D^2$ into $W$ such that

\begin{enumerate}
    \item $T := \partial S = S \cap L$ 
    \item $S$ is transverse to $L$ along $T$
\end{enumerate}

We will call the pair $(L, S)$ a {\em generalized solid mutation configuration} following the terminology in \cite{ChandaHirschiWang}. There exists a neighborhood $U$ of $S$ symplectomorphic to $T^*S$ via a symplectomorphism $\phi$, which identifies $S$ with the zero-section. We can split $T^*S$ as $T^*S^1 \oplus T^*D^2$. Then, one can obtain $S^1$-spun versions of all the objects in the standard model in \cite[Section 2.2]{Yau}, i.e., considering the analogous objects in $T^*D^2$ and taking their product with the 0-section in $T^*S^1$. Following the same recipe, we can surger $L$ along $S$ to obtain a Lagrangian $L'$, which is smoothly isotopic to $L$ but not necessarily Hamiltonian isotopic. The details follow. 

\begin{lemma}\label{lem: lag_surgery}
    Let $(L, S) \subset W$ be a solid mutation configuration. Then there exists a Lagrangian surgery procedure $\eta_S$ which can be applied to $L$ to produce a Lagrangian $L' \subset W$ which is smoothly isotopic to $L$, but not necessarily Hamiltonian isotopic. 
\end{lemma}

\begin{proof}
    Consider a cotangent neighborhood of $S$, which can be thought of as $T^*S^1 \oplus T^*D^2$ with coordinates $(p,q, x_1, x_2, y_1, y_2)$, such that $(p,q) \in T^*S^1$, $p \in S^1$, and $(x_1, x_2, y_1, y_2) \in T^*D^2$. Let $\pi_2$ denote the projection from $T^*S$ to $T^*D^2$. Also consider the rotation subgroup of $SU(2)$ acting on the $T^*D^2$ summand denoted $\mathcal{G}$ as in \cite[Section 2.2]{Yau} -- we extend it to $T^*S$ by being the identity on the $T^*S^1$ summand. It follows similarly as in \cite[Fact 2.2.1]{Yau}, that if $\gamma_1 \subset \{q=x_2=y_2=0\}$ and $\pi_2(\gamma_1)$ is an immersed curve in $\R_{x_1,y_1}$, the $\mathcal{G}$-orbit of $\gamma_1$, denoted $Orb_{\mathcal{G}}(\gamma_1)$, is an immersed Lagrangian in $T^*S$.
    
    We pick an open neighborhood $U \subset W$ of $S$ such that it can be identified with $T^*S$ as above, with the following properties:
    \begin{itemize}
        \item $\pi_2 (U)$ contains the closed ball $B_r$ of radius $r$ with center $0 \in \R^4$
        \item $S = \{x_1^2 + x_2^2 \leq (\sqrt{2}-1)^2r^2, q = y_1 = y_2 = 0\}$
        \item $Q_S := L \cap \pi_2^{-1}(B_r)$ = $Orb_{\mathcal{G}}(\gamma_1)$, where $\gamma_1 \subset \{q=x_2=y_2=0\}$ and $\pi_2(\gamma_1)$ is the curve $\gamma$ in Equation (3) of \cite{Yau}, i.e. $\gamma : (3\pi/4, 5\pi/4) \to \R^2_{x_1,y_1}$ given by $x_1(\gamma(s)) = \sqrt{2}r+r\cos s$ and $y_1(\gamma(s)) = r\sin s$.
     \end{itemize}

    Then, also consider the anti-symmetric linear map $M$ on $T^*S$ which acts on the $T^*S^1$ components by identity, and on the $T^*D^2$ components by interchanging $x_1$ and $y_1$, and $x_2$ and $y_2$, respectively. This action of $M$ commutes with the rotation $\mathcal{G}$. Then define $\gamma_1' := M(\gamma_1)$, and $Q_S' := Orb_{\mathcal{G}}(\gamma_1')$. Then, the surgered Lagrangian $L'$ is defined as follows:
    \[
    L' := \eta_S(L) = (L \setminus Q) \cup Q'
    \]

    Then $L'$ is Lagrangian by the facts stated above. That $L'$ is smoothly isotopic to $L$ follows by the same argument as in \cite[Section 2.3]{Yau}.
 \end{proof}

\section{Background - Sheaves and Clusters}\label{sec: sheaves_and_clusters}

In this section we review material relating to the microlocal theory of sheaves and cluster algebras. The reader is encouraged to look at \cite{KashiwaraSchapira} for a  thorough introduction to the microlocal theory of sheaves, and to refer \cite{FWZ1, FWZ2} for a detailed exposition on the theory of cluster algebras.

\subsection{Microlocal theory of sheaves}\label{sec: sheaves background}

In this section we define the moduli of microlocal rank one sheaves with singular support contained in a Legendrian submanifold. Since we restrict our attention to sheaf moduli related to Legendrian weaves, we are able to give a description of these moduli as a space of flags with specific transversality conditions that can be read off of the $N$-graph, avoiding categorical nuances that arise when considering more complicated Legendrians. We invite the reader to consult the appendices of \cite{CasalsLi} for a careful discussion of these nuances. 

 Let $\Sh(M)$ denote the derived category of complexes of sheaves of $\C$-modules on a smooth manifold $M$ with constructible cohomology. For $x\in M, \xi \in T^*_xM$ a sheaf $\mathcal{F}\in \Sh(M)$ is said to propagate along the codirection $\xi$ if for some neighborhood $U_x$ of $x$ and all smooth functions $\varphi\in C^\infty(M)$ satisfying $\varphi(x)=0$ and $d\varphi(x)=\xi$, we have
 $$\Cone(R\Gamma((\varphi < -\delta)\cap U_x, \mathcal{F})\to R\Gamma((\varphi<\delta)\cap U_x, \mathcal{F}))\cong 0.$$

  The singular support $SS(\mathcal{F})$ of a sheaf $\mathcal{F}\in \Sh(M)$ is then defined as the closure of the set of points $(x, \xi)$ in $T^*M$ along which the sheaf $\mathcal{F}$ does not propagate. In general, $SS(\mathcal{F})$ is a conic Lagrangian in $T^*M$, which allows us to quotient by the $\R^+$ action to obtain a Legendrian submanifold of the unit cotangent bundle $T^\infty(M)$. For $M\cong C\times \R$, we identify the unit cotangent bundle $T^{\infty,-}(M)$ with the first jet space $J^1C$.

  Given a specific Legendrian $\La\in J^1C$, we can ask the question ``for what sheaves $\mathcal{F}\in \Sh(C\times \R)$ do we have $SS(\mathcal{F})\subseteq \La$?'' Such sheaves are necessarily constructible with respect to a stratification of $M$ induced by a front projection of $\La$. 
  We also require that the sheaves we consider have acyclic stalks in a neighborhood of $\C\times \{-\infty\}$.
  By \cite[Theorem 5.3]{CasalsZaslow}, we may simplify the category we work in to consider only sheaves of vector spaces concentrated in degree 0 that are constructible with respect to the induced stratification. This simplification allows us to answer the moduli space question above by understanding a collection of vector spaces assigned to strata and maps between them induced by restriction of sheaves to different strata. 

 We now introduce the concept of microlocal rank of a sheaf $\mathcal{F}$ with $SS(\mathcal{F}) \subseteq \La$. Away from singularities of a front projection, a point $x\in \La$ has a neighborhood with front projection that resembles a smooth hypersurface $A$. In a neighborhood $A$, denote the stalk below this hypersurface by $\mathcal{F}_d$ and the stalk above the hypersurface by $\mathcal{F}_u$. Define the microstalk of $\mathcal{F}$ at a point $x \in \Pi(A)$ to be $\Cone(\mathcal{F}_d\to \mathcal{F}_u)$. We say that $\mathcal{F}$ is of microlocal rank $r$ if the microstalk of $\mathcal{F}$ is a vector space of rank $r$ concentrated in a single degree.

Collecting all of the conditions described above, we arrive at the following definition.

\begin{definition}
    We define $\mathcal{M}_1(\La)$ to be the space of microlocal rank one sheaves $\mathcal{F}\subseteq Sh(C\times \R)$ with singular support contained in $\La$ that have stalk 0 in a neighborhood of $C\times \{-\infty\}$.
\end{definition}

\begin{theorem}[Theorem 3.7, \cite{GKS_Quantization}]\label{thm: quantization} The space $\mathcal{M}_1(\La)$ is a Legendrian isotopy invariant of $\La$.
\end{theorem}

See also \cite[Appendix B]{CasalsLi} for a treatment of this result more closely following our exposition. In particular, following the discussion in the cited appendix, invariance of $\mathcal{M}_1(\La)$ follows from the equivalence of sheaf categories given by \cite[Theorem 3.7]{GKS_Quantization} together with \cite[Remark B.18 and Definition B.24]{CasalsLi}.

\subsubsection{Flag moduli space}

For a Legendrian link $\la=\la(\beta\d)\in J^1S^1$ with $\beta \in Br_n^+$, the data of $\mathcal{M}_1(\la)$ is given by a choice of vector space for each open region of $\S^1\times \R_z \backslash \Pi(\la)$. As we cross arcs of $\Pi(\la)$ in the positive $z$ direction, the microlocal rank one condition requires that we have a nested sequence of vector spaces $V^0\subseteq V^1 \dots \subseteq V^n$ with $\dim(V^i) = i$, i.e., a flag $V^\bullet$. Careful computation of the singular support conditions induced by a crossing of $\Pi(\la)$ yields a transversality condition of the flags $V_1^\bullet$ and $V_2^\bullet$ on either side of the crossing. In particular, a crossing between strands $i$ and $i+1$ corresponding to a braid group generator $\sigma_i$ requires that $V_1^i\neq V_2^i$ and $V_1^j=V_2^j$ for all $j\neq i$ \cite[Proposition 6.2]{STZ_ConstrSheaves}.

For a Legendrian weave, $\La(\Gamma)\subseteq J^1D^2$ with underlying $N$-graph $\Gamma\subseteq D^2$, the data of $\mathcal{M}_1(\La(\Gamma))$ is given by the assignment of a flag $V^\bullet$ to each open region of $D^2\times \R\backslash \Pi(\La(\Gamma))$ and transversality conditions imposed by edges. In particular an edge labeled by $\sigma_i$ imposes the same transversality condition as at a crossing of a Legendrian link. This can be seen by recognizing such an edge as representing a local picture of the product of a crossing with an interval. Local computations at trivalent, tetravalent, and hexavalent vertices show that no additional singular support conditions are imposed \cite[Theorem 5.3]{CasalsZaslow}; see also \cite[Section 4.1]{TreumannZaslow} for a computation of the singular support conditions at a trivalent vertex. As a result, the data of $\SM_1(\La(\Gamma))$ is equivalent to providing the following: \begin{enumerate}
		\item[(i)] An assignment to each face $F$ (connected component of $\D^2\backslash \Gamma$) of a flag $V^\bullet(F)$ in the vector space $\C^N$.
		\item[(ii)] For each pair $F_1, F_2$ of adjacent faces sharing an edge labeled by $\sigma_i$, we require that the corresponding flags satisfy 
		$$V^j(F_1)=V^j(F_2), \qquad 0\leq j\leq N, j\neq i, \qquad \text{ and } \qquad V^i(F_1)\neq V^i(F_2).$$  
	\end{enumerate}
	Finally, we consider the moduli space of flags satisfying (i) and (ii) modulo the diagonal action of $\GL_N(\C)$ on $V^\bullet$. The precise statement \cite[Theorem 5.3]{CasalsZaslow} we require is that the flag moduli space, associated to $\Gamma$ is isomorphic to the moduli of microlocal rank-one sheaves $\SM_1(\La(\Gamma))$. Since $\SM_1(\La(\Gamma))$ is an invariant of $\Lambda(\Gamma)$ up to Hamiltonian isotopy, it follows that the flag moduli is invariant as well.

\subsubsection{The decorated sheaf moduli}\label{sub: decorated sheaf moduli}

In preparation for describing cluster $\mathcal{A}$ spaces in the next section, we introduce the decorated sheaf moduli $\FM(\la, T)$. To do so, we must first specify a trivialization of the abelian local system given by the microlocal monodromy about each component of $\la$. Let $T=\{t_1, \dots, t_k\}$ be a set of marked points on $\la$ with every component of $\la$ carrying at least one marked point. Label the connected components of $\la\backslash T$ by the pair of endpoints of each segment $(t_i, t_{i+1})$ where indices are taken modulo the number of marked points on the appropriate component. The decorated moduli stack $\FM(\la, T)$ is then given by the additional data of a trivialization of the local system $m_\la(F)$ obtained by applying the  microlocalization functor $m_\la$ (see \cite[Section B.2]{CasalsLi}) to a sheaf $F\in \SM_1(\la)$ for every segment of each component of $\la$. 

\begin{equation*}
    \FM(\la, T):=\{(F, \phi_1, \dots, \phi_k) | F\in \SM_1(\la), \phi_i \text{ is a trivialization of } \m_\la(F) \text{ on } (t_i, t_{i+1})\}
\end{equation*}

Two framings are equivalent if they differ by a global factor of $\C^\ast$. As in \cite[Section 2.8]{CasalsWeng} we define the moduli space $\SM_1(\la, T)$ analogously by the addition of the trivialization data associated with the marked points. In this work, we suppress $T$ from our notation when our discussion does not depend on the presence of marked points.

The choice of these trivializations has two primary motivations. First, it allows us to describe $\FM(\la, T)$ as a smooth affine scheme rather than an Artin stack. Second, the trivializations give us the necessary data for defining microlocal merodromies as a microlocal version of parallel transport, yielding regular functions on $\FM(\la, T)$, as we now describe.

Consider a basis $\{\gamma_i\}_{i=1}^n$ of $H_1(L, T)$ containing a maximal linearly independent subset of $\mathbb{L}$-compressing cycles of $L$; see Definition~\ref{def: L-compressing cycle} for the definition of an $\L$-compressing cycle. We then identify the lattice $H_1(L, T)$ with an isomorphic lattice $H_1(L\backslash T, \la\backslash T)$, and consider the dual basis of cycles $\{\gamma_i^\vee\}_{i=1}^n$.

Given an oriented relative cycle $\gamma^\vee\in H_1(L\backslash T, \la\backslash T)$ starting at $s$ and ending at $t$ we first note that the framing data of $\FM(\la, T)$ specifies two vectors $\phi_{s} \in \Phi_{s}$ and $\phi_{t}\in \Phi_{t}$ at $s$ and $t$. The result of parallel transport along $\gamma^\vee$ yields a nonzero vector $\gamma^\vee(\phi_{s})\in \Phi_{t}$.\footnote{For more details on the particulars of signs and spin structures in this computation, we refer the interested reader to \cite[Section 4.5]{CasalsWeng}.} 

\begin{definition}\label{def: merodromy}
    Given an $\mathbb{L}$ compressing cycle $\gamma\in H_1(L, T)$, the microlocal merodromy along its dual relative cycle $\gamma^\vee\in H_1(L\backslash T, \la\backslash T)$ is the function $A_{\gamma^\vee}=\gamma^\vee(\phi_{s})/\phi_{t}$.

\end{definition}

One can compute $A_{\gamma^\vee}$ explicitly from a Legendrian weave $\La(\Gamma)$ by lifting the framing data to a set of decorations. Given a flag $V^\bullet$ with framing data $\phi_i=V^i/V^{i-1}$, we can construct a volume form $\alpha_i\in \bigwedge^i V^i$ by first lifting $\phi_i$ to a nonzero vector $\tilde{\phi}_i\in V_i$ and then setting $\alpha_i=\tilde{\phi}_i\wedge \dots \wedge \tilde{\phi}_1.$ At an edge of $\Gamma$ labeled by $\sigma_i$, we have flags $\mathcal{L}^\bullet$ and $\mathcal{R}^\bullet$ to the left and right of the $\sigma_i$ edge with framing data $\{\la_i\}$ and $\{\rho_i\}$, respectively. Denote by $\{\alpha_i\}$ and $\{\beta_i\}$ the decorations corresponding to the framings on $\mathcal{L}^\bullet$ and $\mathcal{R}^\bullet$. The parallel transport of $\la_i$ along an oriented curve $\eta$ from the $i$th sheet on the left to the $i+1$st sheet on the right can then be computed by 
\begin{equation}\label{eq: merodromy1}\eta(\la_i)=\left(\frac{\tilde{\la}_i\wedge \tilde{\rho}_{i}\wedge \alpha_{i-1}}{\tilde{\rho}_{i+1}\wedge\beta_i}\right)\rho_{i+1}.\end{equation}
Similarly, the parallel transport of $\la_{i+1}$ along the oriented curve $\eta'$ from the $i+1$st sheet on the left to the $i$th sheet on the right yields 

\begin{equation}\label{eq: merodromy2}\eta(\la_{i+1})=\left(\frac{\alpha_{i+1}}{\tilde{\rho}_i\wedge \tilde{\la}_{i}\wedge \alpha_{i-1}}\right)\rho_{i}.\end{equation}

Composing Equations~\ref{eq: merodromy1} and \ref{eq: merodromy2}, allows us to compute the microlocal merodromy along any oriented curve in $\La(\Gamma)$. By \cite[Proposition 4.29]{CasalsWeng}, the microlocal merodromy $A_{\gamma^\vee}$ along a relative cycle $\gamma^\vee \in H_1(\La(\Gamma)\backslash T, \la\backslash T)$ dual to an $\mathbb{L}$-compressing cycle $\gamma$ is a regular function on $\FM(\la, T)$. As we explain in the following section, this regular function is actually a cluster variable, allowing us to define a cluster-$\mathcal{A}$ structure on $\FM(\la, T)$.

\subsection{Cluster algebras and cluster ensembles}

In this subsection, we give the necessary background on cluster theory. We start by briefly defining cluster algebras and folding; see \cite{FWZ1} and \cite{FWZ2} for more details. Following \cite[Section 2]{GHK15}, 
we then give a description of the notion of a cluster ensemble Finally, we explain the necessary ingredients to understand how this structure arises from the pair of sheaf moduli $\FM(\la, T)$ and $\SM_1(\la, T)$. 

\subsubsection{Cluster algebras}

We start with a quiver, the combinatorial input used to define a cluster algebra. A quiver $Q$ is a directed graph without loops or directed 2-cycles. We designate a subset of vertices $Q^{mut}$ to be the mutable part of the quiver. The vertices in $Q\backslash Q^{mut}$ are designated as frozen and we require that there are no arrows between them. To a quiver $Q_0$ with $n$ vertices and no arrows between frozens, we associate an initial set of variables $a_1, \dots a_n$, one for each vertex. Together, the $n$-tuple ${\bf a}=(a_1, \dots, a_n)$ and the quiver $Q_0$ form a cluster seed $({\bf a}, Q_0)$. For a cluster seed $({\bf a}, Q)$ we can denote by $b_{ij}$ the multiplicity of arrows in $Q$ from vertex $i$ to vertex $j$. We then obtain a skew-symmetric matrix, $B=(b_{ij})$, known as the exchange matrix, encoding the information of the quiver.
There are two types of cluster algebras, type $\mathbb{A}$ or type $\mathbb{X}$, depending on the precise form of the mutation formula relating different cluster variables.

\begin{definition}
    Let $({\bf a}, Q)$ be a cluster seed and $k\in Q^{mut}$ be a mutable vertex. The cluster $\mathbb{A}$ seed mutation $\mu_k$ is an operation taking as input the seed $({\bf a}, Q)$ and outputting the new seed $({\bf a}', Q')$ where $Q'$ is related to $Q$ by quiver mutation at vertex $k$ and ${\bf a}'$ is related to ${\bf a}$ by $a_i'=a_i$ for all $i\neq k$ and 
  
   \begin{equation}\label{eq: A-mutation}
       a_ka_k'=\prod_{b_{ik}>0} a_i^{b_{ik}} + \prod_{b_{ik<0}} a_i^{-b_{ik}}.
   \end{equation}

\end{definition}

Note that seed mutation is an involution, so that $\mu_k^2({\bf a}, Q)=({\bf a}, Q)$.

Denote by $\mathcal{F}$ the field of rational functions $\C(a_1, \dots a_n)$ and consider an initial seed $({\bf a}, Q_0)\subseteq \mathcal{F}$. 
\begin{definition}
The type $\mathbb{A}$ cluster algebra generated by $({\bf a}, Q_0)$ is the $\C$-algebra generated by all cluster variables arising in arbitrary mutations of the initial seed.     
\end{definition}

The type $\mathbb{X}$ cluster algebra is generated from an initial seed $({\bf x}, Q_0)$ by the mutation formula 

  \begin{equation}\label{eq: X-mutation}
  x_j'=\begin{cases}
    x_j^{-1} & i=j\\
    x_j(y_k+1)^{-b_{kj}} & j\neq k, b_{kj}\leq 0 \\
    x_j(x_k^{-1}+1)^{-b_{kj}} & j \neq k, b_{kj}\geq 0
\end{cases} \end{equation}

\subsubsection{Folding cluster algebras}\label{sub: folding}

Given a cluster algebra $\SA$ and a finite group $G$ acting on it in a way that respects mutation, one can form a new cluster algebra by quotienting by this $G$-action via a process called folding, as we discuss below; see \cite[Section 4.4]{FWZ2} for more details.    

Let $G$ be a finite group acting on a cluster algebra $\SA$ and suppose that there is a seed $({\bf a}, Q)$ of $\mathcal{A}$ such that $g\cdot Q=Q$ for all $g \in G$. We denote by $i\sim i'$ the vertex label $i'$ of any vertex $v_{i'}$ lying in the $G$-orbit of $v_i$ and recall that $\tilde{B}(Q)=(b_{ij})$ denotes the exchange matrix
of $Q$. We define the following compatibility condition for $Q$ with respect to the $G$-action:

\begin{definition}
    The quiver $Q$ is $G$-admissible if the following holds:
    \begin{enumerate}
      \item If $i\sim i'$, then $i$ and $i'$ are either both mutable or both frozen.
      \item For all $i, j$ and any $g\in G$, we have $b_{ij}=b_{g\cdot i, g\cdot  j}$.
      \item  For all $i \sim i'$, we have $b_{i i'}=0$.
      \item For all  $i \sim i'$, and any mutable $j$ we have $b_{ij}b_{i'j}\geq 0$.
    \end{enumerate}
\end{definition}

$G$-admissibility ensures that the result of mutating a quiver at each of the vertices in a $G$-orbit does not depend on the order of mutations. 

Let $I$ be the $G$-orbit of a vertex $v_i$. We denote by $\#\{I \rightarrow J\}$ the sum $\sum_{i\in I} \#\{i\rightarrow j\}$ for some arbitrary $j\in J$. To package the data $\#\{I\rightarrow J\}$ as part of a graph, we produce a weighted quiver. That is, given a $G$-admissible quiver $Q$, we consider the weighted graph $Q^G$ whose vertices are $G$-orbits of vertices in $Q$ and with an edge from $v_I$ to $v_J$  if $b_{ij}>0$ for any $i\in I, j \in J$; this choice of direction is well-defined because $Q$ is $G$-admissible. 
We label the edges with the pair of natural numbers $\#\{I \rightarrow J\}$ and $\#\{I \rightarrow J\}$. Note that $\#\{I \rightarrow J\}\neq -\#\{J\rightarrow I\}$ in general; see Figure~\ref{fig: FoldingExample} for an example of a weighted quiver obtained from folding a $\Z_2$-admissible quiver. Alternatively, we can define a matrix $\tilde{B}^G$ with entries $b_{ij}^G=\sum_{i\in I} b_{ij}$ where $b_{ij}$ is the $(i,j)$ entry of the original exchange matrix $B$. The mutable part of the exchange matrix $\tilde{B}^G$ is skew-symmetrizable, i.e. there is a diagonal matrix $D$ with positive integer entries such that $B^GD$ is skew-symmetric. Here the $j$th nonzero entry of $D$ is given by the size of the orbit of the quiver vertex $v_j$. 

Given a $G$-admissible quiver $Q$ and a mutable $G$-orbit $I$, we can consider a sequence of mutations $\mu_I=\prod_{i\in I}\mu_i$. When both $Q$ and $\mu_I(Q)$ are $G$-admissible, we have that $\mu_I(Q^G)=\mu_I(Q)^G$ where mutation of the folded quiver $Q^G$ is defined as above, replacing $b_{ij}$ by $b_{ij}^G$. As a sufficient condition for defining a cluster algebra from $Q^G$, we can ask that we always have $\mu_I(Q^G)=\mu_I(Q)^G$ for any sequence of mutations.

\begin{definition}\label{def: globally foldable}
    A quiver $Q$ is globally foldable with respect to a given $G$-action if $Q$ is $G$-admissible and, for any sequence of mutable $G$-orbits $I_1, \dots I_k$, the quiver $\mu_{I_k}\dots \mu_{I_1}(Q)$ is $G$-admissible as well. 
\end{definition}

See \cite[Example 4.4.6]{FWZ2} for an example of a quiver that is $G$-admissible but is not globally foldable with respect to a specified $G$-action.

Let $({\bf a}, Q)$ be a globally foldable seed with respect to a $G$-action. Recall that for ${\bf a}=(a_1, \dots, a_m)$, we denote by $\mathcal{F}$ the field of rational functions in $m$ variables. Let $\mathcal{F}^G$ be the field of rational functions in $m_G$ variables, where $m_G$ denotes the number of $G$-orbits of vertices $v\in Q$, equivalently the number of vertices in $Q^G$. We define the folded seed obtained from $({\bf a}, Q)$ to be the variables $a_{I_1}, \dots a_{I_{m_{g}}}$ together with the weighted quiver $Q^G$. The key consequence of the globally foldable condition is that it allows us to mutate at any mutable element in a folded seed to obtain another folded seed. The collection of cluster variables given by arbitrary mutations of the folded seed then defines the folded cluster algebra.

\begin{definition}
    Given a cluster algebra $\mathbb{A}$ with initial seed $({\bf a}, Q)$ that is globally foldable with respect to a group $G$, the folded cluster algebra $\mathbb{A}^G$ is the $\C$-algebra generated by the cluster variables arising from arbitrary mutations of the seed $((a_{I_1}, \dots, a_{I_{m_g}}), Q^G)$. 
\end{definition}

\subsubsection{Cluster ensembles}


Following the exposition given in \cite[Section 2]{GHK15} for cluster ensembles of skew-symmetrizable cluster algebras, we define a collection of lattices and a bilinear form comprising the fixed data of the cluster ensemble. We first consider a lattice $N$ with a $\Q$-valued skew-symmetric bilinear form $\langle \cdot , \cdot \rangle$. The lattice $N$ contains an unfrozen sublattice $N_{\uf}$, which is a saturated sublattice of $N$. For skew-symmetrizable cluster structures, there is an additional choice of sublattice $N^\circ$ of $N$ that is of finite-index in $N$ that satisfies the condition that $\langle N_{uf}, N^\circ\rangle$ and $\langle N, N_{uf}\cap N^\circ\rangle$ are both integer valued. We also obtain dual lattices $M:= \Hom(N, \Z)$ and $M^\circ :=\Hom(N^\circ, \Z)$ where $M$ is necessarily a finite-index sublattice of $M^\circ$. The lattice $N$ is the character lattice for the $\SX$ cluster tori, while the lattice $M^\circ$ gives the character lattice for the $\SA$ cluster tori. The toric coordinates obtained from the basis vectors of the corresponding lattices yield cluster variables of the appropriate cluster charts. Finally, the data of the exchange matrix is encoded in a (generally non-skew-symmetric) bilinear form as follows. Let $e_1, \dots, e_r$ form a basis for $N$ -- these basis vectors comprise what \cite{GHK15} refer to as the seed data -- and define $d_1, \dots, d_r$, $d_i\in \N$ by the condition that $d_1e_1, \dots d_re_r$ form a basis for $N^\circ$. Then we can obtain a bilinear form $[ \cdot, \cdot]$ from $\langle \cdot, \cdot \rangle $ by the formula $[e_i, e_j]=\langle e_i, e_j\rangle d_j$. In the cluster algebraic setup, the bilinear form $[ \cdot, \cdot]$ plays the role of the exchange matrix. 

Given these initial data, we construct cluster $\mathcal{A}$ and cluster $\mathcal{X}$ varieties by gluing algebraic tori together via the birational map induced by the mutation formulas given in Equations \eqref{eq: A-mutation} and \eqref{eq: X-mutation}. The initial data used to construct these $\mathcal{A}$ and $\mathcal{X}$ varieties together define a cluster ensemble. The coordinate rings $\C[\mathcal{A}]$ and $\C[\mathcal{X}]$ are then (upper) cluster algebras of types $\mathbb{A}$ and $\mathbb{X}$, respectively.

An abbreviated statement of the main theorem of Casals and Weng tells us that for families of Legendrians $\la$ arising from a specific combinatorial construction, the moduli $\SM_1(\la, T)$ and $\FM(\la, T)$ form a cluster ensemble. The class of Legendrians they consider arises from a combinatorial construction known as a (complete) grid plabic graph and includes all braid positive Legendrians. The class of Legendrians $\la$ for which $\C[\FM(\la, T)]$ is known to be isomorphic to a cluster algebra has since been expanded to Legendrians of the form $\la(\beta\Delta)$ with $\delta(\beta)=w_0$ by \cite[Theorem 1.1]{CGGLSS}. 
A more precise summary of the main result of Casals and Weng is as follows:

\begin{theorem}[Theorem 1.1 \cite{CasalsWeng}]
    For $\la$ a Legendrian arising from a complete grid plabic graph, the decorated sheaf moduli $\FM(\la, T)$ admits a cluster $\mathcal{A}$ structure. Moreover, there is an explicitly constructed Legendrian 
weave filling $L$ of $\la$ with intersection quiver and sheaf moduli $\FM(L)$ giving the data of the initial seed.
 \end{theorem}

Casals and Weng obtain mutable cluster $\mathcal{A}$ coordinates by computing microlocal merodromies along relative homology cycles of $L$ dual to $\mathbb{L}$-compressing cycles \cite[Section 4]{CasalsWeng}. They also obtain mutable cluster $\mathcal{X}$ variables of $\SM_1(\la)$ as microlocal monodromies about these $\mathbb{L}$-compressing cycles. These two homology lattices $H_1(L, \mathfrak{t})$ and $H_1(L\backslash \mathfrak{t}, \la \backslash \mathfrak{t})$ yield the dual lattices appearing in the definition of a cluster ensemble. Legendrian mutation of an $\mathbb{L}$-compressing cycle induces a cluster-$\mathcal{X}$ mutation on toric charts in $\SM_1(\la, T)$, while Legendrian mutation of its dual induces a cluster-$\mathcal{A}$ mutation on toric charts in $\FM(\la, T)$. Frozen variables correspond to either marked points or to homology cycles in a Legendrian weave filling $L$ that do not bound embedded Lagrangian disks in the complement $\D^4\backslash L$.

\begin{remark}
    The cluster structures of \cite{CasalsWeng} are defined up to quasi-cluster equivalence. In order to simplify exposition, we omit any consideration of quasi-cluster equivalences and refer the interested reader to \cite[Appendix A]{CasalsWeng} for further details.
\end{remark}

\section{Legendrian doubles}\label{sec: legendrian_doubles}

In this section, we discuss our results related to Legendrian doubles. We start by discussing the sheaf-theoretic invariants of doubles and use that to derive more general statements about doubles and their (lack of) exact Lagrangian fillings. We then restrict our attention to doubles coming from Legendrian weaves and discuss Legendrian isotopy characterizations for certain families of doubles.

\subsection{Sheaf invariants of doubles}

    Let us now describe the sheaf moduli of Legendrian doubles and related polynomial point counts. We start by first describing the general categorical framework following \cite{CasalsLi} and \cite{li2023lagrangian}. 
    Denote by $\text{mod}(\mathbbm{k})$ the dg-derived category of chain complexes of $\mathbbm k$-modules for $\mathbbm{k}$ a field of characteristic zero. Denote by $\Sh_\La(M)_0$ the dg-derived category of sheaves of $\text{mod}(\mathbbm{k})$ with singular support in $\La\subseteq J^1(M)$ and acyclic stalks at $-\infty$. See \cite[Appendix A]{CasalsLi} for how this category differs from the category of sheaves discussed in Section~\ref{sec: sheaves background} above. Let $\la\subset J^1(\R)$ be a Legendrian link and $L_1$ and $L_2$ be two exact Lagrangian fillings of $\la$. For the double $\La(L_1, L_2)$, the category $\Sh_{\La(L_1, L_2)}(S^2)_0$ can be computed as a homotopy pullback via the following diagram, as in \cite[Section 6.2]{li2023lagrangian}.

\begin{center}
\begin{tikzcd}
\Sh_{\La(L_1, L_2)}(S^2)_{0} \arrow[r] \arrow[d]
& Loc(L_1) \arrow[d] \\
Loc(L_2) \arrow[r]
& || \Sh_{\la}(\R)_0
\end{tikzcd}
\end{center}

As an immediate corollary of this formula, one can produce a criteria for showing that a Legendrian double $\La(L_1, L_2)$ is non-loose.
\begin{theorem}[Section 6.2 \cite{li2023lagrangian}]\label{thm: non-loose}
Given two fillings $L_1$, $L_2$ of $\la$ such that the essential images of $Loc(L_1)$ and $Loc(L_2)$ intersect nontrivially in $Sh_\la(M\times \R)_0$, then $\La(L_1, L_2)$ is a non-loose Legendrian.
\end{theorem}

    If we consider the full dg-subcategory of compact objects $\Sh_{\La(L_1, L_2)}(S^2)$, then following the exposition in \cite[Section 5.1]{CasalsLi}, there is a moduli of objects, which we can denote $\SM(\La(L_1, L_2))$, and similarly with $\SM(\la)$ for $\la\cong \partial L_1\cong \partial L_2$. The fillings $L_1$ and $L_2$ then give embeddings of toric charts $\mathcal{T}_{L_i} \cong (\C^\ast)^{b_1(L_i)}\xhookrightarrow{} \SM(\la)$ and the homotopy pullback formula from \cite{li2023lagrangian} reduces to the following:
  
    \begin{lemma}\label{lem: Double Sheaves}
        The sheaf moduli $\mathcal{M}_1(\La(L_1, L_2))$ is given by the intersection $\mathcal{T}_{L_1}\cap \mathcal{T}_{L_2}$.
    \end{lemma}

Alternatively, one can obtain Lemma~\ref{lem: Double Sheaves} directly by observing that any sheaf in $\SM_1(\La(L_1, L_2))$ must satisfy the singular support conditions imposed by both (the front projections of the Legendrian lifts of) $L_1$ and $L_2$.

We restate Theorem~\ref{thm: intro_not-loose-fillable} and prove the following obstruction to exact fillings for Legendrian doubles arising from fillings that induce distinct toric charts.

\begin{theorem}\label{thm: not filllable}
   Let $L_1, L_2$ be exact, oriented, embedded Lagrangian fillings of $\la$ satisfying that $L_1$ and $L_2$ induce distinct toric charts on $\mathcal{M}_1(\lambda)$. Then the asymmetric Legendrian double $\La(L_1, L_2)$ does not admit any embedded exact Lagrangian fillings.
\end{theorem}

We prove Theorem~\ref{thm: not filllable} with the following lemma restricting possible embeddings of one algebraic torus into another.

\begin{lemma}\label{lem: Daping's Lemma}
    If $\mathcal{T}_1$ and $\mathcal{T}_2$ are two algebraic tori of the same rank, then any open embedding of $\mathcal{T}_1$ into $\mathcal{T}_2$ is an isomorphism.
\end{lemma}

The following proof was explained to us by D. Weng, to whom we are very grateful.

\begin{proof}
    Let $A_1$ and $A_2$ be the coordinate rings of $\mathcal{T}_1$ and $\mathcal{T}_2$, respectively. Then $A_1$ and $A_2$ are both Laurent polynomial rings, hence integral domains. An open embedding $\varphi$ of $\mathcal{T}_1$ into $\mathcal{T}_2$ induces an injective algebra homomorphism $\varphi^*$ of $A_2$ into $A_1$ which sends units to units. Under $\varphi^*$, the character lattice 
    $L_2$ of $A_2$ is sent to the character lattice $L_1$ of $A_1$ and $\varphi^*$ is a linear map of lattices. Since $L_1$ and $L_2$ are of the same rank, the image of $L_2$ must be of finite index inside of $L_1$. This index is precisely determined by the number of points in the fiber $\varphi: \mathcal{T}_1 \to \mathcal{T}_2$. Since $\varphi$ is an open inclusion, this index must be one, so that $\varphi^*$ is a linear isomorphism of character lattices. Hence, $\varphi$ is an isomorphism of varieties.
\end{proof}

\begin{proof}[Proof of Theorem~\ref{thm: not filllable}]
     Note that $b_1(L_1) = b_1(L_2)$ by \cite{Chantraine10}. By Lemma~\ref{lem: Double Sheaves} and \cite{JinTreumann17}, any embedded exact Lagrangian filling $\mathcal{L}$ of $\La(L_1, L_2)$ induces an embedding of $(\C^\ast)^{b_1(\mathcal{L})}\hookrightarrow \SM_1(\La(L_1, L_2))=\mathcal{T}_1\cap \mathcal{T}_2$ where $\mathcal{T}_1$ and $\mathcal{T}_2$ are the toric charts $(\C^\ast)^{b_1(L_1)}\cong (\C^\ast)^{b_1(L_2)}$ induced by the fillings $L_1$ and $L_2$ in $\SM_1(\la)$. Note that $b_1(\mathcal{L})\geq b_1(L_1)=b_1(L_2)$ by the half lives, half dies theorem. It follows that $b_1(\mathcal{L})=b_1(L_1)=b_1(L_2)$. By Lemma~\ref{lem: Daping's Lemma}, this embedding of algebraic tori of the same rank can only be an isomorphism. Therefore, when $L_1$ and $L_2$ inhabit distinct toric charts, $\La(L_1, L_2)$ does not admit any embedded exact Lagrangian fillings.   
\end{proof}

Note that the conjectural classification of exact Lagrangian fillings of braid positive Legendrians can be phrased as a one-to-one correspondence between toric charts and exact Lagrangian fillings \cite[Conjecture 5.1]{CasalsLagSkel}. This conjecture would imply that Theorem~\ref{thm: not filllable} holds for $L_1\not \cong L_2$. When $\la = \partial L_1 = \partial L_2$ is of the form $\la(\beta\Delta)$, every known pair of distinct fillings $L_1\not\cong L_2$ induce distinct toric charts in $\mathcal{M}_1(\partial L_1)$.

\subsection{Legendrian doubles from weaves}\label{sub: weave doubles}

Let $\la=\la(\beta\d)$ be an $n$-component Legendrian link in $(\R^3, \xi_{st})$ given as the (-1)-closure of the positive braid $\beta\d$ with the Demazure product of $\beta$ is equal to $w_0$.

Let $L$ be an embedded exact Lagrangian filling of $\la$ 
given as the Lagrangian projection of a Legendrian weave with boundary $\la$. We denote by $b_1(L)$ the rank of $H_1(L)$. We now restate and prove Theorem~\ref{thm: intro_Symmetric}.

\begin{theorem}\label{thm: Symmetric}
The symmetric double $\La(L, L)$ is Legendrian isotopic to $\#^{b_1(L)} \T^2_{\std}$.
\end{theorem}

\begin{proof}[Proof of Theorem~\ref{thm: Symmetric}]
Let $\la=\la(\beta\Delta)$ and $L=\w_L$ be a Legendrian weave filling of $\la$. Denote the symmetric Legendrian weave double by $\mathfrak{w}(L, L)\subseteq \R^5$. Consider the corresponding $N-$graph $\Gamma(L, L)\subseteq \R^2$ with coordinates $(x_1, x_2)$ and a height function $h(x_2)$. 
As in Section~\ref{sec:double}, we choose coordinates so that $\mathfrak{w}(L, L)|_{x_2<0}$ is Legendrian isotopic to $\mathfrak{w}(L)$ and perform a Legendrian isotopy so that $\mathfrak{w}(L, L)$ is symmetric under reflection through the $x_1-$axis. Perturb $\Gamma(L, L)$ by a planar isotopy (lifting to a Legendrian isotopy of $\mathfrak{w}(L, L)$) so that no two singularities (vertices or crossings) of $\Gamma(L, L)$ lie along the same level set of $h$. Starting at $x_2=0$ and scanning in the negative $x_2$ direction, consider the first singularity of $L$ that appears at value $x_2=t_0$. By the symmetry of $\Gamma(L, L)$, there is an identical singularity contained in the level set $h(|t_0|).$ We consider the three possible cases:

\begin{itemize}
    \item If $v$ and $v'$ are trivalent vertices at $h(t_0)$ and $h(|t_0|)$ respectively, then $\Gamma(L, L)$ forms a bigon with vertices $v$ and $v'$.
    \item If $v$ and $v'$ are hexavalent vertices at $h(t_0)$ and $h(|t_0|)$ respectively, then $\Gamma(L, L)$ forms a candy twist with vertices $v$ and $v'$.
    \item If a strand labeled $\sigma_i$ crosses a strand labeled $\sigma_j$  at height $h(t_0)$ and $h(|t_0|)$ with $|i-j|\geq 2$, then the restriction of  $\Gamma(L, L)$ to $\{t_0 \leq x_2\leq |t_0|\}$ contains a local picture as in Move IV (left) of Figure~\ref{fig:Moves}. 
\end{itemize}

By Theorem~\ref{thm:czsurgery}, removing the bigon corresponds to a 0-surgery of the surface (i.e. topologically, cutting out an $S^1 \times I$ and gluing in $S^0 \times \partial I$). Assuming that the result is a connected Legendrian surface, this yields an additional standard torus in our connect sum decomposition of $\mathfrak{w}(L, L)$. Undoing the candy twist and performing 
Move IV of Figure~\ref{fig:Moves} correspond to isotopies of the weave surface. This allows us to proceed inductively, traversing the weave from $x_2=0$ to $x_2=-\infty$, removing singularities of $\Gamma(L, L)$ as we go.

Since $\mathfrak{w}_L,$ is an embedded weave filling of $\la(\beta),$ the end result of this inductive process is the braid $\Delta^2$ forming concentric circles. Following \cite[Example 6.2]{CasalsZaslow}, this is a standard Legendrian unlink of 2-spheres with $N$ components. 

	We verify that the number of standard tori obtained in this connect sum decomposition is precisely $b_1(L)$ via the following computation. For an $N$-graph, $\Gamma\subseteq S^2$, the genus of $\Lambda(\Gamma)\subseteq J^1(S^2)$ is computed using the Riemann-Hurwitz formula: 
	$$g(\Lambda(\Gamma))=\frac{1}{2}(v(\Gamma)+2-N\chi(S^2))=\frac{1}{2}(v(\Gamma) + 2-2N-|\partial\Lambda(\Gamma)|)$$
	where $v(\Gamma)$ is the number of trivalent vertices of $\Gamma$ and $|\partial\Lambda(\Gamma)|$ denotes the number of boundary components of $\Gamma$. 
Solving for $v(\Gamma)=2g(\Lambda(\Gamma))-2+2N$ for a closed weave, we see that there are precisely $g(\La(\Gamma))+N-1$ bigons that we iteratively remove from the symmetric double $\La(L, L)$ in the process described above. By construction, $g(\La(\Gamma))=b_1(L)$, implying that $\La(L, L)=\#^{b_1(L)}\T^2_{\std}$, as desired. 
\end{proof}

Let $L$ be an exact Lagrangian filling of $\la$ and let $\mu_\gamma(L)$ denote the Lagrangian filling obtained from mutation at a cycle $\gamma$ of $L$. Following the proof of Theorem~\ref{thm: Symmetric}, we observe that when $L$ is a Legendrian weave filling, the double $\La(L, \mu_{\gamma}(L))$ is particularly simple to describe. Indeed, when we have a collection of sufficiently distant cycles $\gamma_1, \dots, \gamma_k$ that satisfy a combinatorial property that we define below, then we can verify that $\mu_{\gamma_1}\circ\dots\circ\mu_{\gamma_k}(L)$ is isotopic to a connect sum of Clifford and standard tori. 

\begin{definition}
    Let $L=\pi(\La(\Gamma))$ be an exact Lagrangian filling obtained as the Lagrangian projection of a Legendrian weave $\La(\Gamma)$.
    We call a collection $\mathcal{C}$ of $\sf{Y}$-trees $\{\gamma_1, \dots, \gamma_k\}\subseteq H_1(\La(\Gamma))$ \term if $L$ is Hamiltonian isotopic to a filling $L'=\pi(\La(\Gamma'))$ where each cycle in $\mathcal{C}$ is represented as a short {\sf I}-cycle in $\Gamma'$.
\end{definition}

Note that in particular, a disjoint collection of {\sf Y}-trees satisfies the conditions of the definition above. For such {\sf Y}-trees, one can use \cite[Proposition 3.5]{CasalsWeng} to reduce each to a short {\sf I}-cycle without introducing intersections between cycles in $\mathcal{C}$ or Reeb chords in the Legendrian weave. In particular, these moves are local and only modify a particular {\sf Y}-tree and the cycles that intersect it.

We now state and prove a preparatory lemma that we require for our proof of Theorem~\ref{thm:intro_std_and_clifford_tori}. 




\begin{lemma}\label{lemma: K4}
    Let $\La(\Gamma_1)$ be a 2-weave filling of $\la(2,n)$ with a collection $\mathcal{C}=\{\gamma_1, \dots, \gamma_k\}$ of short {\sf I}-cycles for $0\leq k \leq b_1(L)$ and let $\La(\Gamma_2)$ be a 2-weave obtained from $\La(\Gamma_1)$ by mutating at each cycle in $\mathcal{C}$ exactly once. For $L_1=\pi(\La(\Gamma_1))$ and $L_2=\pi(\La(\Gamma_2))$, the Legendrian double $\La(L_1, L_2)$ is Legendrian isotopic to $\#^{b_1(L_1)-k}\T^2_{\std}\#^k \T^2_C$.
\end{lemma}

Before proving the above lemma, we include a brief remark describing the relationship between 2-weaves and triangulations, as this will appear in the proof below.

\begin{remark}\label{rem:triangulations_and_2_weaves}
    We clarify how triangulations of $(n+2)$-gons correspond to trivalent graph fillings of the maximal tb torus link $\la(2,n)$ (refer for example \cite[Section 2]{Hughes2021b}). Label the vertices $1,2,\dots,n+2$ in clockwise order. Given a triangulation, the corresponding 2-weave is obtained from the dual graph, i.e. placing a vertex on every edge of the triangulation and inside every interior region, and then connecting the vertices for every region to the edge-vertices. This yields a trivalent graph, with the region-vertices are all trivalent, and we forget the bivalent vertices coming from the diagonals of the triangulation. Given such a triangulation $T$, consider a quadrilateral in it with vertices $i_1,i_2,i_3,i_4$ that are cyclically ordered and contains the diagonal $(i_1,i_3)$ (without loss of generality). We can obtain a new triangulation $T'$ by choosing the diagonal $(i_2,i_4)$ in the aforementioned quadrilateral, and keeping all other diagonals the same as $T$. Then $T'$ is said to be obtained by an {\em edge flip} from $T$. On the corresponding trivalent graph, this change is the same as shown in Figure~\ref{fig:Mutations}.

    Conversely, given a 2-weave filling $\la(2,n)$, draw it vertically with all the leaves of the binary tree at the top. Then label the regions to the left of these leaves $1, 2, \dots, n+2$, such that the rightmost region is again labeled $1$. Then, form a triangulation of the $(n+2)$-gon by including a diagonal $(i,k)$ if and only if the regions labeled $i$ and $k$ are adjacent; see Figure~\ref{fig:Triangulation} for an example.
\end{remark}

\begin{figure}[h!]{ \includegraphics[width=.7\textwidth]{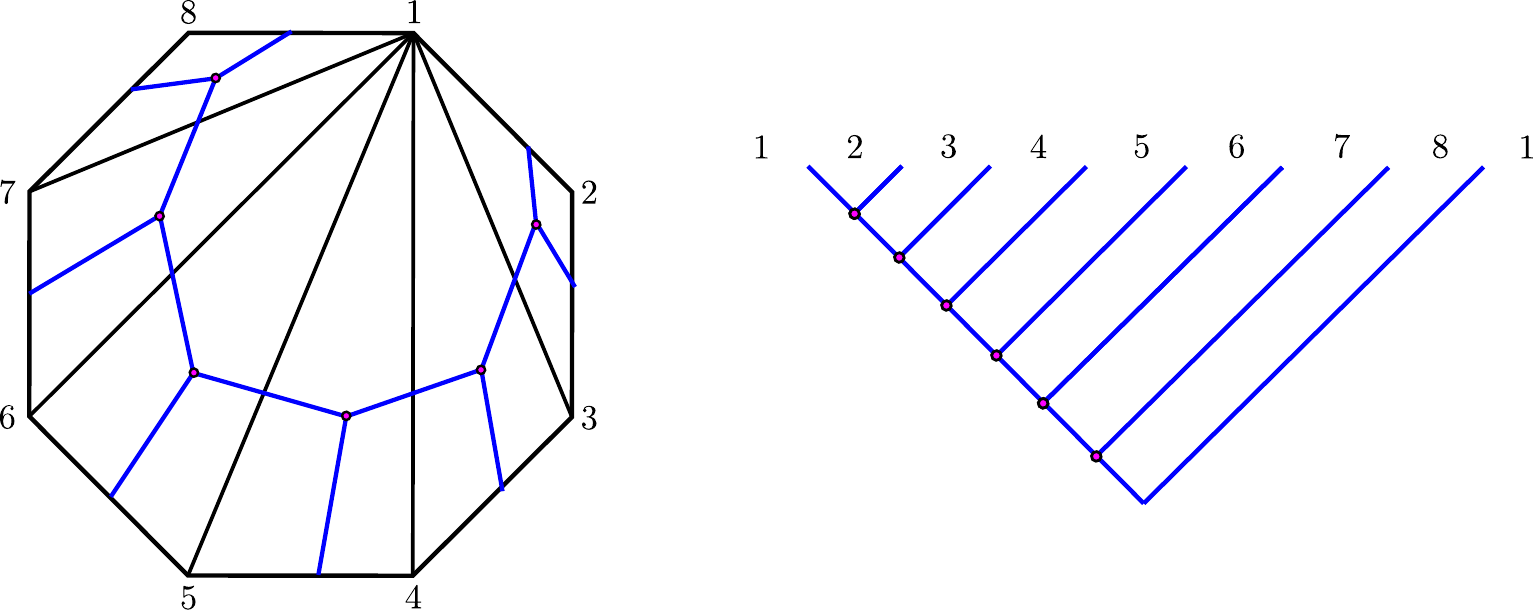}}\caption{A triangulation of the octagon with its dual trivalent graph (left) and the corresponding 2-weave drawn vertically with boundary edges labeled (right).}
\label{fig:Triangulation}\end{figure}

\begin{proof}
    Consider the Legendrian double $\La(L_1, L_2)$ specified as in the statement of the lemma. As in the proof of Theorem~\ref{thm: Symmetric}, we can choose coordinates $x_1, x_2$ and isotope $\La(L_1, L_2)$ so that $\La(L_1, L_2)\cap \{x_2\leq 0\}\cong L_2$ and $\La(L_1, L_2)\cap \{x_2\geq 0\}\cong L_1$. After a planar isotopy of the 2-graph, lifting to a Legendrian isotopy of $\La(L_1, L_2)$, we can assume that no two trivalent vertices appear in the same $x_2$ level set. Starting from $x_2=0$, we freely remove bigons from $\La(L_1, L_2)$ until we reach a trivalent vertex adjacent to a cycle in $\mathcal{C}$; if $\mathcal{C}=\emptyset,$ then $\La(L_1, L_2)\cong \#^{b_1(L_1)}\T^2_{\std}$, as in Theorem~\ref{thm: Symmetric}.

\begin{figure}
    \includegraphics[scale=0.4]{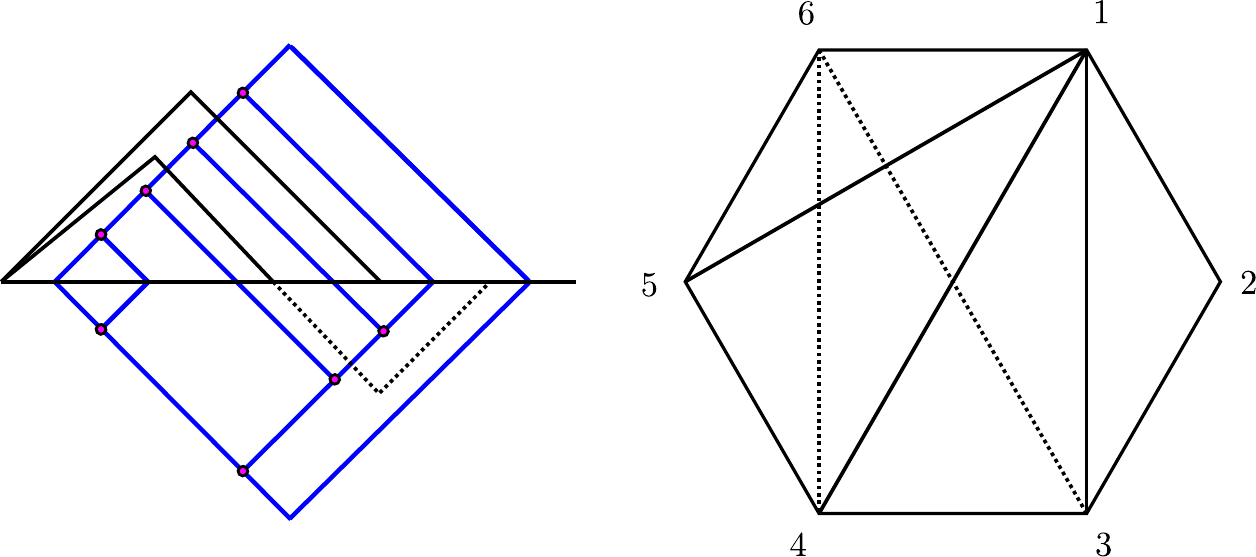}
    \caption{Example of a Legendrian double isotopic to $\T^2_{\std}\#\T^2_C\#\T^2_C$ with a $K_4$ depicted in the weave (left) and the corresponding pair of triangulations superimposed (right). The diagonals of $T_2$ are depicted with dashed lines; note that the diagonal $(1,3)$ appears in both $T_1$ and $T_2$. A $K_4$ appears in $T_1\cup T_2$ with vertices $1,4,5,6$.}
    \label{fig: K_4 example}
\end{figure}

    Now suppose $\mathcal{C} \neq \emptyset$ and $|\mathcal{C}|=k$. In the following, we show that $\La(L_1,L_2) = \La(L_1',L_2') \#T^2_C$, where $L_i'$ are 2-weave fillings of $\la(2,n-1)$ where $L_2'$ is obtained from $L_1'$ by mutating exactly once per cycle, for a set of $(k-1)$ cycles. The result then follows by inducting over $k$.
    
    Let $T_1$ and $T_2$ be the two triangulations of the $(n+2)$-gon corresponding to $\La(L_1)$ and $\La(L_2)$. Bigons of $\La(L_1, L_2)$ correspond to diagonals $(i, i+2)$ appearing in both $T_1$ and $T_2$, possibly after reindexing following the earlier removal of such a diagonal corresponding to another bigon. 
    Let $\Gamma$ be the graph obtained by superimposing $T_1$ and $T_2$ and removing all triangles dual to bigons in $\La(L_1, L_2)$. That is, the vertices of $\Gamma$ are the vertices of the $(n+2)$-gon and the edges are the union of edges in $T_1$ and $T_2$ where any shared edges are counted with multiplicity one and any triangles $i, i+1, i+2$ with diagonal $(i, i+2)$ appearing in $T_1$ and $T_2$ are removed, replacing the vertex labeling $i+2$ with $i+1$. For simplicity, we will continue denoting the number of vertices in $\Gamma$ by $n+2$. We denote by $K_4$ the complete graph on 4 vertices. The existence of a triangle in $\La(L_1, L_2)$ is then equivalent to a $K_4$ appearing as a subgraph of $\Gamma$ with three vertices being consecutive, i.e. the $K_4$ must contain vertices $i, i+1$ and $i+2$ for some $1\leq i \leq n+2$ with addition being taken modulo $n+2$; see Figure~\ref{fig: K_4 example} for an example. 

    Given such a $K_4$ with three consecutive vertices $(i,i+1,i+2)$ appearing in $\Gamma$, let us denote the fourth vertex as $j$. Without loss of generality, we can assume that $(i+1,j) \in T_1$ and $(i,i+2) \in T_2$, and the edge-flip relating these diagonals is precisely the mutation of a cycle in $\mathcal{C}$ that gives rise to the $T^2_C$ summand corresponding to the $K_4$. Now, consider the $(n+1)$-gon, and its triangulations $T_1'$ and $T_2'$ obtained, by deleting the edges $(i,i+1)$, $(i+1,i+2)$, and $(i+1,j)$ from $\Gamma$, and thus from $T_1$ and $T_2$. The 2-weaves dual to $T_1'$ and $T_2'$ are the fillings $L_1'$ and $L_2'$ of $\la(2,n-1)$ above, and $\La(L_1,L_2) = \La(L_1',L_2') \# T^2_C$ (by applying Theorem~\ref{thm:czsurgery} to the triangle in the 2-weave picture). It is easy to see that $T_1'$ and $T_2'$ are related by $(k-1)$ edge flips, corresponding to the remaining flips of $T_1$ and $T_2$.

    We now verify the claim that such a $K_4$ must appear in $\Gamma$, arguing by induction. 
    

Note that if $T_2$ is obtained from $T_1$ by a single mutation, the claim is true, as we can successively remove triangles corresponding to the edges that are not mutated until we are left with a square and the pair of mutated diagonals forming a $K_4$. Note also that the claim is true for $n=2$ and $n=3$, corresponding to the case of the square and pentagon. Now we induct on $n$, assuming $n \geq 4$. Let $(i,k)$ be the first diagonal in $T_1$ that is mutated to $(j,l)$ when going from $T_1$ to $T_2$. Now let $P_{j\to l}$ be the $(|j-l|+1)$-gon formed by the vertices and edges of the $(n+2)$-gon going clockwise from $j$ to $l$, together with $(j,l)$, and let $P_{l \to j}$ be the $(n+2-|j-l|+1)$-gon formed by the vertices and edges of the $(n+2)$-gon going clockwise from $l$ to $j$, together with $(j,l)$. Since $i$ and $k$ lie between $j$ and $l$, we must have $2 \leq |j-l| \leq n$. Since every diagonal in $T_1$ is mutated at most once to get to $T_2$, we must have that $T_2$ restricts to triangulations of $P_{j\to l}$ and $P_{l \to j}$, and $T_1 \setminus \{(i,k)\}$ must also restrict to triangulations of $P_{j\to l}$ and $P_{l \to j}$. If $T_1$ and $T_2$ agree on the restricted polygons, the claim is true by the first sentence. If not, $T_1$ and $T_2$ restricted to $P_{j\to l}$ and $P_{l \to j}$ satisfy the hypotheses of the claim, and one of the two must have at least four vertices, but at the same time fewer than $n+2$. Thus the union of the restricted triangulations to one of the two must contain an induced $K_4$ with three consecutive vertices by the inductive hypothesis. If the three vertices are also consecutive in the $(n+2)$-gon, the claim is true, so now suppose that two of the vertices are $j$ and $l$ in $P_{l \to j}$. Then, without loss of generality, the vertices are $\{(j-1), j, l, l'\}$. We must have that the vertex $l'$ is equal to or clockwise of $i$, as there can be no diagonals that cross diagonals of both $T_1$ and $T_2$; see (I) and (II) of Figure~\ref{fig: k_4-search}. Therefore either $(l,l')$ belongs to $T_1$, or $(l,l')$ and $(l,(j-1))$ must both belong to $T_2$. In both cases, we have that the triangulations $T_1$ and $T_2$ must restrict to triangulations of $P_{l \to (j-1)}$ after removing any diagonals that intersect with $(l, (j-1))$. 

\textbf{Case 1:} We first assume $l' \neq i$, and consider $P_{l \to (j-1)}$. Repeating the argument, $P_{l \to (j-1)}$ must have an induced $K_4$ with three consecutive vertices. Here again, if the three consecutive vertices were not consecutive in the $(n+2)$-gon, we must have them be $(j-2), (j-1), l$ (see (III) and (IV) in Figure~\ref{fig: k_4-search}), and the fourth vertex must be equal to or clockwise of $i$ and counter-clockwise of $(j-2)$. We continue as before with $P_{l \to (j-2)}$.

\textbf{Case 2:} If the fourth vertex $l'$ of the $K_4$ appearing in $P_{l\to j}$ is equal to $i$, then when considering $P_{l \to (j-1)}$, we obtain another $K_4$ by the inductive hypothesis. As above, we consider the case where the vertices are not consecutive in the original $(n+2)$-gon, and without loss of generality include the vertices $(l+1)$,$l$,$(j-1)$. In this case the fourth vertex would be forced to be counter-clockwise or equal to $i$. 
In Case 2, the argument proceeds as before with $P_{(l+1)\to (j-1)}$. 

In both cases, since the number of vertices is strictly decreasing, and the vertices between $i$ and $j$ in the $K_4$ get successively closer as vertices in the original $(n+2)$-gon, we must eventually find an induced $K_4$ where the three consecutive vertices in $P_{l\to (j-m)}$ are actually consecutive in the original $(n+2)$-gon. 


\begin{figure}
    \includegraphics[scale=0.3]{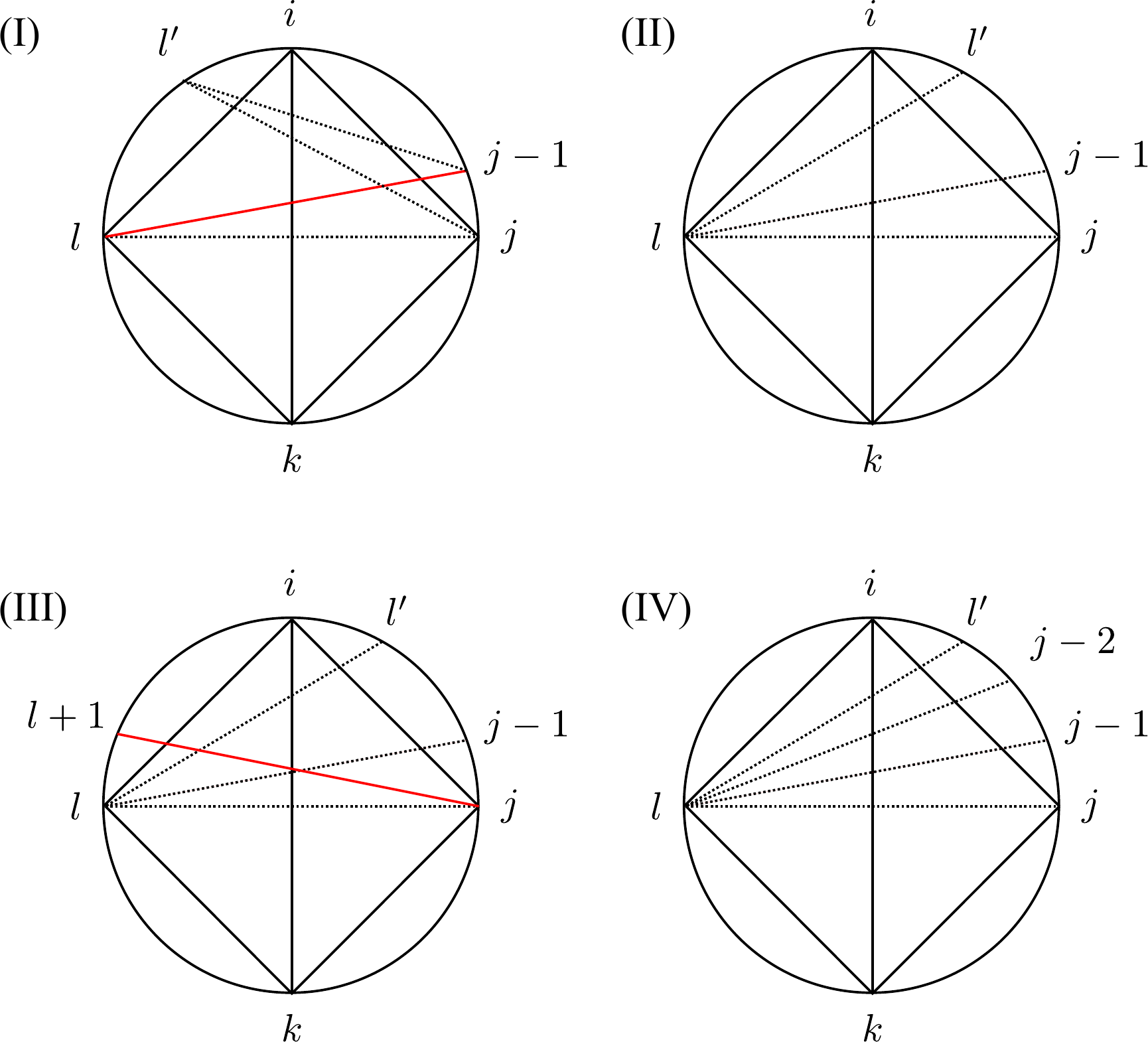}
    \caption{Schematics of the superimposition of triangulations $T_1$ and $T_2$ appearing in the proof of Lemma~\ref{lemma: K4}. The dotted edges depict diagonals in $T_2$ while the solid ones depict diagonals in $T_1$. The red edges in (I) and (III) are obstructed by the fact that they intersect diagonals of both $T_1$ and $T_2$.}
    \label{fig: k_4-search}
\end{figure}
\end{proof}

We now restate and prove Theorem~\ref{thm:intro_std_and_clifford_tori}.

\begin{theorem}\label{thm:std_and_clifford_tori}
Suppose that $L_1=\pi(\La(\Gamma_1))$ is an embedded exact filling of $\la$ and that $\La(\Gamma_1)$ has a collection $\mathcal{C}=\{\gamma_1, \dots, \gamma_k\}$ of \term {\sf Y}-trees with $0\leq k \leq b_1(L_1)$. Let $L_2$ be a filling of $\la$ obtained from $L_1$ by mutating each of the cycles in $\mathcal{C}$ exactly once. Then $\La(L_1, L_2)$ is Legendrian isotopic to $\#^{b_1(L)-k}\T^2_{\std}\#^k \T^2_C$.  
\end{theorem}


 \begin{proof}

We first claim that we can reduce from the case of a general $N$-graph to the case of a 2-graph. Let us denote by $\Gamma_{1}$ and $\Gamma_{2}$, the $N$-graphs corresponding to $L_1$ and $L_2$ and let $\mathcal{C}'=\{\gamma_1', \dots, \gamma_k'\}$ be the collection of $\L$-compressing cycles of $L_2$ obtained from $\mathcal{C}$ by the process of reducing each $\gamma_i$ to a short ${\sf I}$ cycle and mutating them exactly once. After reducing the cycles in $\mathcal{C}$ and $\mathcal{C}'$ to short {\sf I}-cycles, we can see that $L_1$ and $L_2$ are symmetric except for neighborhoods around cycles in $\mathcal{C}$ and $\mathcal{C}'$. As in the proof of Theorem~\ref{thm: Symmetric}, we can choose coordinates $x_1, x_2$ and isotope $\La(L_1, L_2)$ so that $\La(L_1, L_2)\cap \{x_2\leq 0\}\cong L_2$ and $\La(L_1, L_2)\cap \{x_2\geq 0\}\cong L_1$. Denote by $t$ the smallest value of $x_2$ for which one of the trivalent vertices adjacent to a cycle in $\mathcal{C}$ appears; we perform a planar isotopy of $L_2$ so that the first trivalent vertex of $\mathcal{C}'$ appears at $x_2=-t$. We label these two initial trivalent vertices $v_0$ and $v_0'$ and denote the corresponding edge labels by $\sigma_i$.  Proceeding from $x_2=0$ to $x_2=\pm t$, we resolve any singularities of $\La(L_1, L_2)$. Since, after simultaneous simplification, $\mathcal{C}$ and $\mathcal{C'}$ only contain short {\sf I}-cycles, there exists a neighborhood $N$ of $\Gamma_1\cup\Gamma_{2}$ including $v_0$ and $v_0'$ that consists only of edges labeled by $\sigma_i$ connected by trivalent vertices. We choose $N$ so that it is maximal with respect to inclusion of cycles in $\mathcal{C}$ and $\mathcal{C}'$ while satisfying these conditions. By construction, outside of $N$, $\Gamma_1$ and $\Gamma_{2}$ are either identical or contain additional cycles of $\mathcal{C}$ or $\mathcal{C}'$ disjoint from the cycles of $\mathcal{C}$ or $\mathcal{C}'$ appearing in $N$. Since $\mathcal{C}$ is finite, there is a finite number of such neighborhoods, and, after removing triangles from $N$ we can ensure that $\Gamma_1$ and $\Gamma_{2}$ are symmetric within the set $\{-t\leq x_2\leq t\}$, allowing us to repeat the process for the remaining neighborhoods of cycles in $\mathcal{C}$. Therefore, it is enough to show how we can find a triangle in $\La(L_1, L_2)$ in the case that $L_1$ and $L_2$ are represented by a 2-graph, as such a triangle would represent a Clifford torus summand by Theorem~\ref{thm:czsurgery}. We can then use induction to resolve the 2-weave case, and the proof in the general case follows from our argument above.

 \end{proof}

Theorem \ref{thm:std_and_clifford_tori} immediately implies that varying $n$, we can obtain a connect sum of any number of standard and Clifford tori from doubles of fillings of $\la(2, n)$, as we can always find a homology basis of short {\sf I}-cycles for any 2-weave. 

\begin{corollary}
Given a filling $L_1$ of $\la(2,n)$, and an integer $j$ such that $0 \leq j \leq n-1$, there exists $L_2$ such that $\La(L_1,L_2)$ is Hamiltonian isotopic to $\#^{j} \T^2_{\std} \#^{n-1-j} \T^2_{C}$
\end{corollary}


For two embedded exact Lagrangian fillings $L$ and $L'$, of the Legendrian link $\la(\beta\Delta)$,
direct computation of $\mathcal{M}_1(\La(L, L'))$ via Lemma~\ref{lem: Double Sheaves} implies the following, which is a restatement of Lemma~\ref{lem: intro_non_loose}:

\begin{proposition}
Let $L, L'$ be two embedded exact Lagrangian fillings of $\la(\beta\Delta)$. The Legendrian double $\La(L, L')$ is a non-loose Legendrian. 
\end{proposition}

\begin{proof}
By \cite[Theorem 20]{Escobar16} (see also \cite[Section 3.3]{CGGLSS} together with \cite[Theorem 2.39]{CGGS1}), the framed sheaf moduli of $\la\cong \partial(L)\cong \partial(L')$ is an irreducible algebraic variety. Since $\mathcal{T}_{L}$ and $\mathcal{T}_{L'}$ are open subsets of $\SM_1^{fr}(\la)$, the intersection $\mathcal{T}_{L}\cap\mathcal{T}_{L'}$ is necessarily nonempty. By Lemma~\ref{lem: Double Sheaves}, $\SM_1(\La(L, L')$ is therefore nonempty. Hence $\La(L, L')$ is non-loose.   
\end{proof}

Nevertheless, if we allow $L$ or $L'$ to be immersed, then we immediately obtain examples of loose Legendrians from the doubling construction.

\begin{ex}\label{ex: immersed}
Consider any embedded orientable weave filling $L$ of $\la(2,n)$ and an immersed orientable weave filling $L'$ of $\la(2, n)$ obtained from $L$ by deleting a short {\sf I}-cycle $\gamma$; see Figure~\ref{fig:Loosedouble} for an example. The 2-graph $\Gamma$ representing $\La(L, L')$ contains a bridge, i.e. an edge whose deletion increases the number of connected components of $\Gamma$, and is therefore loose by \cite[Proposition 4.24]{CasalsZaslow}.

\begin{figure}[h!]{ \includegraphics[width=.8\textwidth]{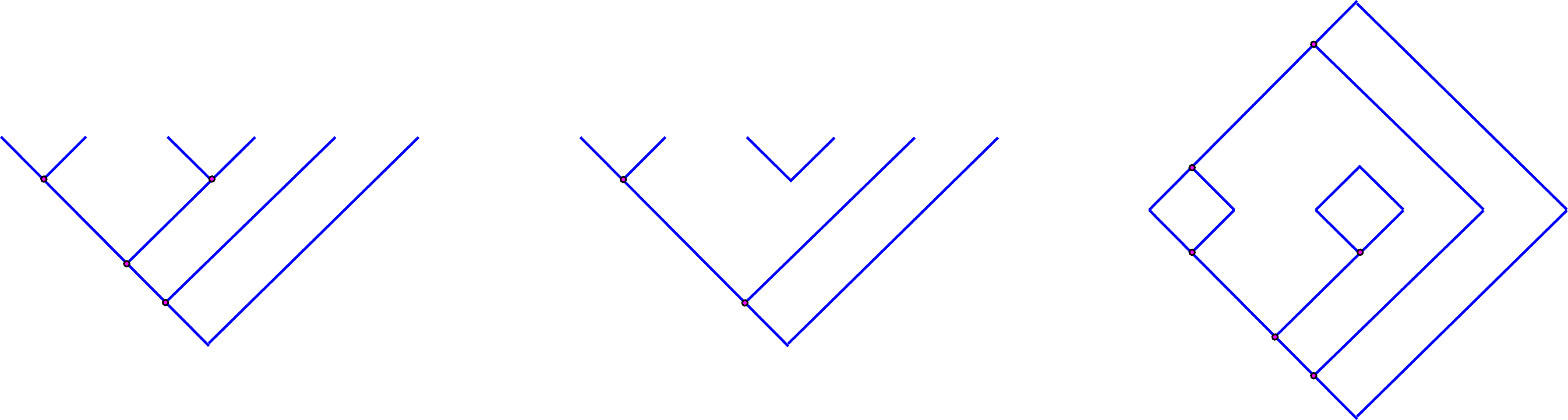}}\caption{A Legendrian weave $L'$ (middle) representing an immersed Lagrangian filling obtained from the weave $L$ (left) by deleting a short {\sf I}-cycle. The resulting Legendrian double $\La(L, L')$ is loose.}
			\label{fig:Loosedouble}\end{figure}

\end{ex}

\section{Doubles of Torus Link Fillings}\label{sec: torus_knot_doubles}

In this section, we continue to explore methods of distinguishing Legendrian doubles having to do with the combinatorics of Legendrian weaves. In the setting of Legendrian torus links, particularly the family $\la(2, n)$, the classification problem is a bit more tractable than the general setting, and we are able to employ techniques involving the combinatorics of triangulations and point counts of the sheaf moduli over finite fields to give Legendrian isotopy characterizations of certain doubles. 


The Legendrian torus link $\la(2,n)$, has at least $C_n$ embedded exact Lagrangian fillings up to Hamiltonian isotopy, which can be represented by Lagrangian projections of Legendrian weaves \cite{Pan-fillings, TreumannZaslow}. Each of these weaves is a trivalent graph with $n$ vertices and can be obtained as the dual of a triangulation of the $(n+2)$-gon; we will refer to the corresponding fillings as {\em trivalent graph fillings}. Recall that we denote by $L_{init}$ the filling of $\la(2,n)$ coming from trivalent graph with all vertices along the outside edge; see Figure~\ref{fig: initial} for an example in the case of $\la(2,4)$.

\begin{theorem}\label{thm: torusknotfillings}
Let $L$ be any trivalent graph filling of $\la(2,n)$. The double $\La(L_{init},L)$ is Hamiltonian isotopic to $\#^{k} \T^2_{\std} \#^l \T^2_{C}$ for some non-negative integers $k$ and $l$ such that $k+l=n-1$.
\end{theorem}
\begin{proof}

Start with the ``highest" or ``interiormost" trivalent vertex in the weave representation of $L$ (if there are none, then it has to be $L_{init}$ and $l=0$). The two edges of this vertex that go to the boundary of the graph continue into the weave for $L_{init}$ and meet the outside edge at two trivalent vertices, thus creating a triangle. This corresponds to a connect summed $\T_c$; doing the reverse operation as in the third row of Figure~\ref{fig:surgery} takes that summand out. The resulting 2-graph is the double of the initial filling of $\la(2,n-1)$ and some other weave filling, and the claim follows by induction.
\end{proof}

\begin{remark}
Theorem \ref{thm: torusknotfillings} can be understood as a geometric consequence of a combinatorial property of triangulations of the ($n+2$)-gon (see Remark~\ref{rem:triangulations_and_2_weaves} for more details). In particular, the triangulation corresponding to $L_{init}$ is a ``fan" with every edge sharing a single vertex; see Figure~\ref{fig:Triangulation} for a depiction. This fan triangulation (or any of its rotations) is special in that it is at most $n-1$ edge flips (i.e. mutations) away from any other triangulation.  
\end{remark}

\subsection{Cubic planar Legendrian doubles and point counts}

In this section, unless otherwise mentioned for statements made in greater generality, $\La$ will always be a connected orientable Legendrian surface in $J^1(S^2)$ arising as a  planar trivalent graph. Denote by $\mathcal{M}_1(\La; \F_q)$ the moduli of microlocal rank-one sheaves with singular support on $\La$ with coefficients in $\F_q$ i.e. the constructible sheaves appearing in the definition of singular support are locally constant and take values in $(\F_q)^n$. For a cubic planar Legendrian $\La$, we will refer to a choice of a  planar trivalent graph as $\La_\Gamma$, and denote its dual graph by $\La^*_\Gamma$. Note that the graph $\La_\Gamma$ is not necessarily well-defined up to isotopy. Then we have (by work of Sackel in \cite[Appendix A]{CasalsMurphy} and independently by Treumann-Zaslow \cite{TreumannZaslow}):

\begin{lemma}[Prop. 1.2 \cite{TreumannZaslow}]\label{lem: point count}
   $$ (q+1)(q)(q-1)|\mathcal{M}_1^{fr}(\La; \F_q)|=P_{\La^*_\Gamma}(q+1)$$
\end{lemma}
where $P_{\La^*_\Gamma}(q+1)$ is the chromatic polynomial of $\La^*_\Gamma$ evaluated at $q+1$.

\begin{remark}
    As mentioned in the introduction (see text after Theorem~\ref{thm: intro_initialfilling}), a natural question arising from  \cite{TreumannZaslow} is whether the chromatic polynomial is a complete invariant for cubic planar Legendrians. The doubles in Figure~\ref{fig:chromaticpolydouble} have the same chromatic polynomials, though it is not clear that the corresponding Legendrian surfaces are isotopic. In fact, the  graphs are not isomorphic, and unlike the examples in \cite[Remark 4.3]{TreumannZaslow} due to Casals-Murphy, do not have any triangles that could correspond to Clifford tori summands. To our knowledge the Casals-Murphy examples are the only situation in which non-isotopic 2-graphs can be shown to correspond to isotopic Legendrian surfaces.
\end{remark}

\begin{lemma}\label{lem:chromatic}
Let $\La$ be a closed, embedded Legendrian surface and $k$ and $l$ be non-negative integers with $k, l\leq b_1(\La)$. For $\La=\La'\#^k \T^2_{\std}\#^l \T^2_C$, we have $P_{\La^*_\Gamma}(q+1)=(q-1)^k(q-2)^lP_{\La'^*_G}(q+1)$.
\end{lemma}

\begin{proof}
Let $\La=\La'\#^k \T^2_{\std}\#^l \T^2_C$. Since connect sum is a local operation, there is, up to Legendrian isotopy, some neighborhood of each $\T^2_{\std}$ summand where the singularities of the front projection of $\La$ are given by the right hand side of Move (2) of Figure~\ref{fig:surgery}. Similarly, for each $\T^2_C$ summand, there is a neighborhood where the singularities of the front projection of $\La$ are given by the right hand side of Move (3) of Figure~\ref{fig:surgery}. By \cite[Theorem 5.12]{CasalsZaslow}, we have $\SM_1(\La'\#\T^2_{\std};\F_q)\cong \F_q^\ast \times \SM_1(\La';\F_q)$ and similarly, $\SM_1(\La'\#\T^2_{C};\F_q)\cong \F_q\backslash\{0,1\} \times \SM_1(\La';\F_q)$. After applying Lemma~\ref{lem: point count}, we obtain the desired result, as $|\F_q^\ast|=q-1$ and $|F_q\backslash\{0,1\}|= q-2$.

\end{proof}

\subsection{Mutation distance}

As an alternative method to distinguishing doubles by point counts, we introduce the notion of mutation distance. Let $\la$ be a braid-positive Legendrian link for which we know the ring of regular functions on the sheaf moduli $\mathcal{M}_1(\la)$ admits a cluster structure. By \cite[Corollary 1.3]{CasalsGao24} for every cluster seed in $\mathbb{C}[\mathcal{M}_1(\la)]$, there is an exact embedded Lagrangian filling, equipped with an $\mathbb{L}$-compressing system $\mathcal{C}$, which define toric coordinates $A(\mathcal{C})$ such that they are cluster coordinates for a cluster seed for the cluster algebra structure on $\mathbb{C}[\SM_1(\la)]$. 

\begin{definition}
     Let $\la$ be Legendrian link in standard $S^3$. The {\em mutation distance} $d_\mu(L_1, L_2)$ between embedded exact Lagrangian fillings $L_1$ and $L_2$ of $\la$ is the smallest integer $n$ such that $\mu_n\circ \dots \mu_1(L_1)\cong L_2$ where $\cong$ denotes Hamiltonian isotopy, and $\{\mu_i\}_{1\leq i\leq n}$ denotes a sequence of $n$ Lagrangian disk surgeries. 
\end{definition}

If $\la$ is a braid-positive Legendrian link for which the ring of regular functions on the sheaf moduli $\mathcal{M}_1(\la)$ admits a cluster structure, it is known that there exists a sequence whose length is equal to the smallest number of mutations relating the two associated cluster seeds \cite{CasalsGao24}. However, note that the correspondence between exact Lagrangian fillings and cluster charts is via an $\mathbb{L}$-compressing system. We currently do not know whether a given exact Lagrangian filling can admit two or more distinct $\mathbb{L}$-compressing systems. Consequently, the toric chart induced by a filling could potentially belong to distinct cluster structures on $\mathbb{C}[M_1(\la)]$, induced by the distinct $\mathbb{L}$-compressing systems. This means that given two exact embedded fillings $L_1$ and $L_2$ of the link $\la$, there could be distinct sequences of Lagrangian disk surgeries relating $L_1$ and $L_2$, such that each sequence has length equal to the smallest number of mutations on some cluster structure on $\mathbb{C}[M_1(\la)]$.

For this reason, we introduce another notion, that of an {\em algebraic} mutation distance, which measures the distance between the associated cluster charts corresponding to fillings $L_1$ and $L_2$. 

\begin{definition}
    The {\em algebraic mutation distance} $d_\mu^a(L_1, L_2)$ between fillings $L_1$ and $L_2$ is the smallest integer $n$ such that the associated cluster seeds are related by $n$ mutations.
\end{definition}

We define $d_\mu^a(L_1, L_2)$ to be infinite if $L_1$ and $L_2$ induce cluster seeds in distinct cluster algebras. This can occur if $\SM_1(\la)$ is not irreducible, as in the case of $m(8_{21})$, or if there are distinct cluster structures on the same irreducible variety induced by different collections of fillings of $\la$. There are as yet no known examples of the latter kind, nor are such examples expected. As a consequence, it is reasonable to conjecture that $d_\mu(L_1,L_2) = d_\mu^a(L_1,L_2)$. Theorem~\ref{thm: torusknotfillings} above gives an example of a cases in which $d_\mu^a$ is readily computable and determines the Legendrian isotopy class of Legendrian doubles formed from Legendrian weaves. These cases, together with additional computational evidence, suggest the following. 

\begin{conjecture}\label{conj: Mutation distance invariant}
    The mutation distance $d_\mu(L_1, L_2)$ between $L_1$ and $L_2$ is a Legendrian isotopy invariant of the Legendrian double $\La(L_1, L_2)$.
\end{conjecture}

Several results below give additional evidence for this conjecture in simple cases where the combinatorics of trivalent graphs allows us to compute $d_\mu$.

We will also define another notion of distance of Legendrian surfaces. Let $\La$ be a Legendrian surface (not necessarily cubic planar). Recall that Legendrian disk surgery is a lift of Lagrangian disk surgery. In particular, if the Legendrian surface in question is a Legendrian double $\La(F,G)$, and $\mu(F)$ is an exact Lagrangian obtained by Lagrangian disk surgery on some cycle in $F$, the double $\La(\mu(F),G)$ can be obtained from $\La(F,G)$ by a Legendrian disk surgery. 

By \cite[Proposition 4.24]{CasalsZaslow}, a cubic planar Legendrian whose graph contains a bridge is `loose' (defined by Murphy \cite{Murphy??}). These Legendrians are known to satisfy an $h$-principle. We refer to a Legendrian disk surgery on a non-loose Legendrian to be `non-loose' if the resulting Legendrian is also non-loose.

\begin{definition}
    We denote by $d_\std(\La)$ the minimum number of non-loose Legendrian disk surgeries required to produce a connect sum of standard tori from $\La$. We define $d_\std(\La)$ to be infinite if no such sequence of Legendrian disk surgeries exists.
\end{definition}

The necessity of `non-loose' in the above is that by the h-principle for loose Legendrians, any two cubic planar Legendrians with the same topology can be connected by two disk surgeries if loose surgeries are allowed. Note that when $\la$ is a braid positive Legendrian link, Theorem~\ref{thm: Symmetric} and the discussion above immediately imply that \[d_\mu^a(L_1, L_2) \geq d_\mu(L_1, L_2)\geq d_\std(\La(L_1, L_2)).\] Since $d_\std(\La(L_1, L_2))$ is manifestly a Legendrian isotopy invariant of $\La(L_1, L_2)$, Conjecture~\ref{conj: Mutation distance invariant} could be verified by showing that $d_\mu^a(L_1, L_2) = d_\std(\La(L_1, L_2))$. 

Given any two fillings $L_1$ and $L_2$ of a Legendrian $\la$, even the quantity $d_\mu^a(L_1, L_2)$ is generally not straightforward to compute. For example, if $L_1$ and $L_2$ are fillings of $\la(2, n)$, the algebraic mutation distance between $L_1$ and $L_2$ is directly related to a longstanding combinatorial problem of computing the number of edge flips between any two triangulations of the $n$-gon (see \cite{SleatorTarjanThurston} for more on the combinatorial problem).

\begin{remark}\label{rem: SSZcluster}
There is a putative cluster structure \cite{scharitzer-shende, SchraderShenZaslow} on the space of cubic planar Legendrians. In those articles, the authors investigate wavefunctions or counts of holomorphic disks in 6-space that (conjecturally) correspond to Gromov-Witten invariants of Lagrangian fillings of the cubic planar Legendrians. According to \cite{SchraderShenZaslow}, given a cubic planar graph $\Gamma \subset S^2$, the associated cluster chart can be identified with a torsor over rank one local systems on the Legendrian surface $\La(\Gamma)$. The cluster variety is the space of Borel-decorated $\PGL_2$ local systems on a punctured sphere $S$ with unipotent monodromy around the punctures. Mutation in this cluster algebra corresponds to edge flips on the cubic planar graphs. This allows one to speculate about an algebraic analogue of $d_{\std}$. The graph $\Gamma_{neck}$ \cite[Figure 1.5.1]{SchraderShenZaslow} lifts to a connect sum of standard tori. Given a graph $\Gamma$, one can define $d_{\std}^a(\La(\Gamma))$ to be the minimum number of edge flips relating $\Gamma$ with a cubic planar graph $\Gamma'$ which corresponds to a connect sum of standard tori. This raises the following interesting question about distances in two seemingly unrelated cluster algebras -- with $L_1$ and $L_2$ as above, is $d_\mu^a(L_1,L_2) = d_{\std}^a(\La(L_1,L_2))$? We point out that moduli of sheaves $\SM_1(\La)$ for a cubic planar Legendrian appears in {\em loc. cit.} as the intersection of the cluster chart corresponding to $\La$ with the chromatic Lagrangian $\SM$; see \cite[Remark 4.2]{SchraderShenZaslow}. 
\end{remark}

We show now that if the algebraic mutation distance between two trivalent graph fillings of $\la(2,n)$ is sufficiently large, the corresponding Legendrian double is not a connect sum of standard and Clifford tori. This is a restatement of Theorem~\ref{thm: intro_notdecomposable}.

\begin{theorem}\label{thm: notdecomposable}
    Let $L_1$ and $L_2$ be two embedded trivalent graph fillings of $\la(2,n)$. For $n \geq 4$, if $d_{\mu}^a (L_1, L_2) \geq n$, the double $\La(L_1,L_2)$ is not a connect sum of standard and Clifford tori.
\end{theorem}

\begin{figure}[h!]{ \includegraphics[width=.8\textwidth]{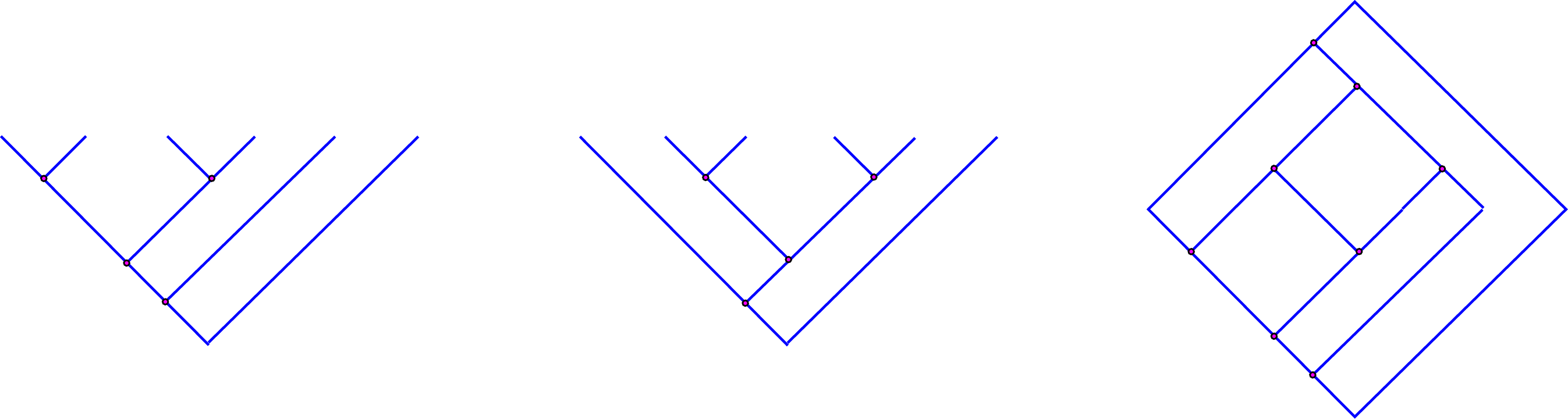}}\caption{Legendrians $L_1$ (left), $L_2$ (middle), and their double $\La(L_1, L_2)$ (right). The double is not a connect sum of standard and Clifford tori.}
\label{fig:CubeGraphdouble}\end{figure}

\begin{remark}\label{rem:t(2,4)_example}
    Recall that the Catalan number of trivalent graph fillings of the Legendrian link $\la(2,n)$ in standard contact $S^3$ are in one-to-one correspondence with triangulations of a regular $n$-gon; see \cite[Section 2.3.3]{TreumannZaslow} or \cite[Section 2.3.1]{Hughes2021b}. The first non-decomposable example is obtained for $n=4$. Let $L_1$ and $L_2$ be trivalent graph fillings of $\la(2,4)$ and consider their double $\La(L_1, L_2)$. Then as long as $d_\mu^a(L_1,L_2) \leq 3$,
    the doubles are connect sums of standard and Clifford tori -- the reader is invited to check this. 
    If $L_1$ and $L_2$ are the two trivalent graph fillings depicted in Figure~\ref{fig:CubeGraphdouble} (left) and (middle), then $\La(L_1, L_2)\not \cong  \#^{k} \T^2_{\std} \#^l \T^2_{C}$. The 2-graph obtained by doubling the trivalent binary trees is the cube graph. In particular, $$P_{\La^*(L_1, L_2)}(q+1)=(q+1)(q)(q-1)(q^3-6q^2+14q-11)$$ where the cubic factor is irreducible over $\Q$ and has only one real root. It follows from the main result of \cite[Theorem 2]{SleatorTarjanThurston} 
    that for $n \geq 9$ there are exact fillings $L_1$ and $L_2$ of $\la(2, n)$ such that $n < d_\mu(L_1, L_2) \leq 2n-6$. Observe that all the triangulations of the hexagon that are distance four apart (see \cite[Figure 4]{SleatorTarjanThurston}) give the cube graph when doubled.
\end{remark}

The above theorem, Theorem~\ref{thm: notdecomposable}, will be proved in a sequence of steps.

We make the following observation about the 2-graph corresponding to a double, and its dual graph. Note that given two trivalent graph fillings $L_1$ and $L_2$ of $\la(2,n)$, we have a labeling of the univalent vertices at the leaves by $\{1,2,\dots,n+2\}$, coming from the region to its left as in Remark~\ref{rem:triangulations_and_2_weaves}. We use these labelings to construct the doubled graph $\La_\Gamma(L_1,L_2)$ as in Definition~\ref{def: doublegraph}.

\begin{lemma}\label{lem: doubledual}
    Let $L_1$ and $L_2$ be two embedded trivalent graph fillings of $\la(2,n)$. Let $T_1$ and $T_2$ be the corresponding triangulations of the $(n+2)$-gon. Let $\La_\Gamma(L_1,L_2)$ denote the trivalent graph whose Legendrian weave is the double $\La(L_1,L_2)$. Then the dual graph $\La_\Gamma^*(L_1,L_2)$ is the graph obtained by superimposing the triangulations $T_1$ and $T_2$ -- i.e. it has vertex set $\{1,2,\dots,n+2\}$, and its edge set is the union of the edges coming from $T_1$ and $T_2$. 
\end{lemma}

\begin{proof}
Recall that the dual graph $\La_\Gamma^*(L_1,L_2)$ is the simple graph obtained by putting a vertex in each component of $S^2 \setminus \La_\Gamma(L_1,L_2)$, and an edge between any two adjacent regions. Observe that $S^2 \setminus \La_\Gamma(L_1,L_2)$ has exactly $(n+2)$ regions, labeled $\{1,2,\dots,(n+2)\}$, inherited from the labels on the top of $L_1$ and $L_2$. Consequently, the dual graph $\La_\Gamma^*(L_1,L_2)$ has $(n+2)$ vertices, and two vertices $i$ and $j$ are connected by an edge if and only if the regions $i$ and $j$ are adjacent in $L_1$ or $L_2$. Following Remark~\ref{rem:triangulations_and_2_weaves}, that happens if and only if the vertices $i$ and $j$ of the $(n+2)$-gon are connected by a diagonal in $T_1$ or $T_2$. The lemma follows.
\end{proof}

Following that, we recall from Section~\ref{sec: legendrian_doubles} a consequence in the dual graph of the presence of triangles in $\La_\Gamma(L_1,L_2)$. If there is a triangle in $\La_\Gamma(L_1,L_2)$, the dual graph has vertices $(i-1), i, (i+1),$ and $k$ such that they form a complete subgraph of the dual graph (see Figure~\ref{fig: K_4 example}). We set up the following notation for the mutation of a triangulation: Given a triangulation $T$ of an $n$-gon and a diagonal $D \subset T$, the triangulation $\mu_D(T)$ is the one obtained by flipping the diagonal $D$.

\begin{lemma}\label{lem: inductionstep}
Let $L_1=\La(\Gamma_1)$ and $L_2=\La(\Gamma_2)$ be two trivalent graph fillings of $\la(2,n)$ and $T_1$, $T_2$ be the corresponding triangulations of the $(n+2)$-gon. If $\La_\Gamma(L_1,L_2)$ contains a triangle, all of the following are true:
\begin{enumerate}[i)]
\item There exists a diagonal $D$ in $T_2$ (or $T_1$) such that $\mu_D(T_2)$ (respectively $\mu_D(T_1)$) has one fewer diagonal distinct from $T_1$ (resp. $T_2$) than $T_2$ (resp. $T_1$)
\item We can remove an edge from both $\Gamma_i$ to get $\Gamma_i'$, that correspond to fillings $L_i'$ of $\la(2,n-1)$ and join to form an edge $E$ in $\La_\Gamma(L_1,L_2)$ which is a side of the triangle, such that $\La_\Gamma(L_1,L_2) - E = \La_\Gamma(L_1',L_2')$
\item $d_\mu^a(L_1',L_2') = d_\mu^a(L_1,L_2)-1$
\end{enumerate} 
\end{lemma}

\begin{proof}
    Note that the triangle cannot be included completely in $\Gamma_1$ or $\Gamma_2$, since the latter are trivalent binary trees, and further the existence of a closed region in the weave would mean that the corresponding Lagrangian is not embedded. We first characterize the edges in $\Gamma_1$ and $\Gamma_2$ that bound the triangle. Firstly, without loss of generality, one of the vertices of the triangle comes from $\Gamma_1$, and the other two from $\Gamma_2$. Using the description of $\Gamma_i$'s as vertical weaves (as in Remark~\ref{rem:triangulations_and_2_weaves}), every edge in $\Gamma_i$ has a top label, from the region to its left. For the vertex in $\Gamma_1$, the labels on its edges must be consecutive (mod $(n+2)$), suppose they are $i-1$ and $i$. In $\Gamma_2$, there are three edges contributing to two vertices. Two of the edges have top labels $i-1$ and $i$ -- suppose the third one is $k$.
    \begin{enumerate}[i)]
        \item Following the above, $T_1$ has the diagonal $((i-1),(i+1))$, while $T_2$ has the diagonals $((i-1), k), (i, k),$ and $((i+1), k)$. The result holds by choosing $D = (i, k)$ and mutating it to $(i-1,i+1)$.
        \item We can remove the edges with top label $i$ in both the $\Gamma_i$'s.
        \item The triangulations $T_i'$, corresponding to the fillings $L_i'$ of $\la(2,n-1)$ found above, are obtained respectively from $T_i$ by deleting the $i$-th vertex of the $(n+2)$-gon. It is shown in \cite[Lemma 2]{Pournin14} that there exists a minimal sequence of mutations (or flips) taking $T_2$ to $T_1$ that starts with the mutation in (i) above. Every mutation after the first step in said sequence can be identified with a mutation on $T_2'$, and conversely -- it follows that $d_\mu(L_1',L_2') + 1 = d_\mu(L_1,L_2)$.
    \end{enumerate}
\end{proof}

\begin{remark}\label{rem:bigonsandtriangles}
    If $T_1$ and $T_2$ have no diagonals in common, $\La_\Gamma(L_1,L_2)$ contains no bigons. In particular, the existence of a bigon in $\La_\Gamma(L_1,L_2)$ implies there are regions labeled $(i-1),i,(i+1)$ such that the diagonal $((i-1),(i+1))$ is included in both $T_1$ and $T_2$. Also, as we see above, every triangle in $\La_\Gamma(L_1,L_2)$ corresponds to a $K_4 \subset \La_\Gamma^*(L_1,L_2)$. 
\end{remark}

\begin{definition}
    We will call a cubic planar graph a {\em generalized cube graph} if it can be reduced to the cube graph after a sequence of edge deletions. Equivalently, a generalized cube graph is one such that it contains a planar subgraph (with the same number of vertices) where no region has three or fewer boundary edges.
\end{definition}

\begin{remark}\label{rem:chrompoly}
    Given a planar graph $\Gamma$ with $n$ vertices, the chromatic polynomial is a polynomial of degree $n$. The coefficients of $x^n$, $x^{n-1}$, and $x^{n-2}$ of the chromatic polynomial of $\Gamma$ are respectively 1, $-m$, and $\binom{m}{2}-t_1$, where $\Gamma$ has $m$ edges and $t_1$ is the number of induced $K_3$, i.e. the number of three-vertex subsets of the vertex set where every pair of vertices is connected by an edge.
\end{remark}

\begin{remark}\label{rem:decomposable_edges}
    Note that if $\La(\Gamma)$ is a cubic planar Legendrian surface corresponding to the $n$-vertex graph $\Gamma$ that is Legendrian isotopic to a connect sum of $l$ Clifford and $(k-l)$ standard tori, $P_{\Gamma^*}(x)$ would be of the form \[p_{k,l}(x) = x(x-1)(x-2)^k(x-3)^l = x^{k+l+2} - (1+2k+3l)x^{k+l+1} + \bigg(\dfrac{(2k+3l)^2-3l}{2}\bigg)x^{k+l} + O(x^{k+l-1})\] where $k \geq 1$ and $k+l = (n-2)/2$. This means, by Remark~\ref{rem:chrompoly}, the dual graph has $(k+l+2)$ vertices, $(1+2k+3l)$ edges, and $(k+3l)$ triangles. In particular, for a fixed number of vertices and edges, the number of triangles in the dual graph is fixed -- when we know it is dual to a trivalent graph representing a decomposable cubic planar Legendrian.
\end{remark}

\begin{figure}[h!]{\tiny
\begin{overpic}[width=.5\textwidth,tics=10] 
{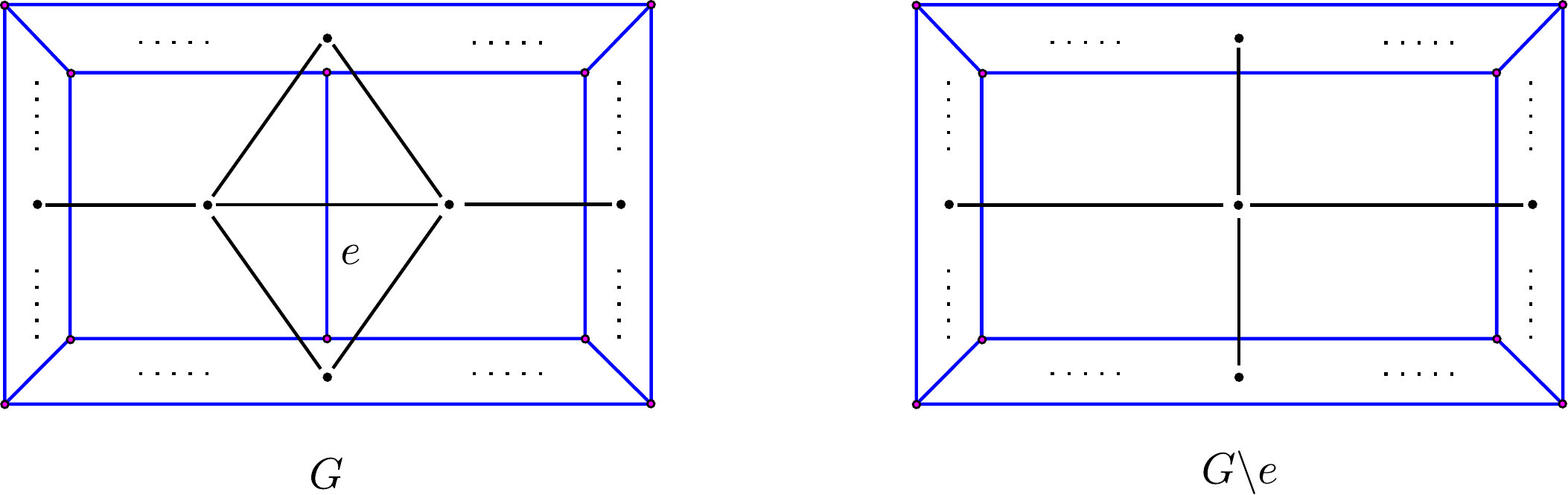} 
\put(25,28){$N$}
\put(10,21){$E$}
\put(30,15){$W$}
\put(15,7){$S$}
\put(83,28){$N$}
\put(73,7){$S$}
\put(72,21){$EW$}
\end{overpic}
}
\caption{The graphs $\Gamma$ and $\Gamma \setminus e$ and their duals.}
			\label{fig:cubes}\end{figure}

\begin{proposition}
    Let $\Gamma$ be a  generalized cube graph. The Legendrian surface $\La(\Gamma)$ is not Legendrian isotopic to a connect sum of Clifford and standard tori.
\end{proposition}

\begin{proof}
    Note that a generalized cube graph must have at least 8 vertices, and the statement is true if there are exactly 8 vertices. Suppose we have a generalized cube graph $\Gamma$ with $n>8$ vertices, and we assume there are no triangles or bigons in the graph (i.e., no regions with three or fewer boundary edges). We will prove this by induction on the number of vertices in the graph. We claim that the number of triangles in $\Gamma^*$ is always smaller than the number of triangles in the dual of a cubic planar graph (such that the dual has the same number of vertices and edges as $\Gamma^*$) which corresponds to a decomposable Legendrian. This proves the proposition, since it follows by Remark~\ref{rem:chrompoly} that the chromatic polynomial of the dual is not the same as that of the dual to the 2-weave of a decomposable Legendrian, and the chromatic polynomial of the dual is a Legendrian isotopy invariant by \cite[Theorem 4.2]{TreumannZaslow}. We now proceed with the details.

    Let $\Gamma$, $e$ be as given in Figure~\ref{fig:cubes} -- i.e., $e$ is an edge of $\Gamma$ which belongs to two regions neither of which are triangles, and $\Gamma-e$ is a generalized cube graph without triangles or bigons. Then $e$ is adjacent to four regions, let us call them $N$, $E$, $S$, and $W$, as shown in Figure~\ref{fig:cubes}. There are five corresponding edges in $\Gamma^*$: $N-E$, $E-S$, $S-W$, $W-N$, and $E-W$. These four regions become three regions  in $\Gamma-e$, namely $N$, $EW$, and $S$, and the five edges of the dual reduce to two edges, namely $N - EW$ and $S - EW$. Note that by our assumption on $e$, there is no edge between $N$ and $S$.

    Suppose $\Gamma^*$ has $m$ edges. Then $(\Gamma-e)^*$ has $m-3$ edges, and one fewer vertex than $\Gamma^*$. By Remark~$\ref{rem:chrompoly}$, the coefficient of $x^{N-1}$ in $P_{\Gamma^*}(x)$ is $-m$. Consequently, in $P_{(\Gamma-e)^*}$, the coefficient of $x^{N-2}$ is $-m+3$, and the coefficient of $x^{N-3}$ is $\binom{m-3}{2}-t_1((\Gamma-e)^*)$, where $t_1((\Gamma-e)^*)$ is the number of triangles, or induced $K_3$, in $(\Gamma-e)^*$. By the induction hypothesis, $(\Gamma-e)$ is a generalized cube graph with fewer vertices than $\Gamma$.

    Denote the possible number of triangles in the dual to the decomposable 2-weave with as many vertices and edges as $(\Gamma-e)^*$ as $T_l$ (note this number is fixed by Remark~\ref{rem:decomposable_edges}). Then we have
    \[
    t_1((\Gamma-e)^*) < T_l
    \]

     by the induction hypothesis.
    
    Consider a triangle in $(\Gamma-e)^*$. A triangle that is not adjacent to any of the edges $N-EW$ and $S-EW$ corresponds to a triangle in $\Gamma^*$. A triangle adjacent to $N - EW$ corresponds to a triangle in $\Gamma^*$, either with the edge $N-E$ or the edge $N - W$. Similarly for any triangle adjacent to $S-EW$. Further, $\Gamma^*$ has two extra triangles $N-E-W$ and $S - E - W$. It follows that:
    \[
    t_1(\Gamma^*) = t_1((\Gamma-e)^*) + 2
    \]
    However, by Remark~\ref{rem:decomposable_edges}, and the fact that the dual to $\Gamma^*$ has one more vertex and three more edges than $(\Gamma-e)^*$, the number of triangles in the dual to a decomposable graph with the same vertex and edge numbers will be $T_l+3$. The statement is thus proved. 
\end{proof}

\begin{proposition}\label{prop: decomposable}
    Let $L_1$ be a trivalent graph filling of $\la(2,n)$. For any trivalent graph filling $L_2$ of $\la(2,n)$, if $\La(L_1,L_2)$ is decomposable, $d_\mu^a(L_1,L_2) \leq n-1$.
\end{proposition}

\begin{proof}
The statement is true for $\la(2,4)$ (refer Remark~\ref{rem:t(2,4)_example}).

Assume that $\La(L_1,L_2)$ is decomposable; by Proposition~\ref{prop: decomposable}, the double graph $\La_\Gamma(L_1,L_2)$ (as in Definition~\ref{def: doublegraph}) is not a generalized cube graph. This means there is an interior region which is a triangle or a bigon. First assume that there is an interior region that is a triangle. Then, apply Lemma~\ref{lem: inductionstep} to obtain fillings $L_1'$ and $L_2'$ of $\la(2,n-1)$ with $d_\mu(L_1',L_2') = d_\mu(L_1,L_2) - 1$. It also follows by construction that $\La_\Gamma(L_1',L_2')$ is not a generalized cube graph. By induction, we can assume $d_\mu^a(L_1',L_2') \leq n-2$, which implies that $d_\mu^a(L_1,L_2) \leq n-1$.

Now, if there is no interior region which is a triangle, there has to be a bigon. Then, by Remark~\ref{rem:bigonsandtriangles}, we again obtain two triangulations $T_1'$ and $T_2'$ of the $(n+1)$-gon (corresponding to fillings $L_1'$ and $L_2'$ of $\la(2,n-1)$), by deleting the $(i-1,i,i+1)$-triangle in the $T_1$ and $T_2$ triangulations of the $(n+2)$-gon, such that the mutation distance between $T_1'$ and $T_2'$ is the same as the mutation distance between $T_1$ and $T_2$. In other words, $d_\mu^a(L_1',L_2') = d_\mu^a(L_1,L_2)$. Again, by induction, we can assume
$d_\mu^a(L_1',L_2') \leq n-2$, which gives our result.

\end{proof}

\begin{proof}[Proof of Theorem~\ref{thm: notdecomposable}]
This follows from the above since Proposition~\ref{prop: decomposable} is the contrapositive.
\end{proof}

The results here are specific to $\la(2,n)$ and trivalent graph fillings. However, we expect this phenomenon to be true in general. This motivates the following conjecture.

\begin{conjecture}
    Let $\la$ be a Legendrian link such that $\SM_1(\la)$ has a cluster structure. Let $L_1$ and $L_2$ be two exact Lagrangian fillings of $\la$. Then, there exists an $N$ such that $d_\mu (L_1,L_2) > N$ implies that $\La(L_1,L_2)$ is not decomposable.
\end{conjecture}

Moreover, there are examples of Legendrian links $\la$ equipped with a Legendrian loop $\varphi$  where one can provide a lower bound for algebraic mutation distance between fillings $L$ and $\varphi^\ell(L)$ for $\ell\in \N$.

\begin{theorem}
Let $L$ be an exact Lagrangian filling of $\la(k, n)$ for $k\geq 3, n\geq 6$ or $(k, n)=(4, 4)$, $(4, 5),$ or $(5, 5)$ and $\varphi$ be the Legendrian loop $\Sigma_i$ defined in \cite[Section 5.3.3]{Hughes2024}.  The mutation distance $d_\mu(L, \varphi^\ell(L))$ is at least $\ell$.  
\end{theorem}

\begin{proof}
    Let $L$ be a filling of $\la(k,n)$. From \cite[Section 7.2.1]{Hughes2024}, we know that the induced action of $\varphi$ fixes cluster variables of the seed induced by $L$. By the non-leaving-face property for cluster exchange graphs, any geodesic path  of mutations from the cluster seed ${\bf x}_L$ to the cluster seed ${\bf x}_{\varphi^k(L)}$ does not mutate at any of these fixed cluster variables. Therefore, the length of any geodesic path of mutations between these two seeds is the same as that of a geodesic path of mutations in the cluster algebra obtained by deleting the fixed vertices. In \cite[Section 7]{Hughes2024}, deleting these fixed vertices is shown to produce a surface-type quiver where the number of boundary components can be determined from $k$ and $n$, but is greater than one for any choice of $k$ and $n$ given above; the action of $\varphi$ corresponds to rotation of the boundary. Given an arc $\gamma$ in the triangulation, the winding number of $\gamma$ about any of the other boundary components of the surface provides a lower bound for mutation distance between $L$ and $\varphi^\ell(L)$, as a single mutation can reduce this winding number by at most one.    
\end{proof}

As a result, verifying Conjecture~\ref{conj: Mutation distance invariant} would prove the existence of an infinite family of doubles of the same underlying Legendrian knot.

\section{Twist-spun Legendrian tori} \label{sec: twist-spuns}

In this section, we describe a cluster structure on the sheaf moduli of twist-spun Legendrians $\Sigma_\varphi(\la)$ satisfying specific hypotheses on both the Legendrian link $\la$ as well as the loop $\varphi$. We start by giving a recipe for constructing exact Lagrangian fillings of twist-spuns. After describing the sheaf moduli of twist-spuns, and the toric charts induced by their fillings, we then discuss the process of folding cluster algebras, as well as the related idea of a skew-symmetrizable cluster ensemble. We conclude the section with a proof of Theorem~\ref{thm: intro_ensembles} and a collection of examples.

Given a Legendrian loop $\varphi$ of $\la$, we obtain an action of $\varphi$ on the set of fillings of $\la$ via concatenation. In particular, we say that a filling $L$ of $\la$ is fixed by the action of $\varphi$ if the concatenation of $L$ with the trace $tr(\varphi)$ is Hamiltonian isotopic to $L$. For fillings obtained as the Lagrangian projection of a Legendrian weave $\La$, we can verify the property of being fixed by the action of $\varphi$ by ensuring that there is a Legendrian isotopy between the concatenation $tr(\varphi)\cup \La$ and the weave $\La$ that does not introduce Reeb chords, ensuring that the Legendrian isotopy covers an exact Lagrangian isotopy of embedded surfaces in the Lagrangian projection. We can construct an exact Lagrangian filling of the twist spun Legendrian torus $\Sigma_\varphi(\la)$ via exact Lagrangian fillings of $\la$ fixed by the action of $\varphi$ in the following manner.

\begin{definition}\label{def:twist_spun_filling}
    For $\theta \in [0,1]$, denote by $\varphi_{\leq \theta}$ the restriction of the Legendrian isotopy defining $\Sigma_\varphi(\la)$ to the domain $\R^3\times [0, \theta]$ and define $L_\theta=L\cup_\la tr(\varphi_{\leq \theta})$ as the concatenation of $L$ with the trace of $\varphi_{\leq \theta}$. Since $L$ is fixed by the action of $\varphi$, this yields an $S^1$ family of exact Lagrangians $\{L_\theta\}\subseteq (\R_t\times \R_z \times T^*\R_x,d(e^t\alpha_{\st}))$. In the symplectization of $\R_z\times T^*\R_x\times T^*S^1$, we can then define 
    \[
    L \times_\varphi S^1 \subset \R_t \times \R_z\times T^*\R_x\times T^*S^1
    \]
    as $(L_{\theta},\theta)$, where $L_\theta$ lives in the $(t,z,p_x,x)$ coordinates and $\theta$ refers to the 0-section coordinates in $T^*S^1$. We thus obtain an exact Lagrangian 3-manifold $L\times_\varphi S^1$ in the symplectization 
$(\R_t\times \R^5, d(e^t\alpha_{\std}))$. By construction, this 3-manifold has boundary $\partial (L\times_\varphi S^1)\cong \Sigma_\varphi(\la)$. 

\end{definition}  

\begin{proposition}\label{prop: twist-spun fillings}
Given an exact Lagrangian filling $L$ and a Legendrian loop $\varphi$ of $\la$ such that $\varphi$ fixes $L$, the 3-manifold $L\times_\varphi S^1$ is an exact Lagrangian filling of $\Sigma_\varphi(\la)$.
\end{proposition}

Proposition~\ref{prop: twist-spun fillings} allows us to construct exact Lagrangian fillings of a Legendrian twist-spun $\Sigma_\varphi(\la)$ by understanding which exact Lagrangian fillings of $\la$ are fixed under the action of $\varphi$. To distinguish these exact Lagrangian fillings of $\Sigma_\varphi(\la)$, we consider the sheaf moduli $\SM_1(\Sigma_\varphi(\la))$.

\subsection{Sheaf invariants of twist-spun Legendrians}\label{sub: twist-spun sheaves}

We describe here how to compute the (decorated) sheaf moduli of twist-spun Legendrian tori and obtain toric charts from local systems on fillings of these twist-spuns.

We start by describing the decorated sheaf moduli $\FM(\Sigma_\varphi(\la))$ for the twist-spun $\Sigma_\varphi(\la)$. Recall from Subsection~\ref{sub: decorated sheaf moduli} that we define $T=\{t_1,\dots, t_k\}$ to be a set of marked points on $\la$ with every component of $\la$ carrying at least one marked point. Let $\tau=\{t_1\times_\varphi S^1, \dots, t_k\times_\varphi S^1 \}$ be an $S^1$ spun collection of marked points $T$ on $\la.$ Similar to our discussion in Section~\ref{sub: decorated sheaf moduli}, we can also consider a decorated version of the sheaf moduli $\mathfrak{M}(\Sigma_\varphi(\la), \tau)$, where local systems on $\Sigma_\varphi(\la)\backslash \tau$ are given by the monodromy of $\varphi$ in the $S^1$ direction and a trivialization in the radial direction, yielding the desired framing data. Analogous to our earlier definition of $\FM(\la, T)$, we define $\FM(\Sigma_\varphi(\la), \tau)$ as $\SM_1(\Sigma_\varphi(\la))$ along with the data of these trivializations in the radial direction. While every result concerning $\SM_1(\Sigma_\varphi(\la), \tau)$ in this subsection will have an analogous statement in terms of $\FM(\Sigma_\varphi(\la), \tau)$, we will generally state these results in terms of $\SM_1(\Sigma_\varphi(\la))$ for the sake of clarity.

Given a Legendrian loop $\varphi,$ we denote by $G$ its induced group action on $\SM_1(\la, T)$ and by $\SM_1(\la, T)^G$ the $G$-invariant subset of $\SM_1(\la, T)$.

\begin{proposition}\label{prop: twist spun sheaf moduli}
    Let $\varphi$ be a Legendrian loop of a Legendrian link $\la$. Then $\SM_1(\Sigma_\varphi(\la), \tau) \cong\SM_1(\la, T)^G\times \C^\ast$ and $\FM(\Sigma_\varphi(\la), \tau)\cong \FM(\la, T)^G\times \C^\ast$. 
\end{proposition}

\begin{proof}

Consider a front projection $\Pi(\Sigma_\varphi(\la))\subseteq \R^2\times \S^1_\theta$ obtained as an $\S^1$ family of fronts of $\varphi_t(\la)\subseteq \R^2 \times \{\theta\}$. Parametrize this family of fronts as $\la_\theta,$ $\theta\in [0, 2\pi],$ and identify $\la_0$ with $\varphi(\la_0)\cong \la_{2\pi}$. If we consider two subsets $\la_{(0, 2\pi)}:=\{\la_\theta | \theta\in (0, 2\pi)\}$ and $\la_{2\pi}$, then we obtain a stratification of $\Pi(\Sigma_\varphi(\la))$, inherited from the stratification of $\Pi(\la)\subseteq \R^2$ given by dividing the front projection into cusps, crossings, and edges as in \cite[Section 3.3]{STZ_ConstrSheaves}. Since $\Pi(\Sigma_\varphi(\la))$ is locally of the form $\Pi(\la)\times [-\epsilon, \epsilon]$, the singular support conditions for $\Sigma_\varphi(\la)$ can be identified with those for $\la$ with the possible addition of conditions coming from the gluing of $\la_0$ to $\la_{2\pi}$. As a result, any sheaf $\mathcal{F}\in \mathcal{M}_1(\Sigma_\varphi(\la))$ can be identified with a sheaf $\mathcal{F}'\subseteq \mathcal{M}_1(\la)$ possibly subject to some additional constraints coming from $\varphi$.

   Given a sheaf $\mathcal{F}\subseteq Sh(\R^2\times \S^1)$, the additional singular support conditions on $\mathcal{F}$ come from the requirement that $SS(\mathcal{F})\subseteq \Sigma_\varphi(\la)$. Denote by $\mathcal{F}_{x, t}$ the stalk of $\mathcal{F}$ at $(x, t)\subseteq \R^2\times \S^1$. In order to satisfy the singular support conditions necessary for $\mathcal{F}$ to be an element of $\mathcal{M}_1(\Sigma_\varphi(\la))$, we need for $\mathcal{F}$ to propagate in the $\S^1$ (co)direction, i.e. we need an isomorphism $\mathcal{F}_{x, t-\epsilon}\cong \mathcal{F}_{x, t}$ of stalks at points $(x, t-\epsilon)$ and $(x, t)$. When $t=1$, this is precisely the condition that $\mathcal{F}_x\cong \mathcal{F}_{\varphi(x)}$. Therefore, as desired, we can identify $\mathcal{M}_1(\Sigma_\varphi(\la))$ with the $G$-fixed locus of $\mathcal{M}_1(\la)$ with an additional $\C^\ast$ parameter coming from the monodromy in the $S^1$ direction. The case of $\FM(\Sigma_\varphi, \tau)$ follows analogously. 

\end{proof}

An exact Lagrangian filling $L$ of $\la$ that is fixed by the action of $\varphi$ together with a choice of local system induces a toric chart $\mathcal{T}_L$ that is fixed by the $G$-action on $\SM_1(\la, T).$ Following our computation in the proof of Proposition~\ref{prop: twist spun sheaf moduli}, fillings of the twist-spun of the form $L \times_\varphi S^1$ therefore can be identified with toric charts inside of $\SM_1(\Sigma_\varphi(\la))$. 

We can realize these toric charts explicitly as local systems on $L\times_\varphi S^1$ in the following way. First, we compute $H_1(L\times_\varphi S^1, \tau)$ via the long exact sequence of homology of the mapping torus.

\begin{equation*}
    \dots \rightarrow H_n(L, T)\xrightarrow{1-\varphi_*} H_n(L,  T) \rightarrow H_n(L\times_\varphi S^1, \tau)\rightarrow H_{n-1}(L, T) \rightarrow \dots
\end{equation*}

From the long exact sequence, we have an injective map $H_2(L\times_\varphi S^1, \tau) \hookrightarrow H_1(L, T)$ and the image is isomorphic to the kernel of the map $1-\varphi_*$. In particular, all elements of $H_2(L\times_\varphi S^1, \tau)$ can be identified with elements of $H_1(L, T)$ that are fixed by $\varphi_*$. Denote by $G$ the subgroup of $\Aut(H_1(L, T))$ generated by $\varphi_*$. We can present a basis $\{I_1, \dots, I_m\}
$ of these cycles by considering $G$-orbits of cycles in $H_1(L)$ 
and summing all elements in each orbit. Similarly, $H_1(L\times_\varphi S^1, \tau)$ can be identified with $G$-orbits of elements in $H_1(L,T)$ with the addition of a homology class represented by $\{x\}\times_\varphi S^1$ coming from injectivity of the map $H_1(L\times_\varphi S^1, \tau)\to H_0(L, T)$.

Armed with this description of $H_1(L\times_\varphi S^1, \tau)$, we can understand abelian local systems on $L\times_\varphi S^1$. Denote by $\Phi: \SM_1(\la, T)^G\to \SM_1(\Sigma_\varphi(\la), \tau)$ the isomorphism between sheaf moduli give by Proposition~\ref{prop: twist spun sheaf moduli}. 

\begin{lemma}\label{lem: local_systems_on_spuns}
 An embedded exact Lagrangian filling of the form $L\times_\varphi S^1$ of $\Sigma_\varphi(\la)$ induces a toric chart $\mathcal{T}_{L\times_\varphi S^1}$ in $\SM_1(\Sigma_\varphi(\la), \tau)$ that can be identified with $\Phi(\mathcal{T}_L)\times \C^\ast$. 
\end{lemma}

\begin{proof}
The singular support conditions corresponding to (the Legendrian lift of) $L\times_\varphi S^1$ can be determined following the proof of Proposition~\ref{prop: twist spun sheaf moduli}. In particular, since $L$ is $\varphi$-invariant and $L\times_\varphi S^1$ locally resembles $L\times [a,b]$, and the only additional singular support conditions beyond those imposed by the strata $L\times\{\theta\}$ are that stalks are fixed by $\varphi_*$. Following the description of $H_1(L\times_\varphi S^1, \tau)$ above, since $L$ is $\varphi$-invariant, the image of the local system $\mathcal{T}_L$ is $G$-invariant, and we can identify a local system on $L$ directly with a local system on $L\times_\varphi S^1$ with the addition of the $\C^\ast$ factor from the monodromy in the $S^1$ direction. 
\end{proof}


\subsection{Obtaining cluster structures on $\FM(\Sigma_\varphi(\la))$ via folding}\label{sub: clusters from folding}

In this subsection we obtain a folded cluster algebra from the sheaf moduli of certain twist-spun Legendrians. Let $\varphi$ be a Legendrian loop of a Legendrian link $\la$ with induced action $\varphi_*$ on $\FM_1(\la, T)$. Recall that we denote the $S^1$-spun collection of marked points $T\times_\varphi S^1$ by $\tau$. If $\la$ is braid positive, then the moduli $\SM_1(\Sigma_\varphi(\la), \tau)$ and $\FM(\Sigma_\varphi(\la), \tau)$ can both be identified with collections of $\varphi_*$-invariant flags satisfying the transversality conditions described in Section~\ref{sec: sheaves background} together with the possible addition of $\varphi_*$-invariant framing data coming from $\tau$. We show below that for specific $\varphi_*$-actions, $\SM_1(\Sigma_\varphi(\la), \tau)$ and $\FM(\Sigma_\varphi(\la), \tau)$ admit skew-symmetrizable cluster structures with contact-geometric descriptions.

\subsubsection{Character lattices from fillings of $\Sigma_\varphi(\la)$}\label{sub: lattices}

Let $\la$ be a Legendrian link such that the ring of regular functions $\C[\FM(\la, T)]$ is a globally foldable cluster algebra. To understand the character lattices associated to an exact Lagrangian filling $L\times_\varphi S^1$ of $\Sigma_\varphi(\la)$, we make use of the Poincar\'e pairing for 3-manifolds. Denote by $N$ the homology lattice $H_2(L\times_\varphi S^1, \tau)$ and $M^\circ$ the lattice $H_2((L\times_\varphi S^1)\backslash \tau, \Sigma_\varphi(\la)\backslash \tau)$.


Recall from Subsection~\ref{sub: twist-spun sheaves} that we can obtain a basis for $H_2(L\times_\varphi S^1)$ from the $G$-orbits of a basis of homology cycles in $H_1(L, T)$. We denote these $G$-orbits by $I_1, \dots, I_k$. Given a cycle $\alpha\in H_2(L\times_\varphi S^1, \tau)$, we can therefore identify an orbit representative $\gamma_\alpha\in H_1(L, T)$ whose $G$-orbit is homologous to $\alpha$. Moreover, from the long exact sequence of the mapping torus, we know that the image of $\tilde\gamma_\alpha$ gives a well-defined class in $H_1(L\times_\varphi S^1, \tau)$. We can then define a (skew-symmetrizable) bilinear form $[ \alpha_1, \alpha_2]$ on $N=H_2(L\times_\varphi S^1)$ to be the intersection pairing between $\alpha_2$ and $\tilde{\gamma}_{\alpha_1}$. We similarly define a skew-symmetrizable bilinear form $[\alpha_1, \alpha_2]$ on $M^\circ=H_2((L\times_\varphi S^1)\backslash \tau, (\la\times_\varphi S^1)\backslash \tau)$ by Poincar\'e-Lefschetz duality. Note that in general, we need not have $[ \alpha_2, \alpha_1] =[ \alpha_1, \alpha_2]$, giving rise to the skew-symmetrizable, as opposed to skew-symmetric, nature of the bilinear form.

\begin{lemma}\label{lem: poincare=folding}
    The bilinear form $[ \cdot, \cdot ]$ defined above is well defined and satisfies
    $[\alpha_1, \alpha_2]=\sum_{\gamma_{\alpha_1}\in I_{\alpha_1}} \langle \gamma_{\alpha_1}, \gamma_{\alpha_2}\rangle$ for $\gamma_{\alpha_2}$ any orbit representative of the $G$-orbit $\alpha_2$.
\end{lemma}

\begin{proof}
Since $L\times_\varphi S^1$ is smoothly the mapping torus of $L$ under the action of $\varphi$, selecting any other orbit representative $\gamma_{\alpha'}$ of $\alpha$ necessarily gives a  cycle $\tilde{\gamma}_{\alpha'}$ homologous to $\tilde{\gamma}_\alpha\in H_1(L\times_\varphi S^1, \tau)$ ensuring the bilinear form is well defined. Moreover, for a fixed $\gamma_{\alpha_2}$, the formula  $\sum_{\gamma_{\alpha_1}\in I_{\alpha_1}} \langle \gamma_{\alpha_1}, \gamma_{\alpha_2}\rangle$ is precisely the intersection number between $\alpha_1$ and $\tilde{\gamma}_{\alpha_2}$. 
\end{proof}

\begin{remark}
While the dual lattice formalism of cluster ensembles requires that we define our bilinear forms on $H_2(L\times_\varphi S^1, \tau)$ and $H_2(L\times_\varphi S^1\backslash \tau, \Sigma_\varphi(\la)\backslash\tau)$, one can produce quivers associated to $L\times_\varphi S^1$ by identifying second homology classes of $L\times_\varphi S^1$ with orbits of (relative) homology classes of $L$ and using the formula from Lemma~\ref{lem: poincare=folding} to compute arrows between vertices, as in the following example. 
\end{remark}


 
\begin{ex}
Consider $\la=\la(2, 6)$ and the filling $\pi(\La)$ given by the weave $\La$ pictured in Figure~\ref{fig: FoldingExample} (left). The intersection form of the associated filling yields an $A_5$ quiver with eight frozen variables.  
The action of $\rho^4$ -- rotation by $\pi$ radians -- fixes $\La$, and therefore $\pi(\La)\times_{\rho^4}S^1$ is an exact Lagrangian filling of the twist-spun $\Sigma_{\rho^4}(\la(2, 6))$. Label the mutable quiver vertices from top to bottom by $\gamma_1,\dots \gamma_5,$ and the frozen quiver vertices by $\gamma_6, \dots, \gamma_{13}$. If we fold according to the $\rho^4$ action, we can obtain the weighted quiver given in Figure~\ref{fig: FoldingExample} (right) where the only nontrivial weights are assigned to 
incoming (2, 1) and outgoing (1, 2) arrows of the quiver vertex corresponding to $\gamma_3$. If we denote by $A_i$ the cluster variable associated to $\gamma_i^\vee$ or, by an abuse of notation, its $\rho^4$ orbit, then the weighted quiver tell us that mutation at $\gamma_3$ in the folded cluster algebra can be computed as 
$A_3\mu(A_3)=A_2^2+A_9^2$

Alternatively, we can choose a set of $\rho^4$ orbits $\gamma_1+\gamma_5, \gamma_2+\gamma_4, \gamma_3$ of mutable vertices and $\gamma_6+\gamma_{13}, \gamma_7+\gamma_{12}, \gamma_{8}+\gamma_{11}, \gamma_9+\gamma_{10}$ of frozen vertices that correspond to a basis for the homology lattice $N=H_2(L_\w\times S^1,\tau)\cong \Z^7$.  
If we denote by $\gamma^\vee \in H_1(L\backslash T, \la\backslash T)$ the dual relative cycle of $\gamma \in H_1(L, T)$, then $N$ has a dual lattice $M=H_2(L_\w\times S^1\backslash\tau,\Sigma_\varphi(\la)\backslash\tau)\cong \Z^7$ with basis corresponding to the $\rho^4$ orbits \[\frac{1}{2}(\gamma_1^\vee+\gamma_5^\vee), \frac{1}{2}(\gamma_2^\vee+\gamma_4^\vee), \gamma_3^\vee, \text{ and } \left\{\frac{1}{2}(\gamma_i^\vee+\gamma_{19-i}^\vee)\right\}_{i=6}^9\] and a sublattice $M^\circ$ with basis corresponding to\[(\gamma_1^\vee+\gamma_5^\vee),(\gamma_2^\vee+\gamma_4^\vee), \gamma_3^\vee, \left\{\gamma_i^\vee+\gamma_{19-i}^\vee\right\}_{i=6}^9.\] 
Similarly, $N$ is a sublattice of $N^\circ$, the dual lattice to $M^\circ$ with basis corresponding to
\[\frac{1}{2}(\gamma_1+\gamma_5),\frac{1}{2}(\gamma_2+\gamma_4), \gamma_3, \left\{\frac{1}{2}(\gamma_i+\gamma_{19-i})\right\}_{i=6}^9.\] 
The bilinear forms  $[\cdot, \cdot]$ for $N$ and $M^\circ$ are given by the matrix 
\[\begin{bmatrix}
\phantom{-}0 & -1 & \phantom{-}0 & -1 & \phantom{-}1 & \phantom{-}1 &\phantom{-}0\\
\phantom{-}1 & \phantom{-}0 & -2 & \phantom{-}0 & \phantom{-}0 & -1 &\phantom{-}1\\
\phantom{-}0 & \phantom{-}1 & \phantom{-}0 & \phantom{-}0 & \phantom{-}0 & \phantom{-}0 &-1\\
\phantom{-}1 & \phantom{-}0 & \phantom{-}0 & \phantom{-}0 & \phantom{-}0 & \phantom{-}0 &\phantom{-}0\\
-1 & \phantom{-}0 & \phantom{-}0 & \phantom{-}0 & \phantom{-}0 & \phantom{-}0 &\phantom{-}0\\
-1 & \phantom{-}1 & \phantom{-}0 & \phantom{-}0 & \phantom{-}0 & \phantom{-}0 &\phantom{-}0\\
\phantom{-}0 & -1 & \phantom{-}2 & \phantom{-}0 & \phantom{-}0 & \phantom{-}0 &\phantom{-}0
\end{bmatrix}\]
\end{ex}

\begin{figure}[h!]{ \includegraphics[width=.6\textwidth]{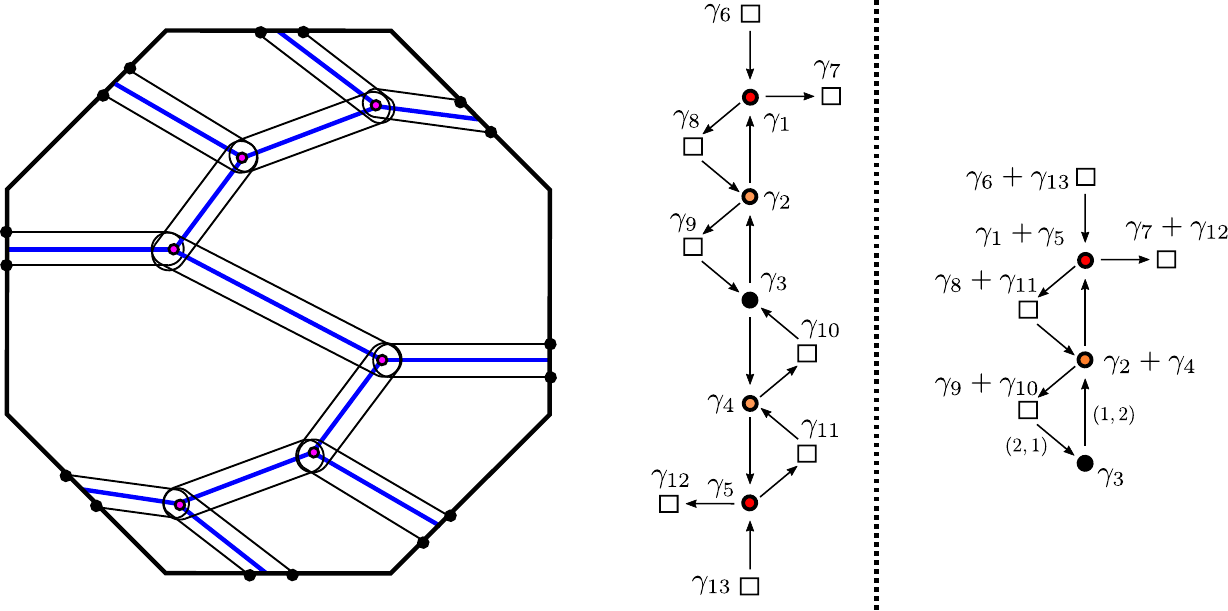}}\caption{Legendrian weave filling of $\la(2, 6)$ and its intersection quiver (left) and the quiver obtained from folding by the action of $\rho^4$ (right). Mutable quiver vertices are colored according to which $\rho^4$ orbits they belong to.}
			\label{fig: FoldingExample}\end{figure}

Given $\gamma\in H_1(L, T)$, denote by $T_\varphi(\gamma)$ the orbit of $\gamma$ under the action of the Legendrian loop $\varphi$. Since $L$ is fixed by the action of $\varphi$, the orbit $T_\varphi(\gamma)$ is a torus and represents a class in $H_2(L\times_\varphi S^1, \tau)$ following our discussion above. To $T_\varphi(\gamma)$, we associate a $\C^\ast$ coordinate via the following definition.

\begin{definition}\label{def: twist_monodromy}
  Let $T_\varphi(\gamma)\in H_2(L\times_\varphi S^1, \tau)$ be a 2-cycle obtained as the $\varphi$ orbit of $\gamma\in  H_1(L, T)$. We define the rational function $X_{T_\varphi(\gamma)}\in \C(\FM(\Sigma_\varphi(\la), \tau))$ to be
  $$X_{T_\varphi(\gamma)}=\sum_{i=0 }^k X_{\gamma_i}$$
  where $\{\gamma_0 \coloneqq \gamma, \dots, \gamma_k \coloneqq \varphi^k(\gamma)\}\subseteq H_1(L,T)$ is the $\varphi$ orbit of $\gamma$, i.e. $(k+1)$ is the minimum positive integer so that $\varphi^{k+1}(\gamma) = \gamma$.
\end{definition}

Note that the $\varphi_*$-invariant condition above implies that $X_{T_\varphi(\gamma)}$ is well-defined as an element of $\C[\SM_1(\Sigma_\varphi(\la), \tau)]$.

Now, let  $T_\varphi(\gamma^\vee)\in H_2(L\times_\varphi S^1, \la \times_\varphi S^1)$ be a relative homology cycle of the form  $\gamma^\vee\times_\varphi S^1$ where $\gamma^\vee\in H_1(L\backslash T, \la \backslash T)$. To define the microlocal merodromy along $T_\varphi(\gamma^\vee)$, we proceed analogously to above.

\begin{definition}\label{def: twist_merodromy}
  The regular function $A_{T_\varphi(\gamma^\vee)}\in \C[FM(\Sigma_\varphi(\la),\tau)]$ is the function $$A_{T_\varphi(\gamma^\vee)}=\sum_{i=0}^k A_{\gamma^\vee}$$ where $\{\gamma_0, \dots, \gamma_k\}$ is the $\varphi$-orbit of $\gamma$. 
\end{definition}

\begin{remark}
    The function $A_{T_\varphi(\gamma^\vee)}$ can geometrically be recognized, as a parallel transport map analogous to Definitions~\ref{def: merodromy}, as follows. Consider a decorated version of the sheaf moduli $\mathfrak{M}(\Sigma_\varphi(\la), \tau)$ as described in Section~\ref{sub: twist-spun sheaves}. Then, $T_{\varphi(\gamma^\vee)}$ contains the relative 1-cycle $\gamma^\vee$ via the inclusion $\la \subset \Sigma_\varphi(\la)$, and the framing data of $\mathfrak{M}(\Sigma_\varphi(\la), \tau)$ specifies non-zero vectors at the end-points of $\gamma^\vee$. Then $A_{T_\varphi(\gamma^\vee)}$ measures the ratio of the parallel transport to the vector specified by the framing data, as in Definition~\ref{def: merodromy}.
\end{remark}

Let $\la$ be any Legendrian for which $\C[\FM(\la, T)]$ is a cluster algebra and suppose that $\la$ admits a Legendrian loop $\varphi$ inducing a $G$-action on $\FM(\la, T),$. Since the cluster variables generate $\C[\FM(\la, T)]$ as a $\C$-algebra, there is a map $\C[\FM(\la, T)] \to \C[\FM(\Sigma_\varphi(\la), \tau)]$ of coordinate rings given by $A_{\gamma^\vee}\mapsto A_{T_\varphi(\gamma^\vee)}$. This map allows us to define the information of a folded seed from the cluster seed defined by an initial filling $L$ and its intersection quiver, assuming that the quiver is $G$-admissible. In what follows below, we show that for $G$-admissible quivers arising from braid positive Legendrians, we can realize the quiver mutations geometrically, providing the last elements needed for our proof of Theorem~\ref{thm: intro_ensembles}.

\subsubsection{Mutation of twist-spun Legendrians}

We say that a collection $K=\{\gamma_i\}_{i=1}^n$ of $\mathbb{L}$-compressing cycles are simultaneously mutable if mutations at any two cycles $\gamma_i, \gamma_{i'}$ in the collection pairwise commute. That is $\mu_{\gamma_i}\circ \mu_{\gamma_{i'}}(L)$ is Hamiltonian isotopic to $\mu_{\gamma_{i'}}\circ \mu_{\gamma_{i}}(L)$. When we have such a collection of cycles, we denote by $\mu_K$  the sequence of mutations at every $\mathbb{L}$-compressing cycle in $K$ in any order.

\begin{lemma}\label{lem: G-mutation}
   Let $G$ be a finite group generated by the induced action of a Legendrian loop $\varphi$ on $\SM_1(\la(\beta), T)$ and $L$ be an exact Lagrangian filling of a braid positive Legendrian link $\la(\beta)$ that admits a $G$-admissible intersection quiver and is fixed by the action of $\varphi$. Then all $\mathbb{L}$-compressing cycles in $H_1(L)$ belonging to the same mutable $G$-orbit $K$ are simultaneously mutable and $\varphi(\mu_K(L))\cong \mu_K(L)$.
      
\end{lemma}

\begin{proof}
Let $L$ be a $\varphi$-invariant filling of $\la(\beta)$ with a $G$-admissible intersection quiver. The $G$-admissible condition implies that any two cycles $\gamma_i$ and $\gamma_{i'}$ belonging to the same $G$-orbit have algebraic intersection number $\langle \gamma_i, \gamma_{i'}\rangle=0$. By \cite[Theorem 1.1]{CasalsGao24}, there is a Hamiltonian isotopy of $L$ so that the geometric intersection of cycles in $H_1(L)$ matches the algebraic intersection, ensuring that $\mu_{\gamma_i}(L)$ fixes a neighborhood of $\gamma_{i'}$ and vice versa. As a result, we can successively perform the Lagrangian disk surgeries $\mu_{\gamma_i}$ and $\mu_{\gamma_{i'}}$, and either ordering produces Hamiltonian isotopic fillings of $\la(\beta)$.

We also observe that the Hamiltonian isotopy of $L$ given in \cite{CasalsGao24} to ensure that geometric intersections match algebraic intersections is, at each step, a local move. In particular, this allows us to perform these local Hamiltonian isotopies in a neighborhood of any $\gamma_i$ in the same $G$-orbit in order to produce a $G$-invariant Hamiltonian isotopy. The result of mutating at each of the cycles in the $G$-orbit then produces another $\varphi$-invariant filling, by construction.  
\end{proof}

\begin{remark}
    Note that while mutation at every cycle in a $G$-orbit produces another $\varphi$-invariant filling, it does not necessarily preserve $G$-admissibility. As explained in Subsection~\ref{sub: folding}, $G$-admissibility is only preserved under an arbitrary sequence of mutations of $G$-orbits if the cluster algebra is globally foldable with respect to the $G$-action.
\end{remark}

We relate a $G$-orbit of $\mathbb{L}$-compressing cycles of $L$ to a solid mutation configuration (as in Section~\ref{sec: lag_surgery}) in a filling for a twist-spun Legendrian torus via the following lemma. Let $D_\gamma$ be an $\mathbb{L}$-compressing disk for $L$ with boundary $\gamma$. As explained above, since $L$ is fixed by the action of $\varphi$, the set $T_\varphi(\gamma)=\{(\gamma_\theta, \theta)\}_{\theta=0}^{2\pi|K|}$ is a torus representing a class in $H_2(L\times_\varphi S^1, \tau)$. Here $|K|$ denotes the size of the orbit of $\gamma$ under the induced action of $\varphi$ on homology. Since $\gamma$ is an $\L$-compressing cycle, $T_\varphi(\gamma)$ bounds a solid Lagrangian torus $S_\varphi(D_\gamma)=\{(D_\gamma, \theta)\}_{\theta=0}^{2\pi|K|}$ obtained as the orbit of $D_\gamma$ under the $\varphi$ action.

\begin{lemma}\label{lem: solid mutation}
             Let $G, \varphi,$ and $L$ be as in the statement of Lemma~\ref{lem: G-mutation}. For any $\mathbb{L}$-compressing cycle $\gamma \in H_1(L, T)$ with Lagrangian compressing disk $D_\gamma$, the pair $(L\times_\varphi S^1, S_\varphi(D_\gamma))$ is a generalized solid mutation configuration.
\end{lemma}

\begin{proof}
    Following the notation in Definition~\ref{def:twist_spun_filling}, for every $\theta \in S^1$, there is a simple closed curve $\gamma_\theta \subset L_{\theta}$ which bounds a transverse Lagrangian disk $D_{\gamma_\theta}$. As described above, the solid torus $S_\varphi(D_\gamma)$ can be parametrized as $(D_{\gamma_\theta},\theta)$ for $0\leq \theta\leq 2\pi|K|$. Since each $D_{\gamma_\theta}$ meets $L_\theta$ transversely, the assertion follows.
   
\end{proof}

\begin{remark}
We comment here on a slight subtlety regarding the order of $G$ as a group action on $\SM_1(\la, T)$ and the order of $G$ as a cluster automorphism. Typically, cluster automorphisms -- automorphisms of a cluster algebra that commute with mutation -- are considered up to a permutation of the cluster variables in a seed. However, in order to ensure that we can glue $S^1$-families of cycles together in the mapping torus $L\times_\varphi S^1$, we need to consider the entire orbit of a cycle $\gamma$. That is, any non-identity permutation of homology cycles in an orbit, even if it results in the same cluster seed, does not yield a torus (or a homology class) in $L\times_\varphi S^1$, as it does not comprise the full orbit of $\gamma$. This means that $|K|$ is in general a positive integer multiple of the order of the cluster automorphism of $\SM_1(\la, T)$ induced by $\varphi$. In each of the specific cases we consider in this work, the order of $G$ as a cluster automorphism is equal to its order as an action on $H_1(L, T)$, as $G$ acts by rotation on $L$.
\end{remark}

Denote by $Q_L^G$ the weighted quiver obtained by folding a $G$-admissible intersection quiver $Q_L$ of a filling $L$ by the group $G$. The observation below follows immediately from Lemma~\ref{lem: solid mutation} and our description of the homology basis for $H_2(L\times_\varphi S^1, \tau)$ above.

\begin{corollary}\label{cor: SMCs}
    Any mutable vertex of $Q_L^G$ is represented by a solid mutation configuration whose boundary $T_\varphi(\gamma)$ is a class in $H_2(L\times_\varphi S^1, \tau)$.
\end{corollary}

If we consider the quiver $Q_L^G$, we can justify our definition of microlocal monodromies and merodromies for cycles $T_\varphi(\gamma)$ and $T_\varphi(\gamma^\vee)$ by the following observation. As above, we assume that $L$ is a $\varphi$-invariant exact Lagrangian filling of $\la$ and $\SM_1(\la, T)$ is a globally foldable cluster algebra with respect to the $G$-action induced by $\varphi.$

\begin{proposition}\label{prop: solid_cluster_mutation}
Let $\gamma\in H_1(L, T)$ be an $\L$-compressing cycle of $L$. Then surgery of $L\times_\varphi S^1$ along the solid Lagrangian torus $S_\varphi(D_\gamma)$ induces a cluster mutation of $A_{T_\varphi(\gamma^\vee)}$ and $X_{T_\varphi(\gamma)}$.
\end{proposition}

\begin{proof}
    By construction, $A_{T_\varphi(\gamma^\vee)}$ and $X_{T_\varphi(\gamma)}$ can be identified with the image of $A_{\gamma^\vee}$ and $X_{\gamma}$ in $\FM(L\times_\varphi S^1, \tau)$ and $\SM_1(L\times_\varphi S^1, \tau)$, respectively. Since $\SM_1(\la, T)$ is globally foldable cluster algebra with respect to $G$, surgery of $L\times_\varphi S^1$, using Lemma~\ref{lem: lag_surgery}, along $S_\varphi(\gamma)$ yields a filling of $\Sigma_\varphi(\la)$ Hamiltonian isotopic to $\mu_I(L)\times_\varphi S^1$ where $\mu_I$ denotes the result of mutating at every cycle in the $G$-orbit of $\gamma$. The corresponding functions $A_{\mu_I(\gamma)^\vee}$ and $X_{\mu_I(\gamma)}$ are related to $A_{\gamma^\vee}$ and $X_{\gamma}$ by cluster mutation, following \cite{CasalsWeng}, and therefore, their images $ A_{T_\varphi(\mu_I(\gamma)^\vee)}$ and $X_{T_\varphi(\mu_I(\gamma))}$ are related to $A_{T_\varphi(\gamma^\vee)}$ and $X_{T_\varphi(\gamma)}$ by cluster mutation. 
\end{proof}

\subsubsection{Proof of theorem}
We now have all the necessary ingredients to prove Theorem~\ref{thm: intro_ensembles}, restated here for clarity.

\begin{theorem}\label{thm: ensembles}
Let $G$ be a finite group generated by the action of a Legendrian loop $\varphi$ of a braid positive Legendrian link $\la$ and suppose that $C[\FM(\la, T)]$ is a globally foldable cluster algebra with respect to the $G$-action. Assume that $L$ is an embedded exact Lagrangian filling of $\la$ fixed by the action of $\varphi$ and equipped with a maximal linearly independent collection of $\L$-compressing cycles. The moduli stacks $\FM(\Sigma_\varphi(\la, T))$ and $\SM_1(\Sigma_\varphi(\la, T))$ form a cluster ensemble with every cluster chart induced by an embedded exact Lagrangian filling of $\Sigma_\varphi(\la)$. 
\end{theorem}

\begin{proof}
     Let $\la(\beta)$ be a Legendrian link such that the ring of regular functions $\C[\FM(\la(\beta), T)]$ is a globally foldable cluster algebra  and assume that $L$ is a filling of $\la$ fixed by the action of $\varphi$ with a maximal linearly independent collection of $\L$-compressing cycles $\gamma_1,\dots \gamma_{n}\subseteq H_1(L)$. 
     From $L\times_\varphi S^1$, we obtain a weighted quiver with vertices given by the solid mutation configurations identified by orbits of $\gamma_i$ as in Corollary~\ref{cor: SMCs} and weighted edges given by the formula from Lemma~\ref{lem: poincare=folding}. From the perspective of Subsection~\ref{sub: lattices}, the choice of homology basis given in Subsection~\ref{sub: folding} gives us a lattice structure on $N=H_2(L\times_\varphi S^1, \tau)$ and $M^\circ=H_2(L\times_\varphi S^1\backslash \tau, \la\times_\varphi S^1\backslash \tau)$. 
     The corresponding microlocal monodromies and merodromies about $T_\varphi(\gamma_i)$ and $T_\varphi(\gamma_i^\vee)$, as defined in Definitions~\ref{def: twist_monodromy} and \ref{def: twist_merodromy} together form the fixed data required to define a skew-symmetrizable cluster ensemble structure on the pair $\FM(\Sigma_\varphi(\la, T))$ and $\SM_1(\Sigma_\varphi(\la, T))$, where we rely on the globally foldable condition to ensure that $\C[\FM(\Sigma_\varphi(\la, T))]$ is a cluster algebra defined by this fixed data.

To show that every cluster chart of $\FM(\Sigma_\varphi(\la, T))$ is induced by a filling of $\Sigma_\varphi(\la, T)$, we note that the globally foldable condition implies that mutation commutes with folding. Therefore, we can apply Lemmas~\ref{lem: G-mutation} and \ref{lem: solid mutation} to realize the cluster mutations of the folded cluster algebra as the Lagrangian surgery defined in Lemma~\ref{lem: lag_surgery}, thereby producing a filling of $\Sigma_\varphi(\la, T)$ for every cluster seed.
 
\end{proof}

\subsection{Applications}\label{sub: twist-spun enumeration}

We now discuss some applications of the discussion above to constructing exact Lagrangian fillings and obstructing fillability of various families of twist-spuns. 

\subsubsection{Groups actions and fillings of twist spun Legendrians}

The first author previously studied orbits of exact Lagrangian fillings of $\la(2,n)$ under the K\'alm\'an loop $\rho$, described in Subsection~\ref{sub: twist-spun def}. As a corollary of \cite[Theorem 1.2]{Hughes2021b}, we obtain the following result on twist-spun Legendrian tori of $\la(2,n)$. Recall that we denote by $C_n$ the $n$th Catalan number $\frac{1}{n+1}{2n \choose n}$. We restate Theorem~\ref{thm: intro_catalan_fillings} for clarity.
    
\begin{theorem}\label{thm: Catalan_fillings}
There are at least $f(k)$ embedded exact Lagrangian fillings of $\Sigma_{\rho^k}(\la(2, n))$ where 
$$f(k)=\begin{cases}
    C_{\frac{n+2}{3}} & k=\frac{n+2}{3}\in \N\\ 
    C_{\frac{n}{2}} & k=\frac{n+2}{2} \in \N  \\
    C_n & k=0\\
    0 & \text{otherwise}
\end{cases}$$
and $k\in \{0, \dots, n-1\}.$
\end{theorem}

Note that in the case of $\rho^{n+2}$ (or the identity, more generally) we obtain a spherical spun $\Sigma_{id}(\la)) \cong \la\times S^1$. Since the identity fixes any exact Lagrangian filling of $\la$, exact Lagrangian fillings of spherical spuns can be constructed directly from exact Lagrangian fillings of $\la$, as done in \cite{Golovko22, capovillasearle2023newton}. 

\begin{proof}[Proof of Theorem~\ref{thm: Catalan_fillings}]
    The cases of $k=\frac{n+2}{2}$ and $k=0$ are immediate corollaries of Theorem~\ref{thm: ensembles}. The case of $k=\frac{n+2}{3}$ is slightly more involved, as the cluster algebra $\C[\FM(\la(2, n))]$ is not globally foldable with respect to the action of $\rho^3$. In this case, folding produces a \emph{generalized} cluster algebra \cite[Conjecture 9.1 and Example 9.3]{fraser2018braid} on $Gr^\circ (2, n+2)^{\rho^{\frac{n+2}{3}}}\cong \C[\Sigma_{\rho^{\frac{n+2}{3}}}(\la(2, n)]$ with precisely $C_{\frac{n+2}{3}}$ seeds indexed by rotationally symmetric triangulations of the $(n+2)$-gon. Following the construction of symmetric weaves dual to triangulations in \cite[Theorem 1.2]{Hughes2021b}, we can apply Proposition~\ref{prop: twist-spun fillings} to produce the given number of fillings, which we can then distinguish as belonging to distinct seeds in the generalized cluster algebra $\C[\Sigma_{\rho^{\frac{n+2}{3}}}(\la(2, n)]$.
\end{proof}

\begin{remark}
Outside of the case of $\Sigma_{\rho^{\frac{n+2}{3}}}(\la(2, n))$, results of Fraser in  \cite{fraser2020cyclic} hint at the possibility of obtaining a generalized cluster algebra structure in cases where $\la$ admits no $\varphi$-invariant exact Lagrangian fillings. Initial computations suggest that singular exact Lagrangian fillings correspond to generalized cluster seeds in these cases. We believe this to be the first low-dimensional contact-geometric setting in which such generalized cluster structures appear and plan to explore this further in future work.
\end{remark}

As an additional corollary to Theorem~\ref{thm: ensembles}, we obtain infinitely many fillings for several families of twist-spun tori. With a bit more care, we can even leverage the relationship between Legendrian loops and their relationship to cluster modular groups to describe specific infinite group actions on the exact Lagrangian fillings of certain twist spun Legendrians. The cluster modular group of a cluster algebra $\mathcal{A}$ can be defined as the set of all algebra automorphisms of $\mathcal{A}$ that preserve cluster seeds and commute with mutation. This group is difficult to compute for an arbitrary cluster algebra, but is known for classes of cluster algebras corresponding to finite and affine Dynkin types and is conjecturally computed for cluster algebras corresponding to coordinate rings of Grassmannians \cite[Conjecture 8.2]{fraser2018braid}.

We denote by $\langle \varphi\rangle^H$ the normal closure of the group generated by an element $\varphi$ in a group $H$. In \cite[Section 6.4]{Hughes2024}, the first author gave descriptions of Legendrian loops that generate $H/\langle \varphi\rangle^H$, where $H$ is a cluster modular group for Legendrian links corresponding to cluster algebras of finite, affine, and extended affine type and $H/\langle \varphi\rangle^H$ is thought of as the cluster modular group for a corresponding folded cluster algebra. These families are summarized in Table~\ref{tab: folded}, and from this, we obtain the following result on fillings of twist-spun Legendrians.

\begin{theorem}\label{thm: infinitely many fillings}
  
Let $\la$ be a $(-1)$-closure of a braid appearing in the leftmost column of Table~\ref{tab: folded} and let $\varphi$ and $H$ be the corresponding loop and group, respectively. Then $H/\langle\varphi\rangle^H$ acts faithfully on the set of exact Lagrangian fillings of $\Sigma_\varphi(\la)$.
\end{theorem}

\begin{table}[]
\begin{tabular}{l|l|l|l}
 Braid word $\beta$  & $H$ & $\varphi$ & $H/\langle \varphi \rangle^H$  \\
\hline

 $\sigma_1^{2n+2}$ &  $\Z_{2n+2}$ &  $\rho^{n+1}$ & $\Z_{n+1}$ \\

  $(\sigma_2\sigma_1)^6$ & $\Z_4\times \S_3$ & $\rho^2$ & $\Z_4$
\\
 $\sigma_1^{n-2}\sigma_2\sigma_1\sigma_1\sigma_2(\sigma_1\sigma_2)^3$, $n>4$ & $\Z_n\times \Z_2$ & -- & --\\

 $(\sigma_2\sigma_1)^7$ & $\Z_{14}$ & -- & --\\ 
  $(\sigma_2\sigma_1)^7\sigma_1$ & $\Z_{10}$ & -- & -- \\ 
   $(\sigma_2\sigma_1)^8$ & $\Z_{16}$ & -- & --\\

 $(\sigma_2\sigma_1^4)^3$ & $S_3\times\Z$  & $\delta^5$ & $\Z$ \\
 
 $(\sigma_1\sigma_2\sigma_3)^8$ &  $\pi_1(Conf(S^2, 4))\rtimes \Z_2$   & $\rho^4$ & $\text{Mod}(S^2, 4)\rtimes \Z_2$\\  

$(\sigma_1\sigma_2)^9$ &  $\PSL_2(\Z)\rtimes \Z_6$  & $\rho^3$ & $\PSL_2(\Z)\rtimes \Z_2$ \\

\end{tabular}

\caption{Table of known cluster modular groups of $\FM(\la(\beta))$ and $\FM(\Sigma_\varphi(\la(\beta))$. The first column consists of positive braid words for known Legendrian links corresponding to cluster algebras of finite, affine, and extended affine type. The second column gives the cluster modular group of $\SM(\la(\beta)$, the third column gives a Legendrian loop inducing a globally foldable group action, if known, and the final column lists the corresponding cluster modular groups of the folded cluster algebras obtained.}
\label{tab: folded}
\end{table}

\begin{proof} 
For any triple $\la, \varphi, H$ appearing in the table, Theorems 1.3 and 1.4 of \cite{Hughes2024} prove that $H$ acts faithfully on fillings of $\la$. When $\varphi= id$, the statement follows immediately. In the cases where $\varphi\neq id$, any Legendrian loop inducing a generator of $H$ commutes with $\varphi$, as argued in \cite[Section 6.4]{Hughes2024}. As a result, we have that $\psi$ acts on fillings of $\Sigma_\varphi(\la)$, as we have
$$\varphi \circ \psi(L)=\psi \circ\varphi (L)=\psi(L)$$ for any $\psi$ with $\psi_*\in H$.
\end{proof}

\begin{proof}[Proof of Theorem~\ref{thm: intro_infinite}]
There are faithful $\PSL(2, \Z)$ and $\text{Mod}(\Sigma_{0,4})$ actions on the set of exact Lagrangian fillings of $\Sigma_{\rho^3}(\la(3,6))$ and $\Sigma_{\rho^4}(\la(4,4))$, respectively.
The Legendrian loops inducing the group actions are constructed in \cite[Section 2]{CasalsGao} and give the first construction of an infinite exact Lagrangian fillings. Figure~\ref{fig: symmetric fillings} gives  initial fillings for these Legendrian loops to act on constructed from symmetric plabic graphs via the $T$-shift process of \cite[Section 3.1]{CLSBW}. This proves Theorem~\ref{thm: intro_infinite}.
\end{proof}

\begin{figure}[h!]{ \includegraphics[width=.8\textwidth]{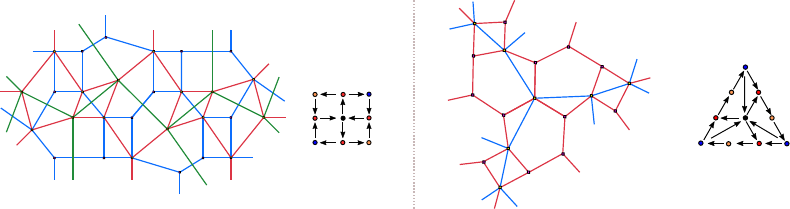}}\caption{Symmetric fillings of $\la(4,4)$ (left) and $\la(3, 6)$ (right)  corresponding to fillings of $\Sigma_{\rho^4}(\la(4, 4))$ and $\Sigma_{\rho^3}(\la(3, 6))$, respectively.}
			\label{fig: symmetric fillings}\end{figure}

Given the existence of cluster structures on sheaf moduli of twist-spuns, we extend the conjectural ADE classification of exact Lagrangian fillings of Legendrian rainbow closures of positive braids due to Casals \cite[Conjectures 5.1 and 5.4]{CasalsLagSkel} to our context. In particular, for such $\la$ we conjecture that there are exactly as many embedded exact Lagrangian fillings of the twist spun $\Sigma_\varphi(\la)$ as there are $\varphi$-invariant fillings of $\la$, as there are seeds in the folded cluster algebra. 


\subsubsection{Filling obstructions for twist spun Legendrians}\label{sub: nonfillable_spuns}

In this subsection, we exhibit infinitely many twist spuns of Legendrian positive torus links that do not admit any exact Lagrangian fillings. Recall that we denote the K\'alm\'an loop and its induced action on $\SM_1(\la(k,n))$ by $\rho$.

Our strategy for obstructing fillings of $\Sigma_{\rho^{\ell}}(\la(k,n))$ for specific values of $\ell, k,$ and $n$ is to show that $\SM_1(\Sigma_{\rho^\ell}(\la(k, n)))$ does not have any rational points. By \cite{JinTreumann17}, an embedded algebraic torus induced by an exact Lagrangian filling of $\Sigma_{\rho^\ell}(\la(k, n))$ necessarily implies the existence of rational points in $\SM_1(\Sigma_{\rho^\ell}(\la(k, n)))$, so a lack of rational points gives an obstruction for fillability. We thank David Treumann for suggesting this strategy and for describing the case of $\Sigma_\rho(\la(2, 3))$ to us.

For positive torus links $\la(k, n-k)$, the sheaf moduli $\SM_1(\la(k, n-k))$ is closely related to the Grassmannian $Gr(k, n)$. In order to determine when $\SM_1(\Sigma_{\rho^\ell}(\la(k, n-k)))$ admits rational points, we require a characterization of fixed points of the Grassmannian under the action of the cyclic shift, an automorphism connected to the K\'alm\'an loop. Given a point $V$ represented by a $k\times n$ matrix with columns $v_1, \dots, v_n,$ the cyclic shift acts by $v_i\mapsto v_{i-1}$ for $i\geq 2$ and $v_1\mapsto (-1)^{k-1}v_n$. The fixed points of this action are studied in \cite{Karp}.

\begin{theorem}{\cite[Theorem 1.1]{Karp}} \label{thm: Karp}
There are ${n \choose k}$ fixed points of $Gr(k , n)$, $2\leq k\leq n$, under the cyclic shift action represented by the span of vectors $\{(1, \zeta_j, \dots, \zeta_j^{n-1}) \mid 1 \leq j \leq k\}$ where $\zeta_1, \dots, \zeta_k$ are any $k$ distinct $n$th roots of $(-1)^{k-1}$. 
\end{theorem}

\begin{ex}
    Let $k=2$, $n=5$, and choose $\zeta_1=\zeta_2^{-1}=e^{\pi i/5}$. Then $$V=\begin{bmatrix}
        1 & \zeta_1 &\zeta_1^2 & \zeta_1^3& \zeta_1^4\\
        1 & \zeta_1^{-1} &\zeta_1^{-2} & \zeta_1^{-3}& \zeta_1^{-4}\\
    \end{bmatrix}$$
    represents a fixed point of the cyclic shift acting on $Gr(2, 5)$. In particular, $V$ is the unique positive real fixed point of this action, as can be verified by checking that the ratio $\Delta_{ij}/\Delta_{12}$ is a positive real number for each minor $\Delta_{ij}$ of $V$.
\end{ex}

In order to produce fillability obstructions using the fixed points identified in Theorem \ref{thm: Karp}, we restrict to torus links $\la(2, n-2)$ and $\la(3, n-3)$, with $n\geq 6$ satisfying the following restrictions.

\begin{align}\label{eqn: divisibility1}
&\text{For } k\in 2\Z, \text{ we require that } n\neq 2m, 3m \text{ for } m\in 2\Z+1.\\
\label{eqn: divisibility2}
&\text{For } k\in 2\Z+1, \text{ we require that } n \text{ is not divisible by } 2, 4, \text{ or } 6.
\end{align}

We now have all of the necessary ingredients to state our fillability obstruction. Theorem \ref{thm: Intro_nonfillable_spuns} is an immediate consequence of the following.

\begin{theorem}\label{thm: nonfillable_spuns}
Let $k=2, 3$ and assume that $\ell$ and $n$ are relatively prime positive integers with $n$ satisfying Equations~\ref{eqn: divisibility1} and \ref{eqn: divisibility2}. The twist-spuns $\Sigma_{\rho^\ell}(\la(2, n-2))$ and $\Sigma_{\rho^\ell}(\la(3, n-3))$ do not admit any exact Lagrangian fillings.
\end{theorem}

\begin{proof}
By Proposition~\ref{prop: twist spun sheaf moduli}, we have $\SM_1(\Sigma_\rho(\la(k, n-k)))\cong\SM_1(\la(k, n-k))^\rho\times \C^\ast$. Moreover, we can identify $\SM_1(\la(k, n-k))$ with the top-dimensional positroid cell $\hat{Gr}(k, n)$ of the Grassmannian $Gr(k, n)$. Under this identification, the action of the K\'alm\'an loop $\rho$ on $\SM_1(\la(k,n-k))$ intertwines with that of the cyclic shift $\rho$ on $\hat{Gr}(k, n)$; see e.g. Lemma 5.8 of \cite{Hughes2024}. 
Therefore, we can obstruct fillability of our family of twist-spuns by showing that none of Karp's fixed points are rational when Equations~\ref{eqn: divisibility1} and \ref{eqn: divisibility2} are satisfied. 

Let $V$ denote one of the fixed points identified in Theorem~\ref{thm: Karp}. The Pl\"ucker coordinates of $V$ descend to affine coordinates in the top-dimensional positroid cell $\hat{Gr}(k, n)$, and we can see that to obtain real fixed points of $\SM_1(\Sigma_{\rho^\ell}(\la(2, n-2))$, we must choose $\zeta_1, \zeta_2$, $n$th roots of $-1$, so that $\zeta_{2}=\zeta_{1}^{-1}$; for real fixed points of $\SM_1(\Sigma_{\rho^\ell}(\la(3, n-3))$, we choose $n$th roots of unity $\zeta_1,\zeta_2, \zeta_3$ so that $\zeta_{2}=\zeta_{1}^{-1}$, and $\zeta_3=1$. To verify this for $k=3$, we note that the ratio of minors $\Delta_{124}/\Delta_{123}$ is given by  $\zeta_1+\zeta_2 + \zeta_3$, while $\Delta_{234}/\Delta_{123}=\zeta_1\zeta_2\zeta_3$.\footnote{In general, the ratio $\Delta_I/\Delta_{12\dots k}$ is a Schur polynomial evaluated at roots of $(-1)^{k-1}$. As far as we are aware there is no general study of when this particular choice of specialization yields a rational number.} Together, requiring that these two ratios are real implies the above condition up to re-indexing. The case of $k=2$ is similar. 

To show that we cannot obtain a rational point from the Pl\"ucker functions of one of these fixed points, we must show that the ratio $\frac{\Delta_I}{\Delta_{1\dots k }}$ is irrational for some Pl\"ucker function $\Delta_I$.

For $k=2,$ we have $V=\begin{bmatrix}
    1 & \zeta_1 & \dots & \zeta_1^{n-1}\\
    1 & \zeta_1^{-1} & \dots & \zeta_1^{-(n-1)}
\end{bmatrix}$ and $\Delta_{12}=i\sin(\pi k/n)$, $1 \leq |k| \leq n-1, k\in 2\Z+1$. $\Delta_{13}=\zeta_1^2-\zeta_1^{-2}$ and $\Delta_{13}/\Delta_{12}=(\zeta_1+\zeta_1^{-1})=2\cos(\pi k/n)$. When $n$ satisfies Equations~\ref{eqn: divisibility1} and \ref{eqn: divisibility2}, we have that $2\cos(\pi k/n)\not\in \Q$ by Niven's theorem.

For $k=3$, we have that $\Delta_{124}/\Delta_{123}=\zeta+\zeta^{-1}+1$ for $\zeta=e^{2\pi i j/n}, 1\leq j \leq n-1$. This simplifies to $2\cos(2\pi j/n)+1$, which again is irrational when $n$ satisfies Equations~\ref{eqn: divisibility1} and \ref{eqn: divisibility2}.

\end{proof}

\begin{remark}\label{rmk: obstruction_hard}
    In line with Conjecture \ref{conj: number_fillings}, we expect that the conclusions of Theorem~\ref{thm: nonfillable_spuns} hold more generally, both for $k\geq 4$, and for cases where $n$ does not satisfy Equations~\ref{eqn: divisibility1} and \ref{eqn: divisibility2}. Our approach to the proof of Theorem~\ref{thm: nonfillable_spuns} can be extended to $k\geq 4$, assuming the restrictions on $n$ from Equations~\ref{eqn: divisibility1} and \ref{eqn: divisibility2}, by showing that the sum or product of the roots of $(-1)^{k-1}$ defining $V$ are irrational, as these expressions correspond to the ratios of Pl\"ucker functions $\Delta_{1 2 \dots (k-1) (k+1)}/\Delta_{12 \dots k}$ and $\Delta_{2 3 \dots (k+1)}/\Delta_{1 2 \dots k}$, respectively.
\end{remark}

\subsection*{Conflicts of interest:} none.

\subsection*{Financial support:} 
J.H. was partially supported by NSF grant DMS-1942363. A.R. was partially supported by NSF grants DMS-2203312 and DMS-1907654, and an AMS-Simons Travel Grant.

	\bibliographystyle{alpha}
	\bibliography{main}

\end{document}